\documentclass[11pt]{article}

\usepackage[sort]{cite}       
\usepackage{amssymb}          
\usepackage{latexsym}         
\usepackage{amsmath}          
\usepackage{amsxtra}          
\usepackage{amscd}            
\usepackage[utf8]{inputenc}   
\usepackage{eucal}            
\usepackage{bbm}              
\usepackage{theorem}          
\usepackage{mathrsfs}         
\usepackage{paralist}         
\usepackage{stmaryrd}         
\usepackage{mathtools}        

\title{Fréchet algebraic deformation quantization of the Poincaré
  disk}

\author{\textbf{Svea Beiser}\thanks{E-mail:
    Svea.Beiser@physik.uni-freiburg.de}\addtocounter{footnote}{2}
  \\[0.1cm]
  Fakult{\"a}t f{\"u}r Mathematik und Physik\\
  Albert-Ludwigs-Universit{\"a}t Freiburg\\
  Physikalisches Institut\\
  Hermann Herder Stra{\ss}e 3\\
  D 79104 Freiburg\\
  Germany\\[1cm]
  \textbf{Stefan Waldmann}\thanks{E-mail:
    Stefan.Waldmann@wis.kuleuven.be}
  \\[0.1cm]
  Katholieke Universiteit Leuven \\
  Department of Mathematics \\
  Celestijnenlaan 200B bus 2400 \\
  BE-3001 Heverlee \\
  Belgium
}

\date{August 2011}

\renewcommand{\mathbb}[1]{\mathbbm{#1}}


\newcommand{\refitem}[1] {~\textit{\ref{#1}.)}}

\numberwithin{equation}{section}

\allowdisplaybreaks

\newcommand{\I}          {\mathrm{i}}
\newcommand{\E}          {\mathrm{e}}
\newcommand{\D}          {\operatorname{\mathrm{d}}}
\newcommand{\cc}[1]      {\overline{{#1}}}

\newcommand{\cl}         {\mathrm{cl}}
\newcommand{\At}[1]      {\Big|_{#1}}
\newcommand{\argument}   {\,\cdot\,}
\newcommand{\pr}         {\mathrm{pr}}
\newcommand{\tr}         {\operatorname{\mathsf{tr}}}
\newcommand{\diag}       {\operatorname{\mathrm{diag}}}
\newcommand{\SP}[1]      {\left\langle{#1}\right\rangle}
\newcommand{\spann}      {\operatorname{\mathrm{span}}}
\newcommand{\Lie}        {\operatorname{\mathscr{L}\!}}
\newcommand{\norm}[1]    {\left\|{#1}\right\|}
\newcommand{\supnorm}[1] {\left\|{#1}\right\|_\infty}
\newcommand{\tensor}[1][{}]{\mathbin{\otimes_{\scriptscriptstyle{#1}}}}

\theoremheaderfont{\normalfont\bfseries}
\theorembodyfont{\itshape}
\newtheorem{lemma}{Lemma}[section]
\newtheorem{proposition}[lemma]{Proposition}
\newtheorem{theorem}[lemma]{Theorem}
\newtheorem{corollary}[lemma]{Corollary}
\newtheorem{definition}[lemma]{Definition}

\theorembodyfont{\rmfamily}

\newtheorem{remark}[lemma]{Remark}

\makeatletter
\newcommand\qedsymbol{\hbox{$\boxempty$}}
\newcommand\qed{\relax\ifmmode\boxempty\else
  {\unskip\nobreak\hfil\penalty50\hskip1em\null\nobreak\hfil\qedsymbol
  \parfillskip=\z@\finalhyphendemerits=0\endgraf}\fi}
\makeatother

\newenvironment{proof}[1][{}]{\par\noindent Proof{#1}. }{\qed\medskip}

\newenvironment{lemmalist}{\begin{compactenum}[\itshape i.)]}{\end{compactenum}}
\newenvironment{theoremlist}{\begin{compactenum}[\itshape i.)]}{\end{compactenum}}
\newenvironment{propositionlist}{\begin{compactenum}[\itshape i.)]}{\end{compactenum}}

\newcommand{\complete}[1]    {\widehat{#1}}
\newcommand{\length}         {\mathop{\mathrm{L}}}
\newcommand{\starwick}       {\mathbin{\star_{\scriptscriptstyle\mathrm{Wick}}}}

\newcommand{\tildewick}      {\mathbin{\widetilde{\star}_{\scriptscriptstyle\mathrm{Wick}}}}
\newcommand{\stardisk}       {\mathbin{\star_{\scriptscriptstyle\mathbb{D}_n}}}
\newcommand{\lexp}           {\operatorname{\mathrm{Lexp}}}
\newcommand{\Pochhammer}[1]  {\left(#1\right)}
\newcommand{\basis}[1]       {\mathsf{#1}}
\newcommand{\HolAntiHol}     {\mathcal{O}\cc{\mathcal{O}}}
\newcommand{\metric}         {\mathsf{g}}

\begin{document}

\maketitle

\begin{abstract}
    Starting from formal deformation quantization we use an explicit
    formula for a star product on the Poincaré disk $\mathbb{D}_n$ to
    introduce a Fréchet topology making the star product
    continuous. To this end a general construction of locally convex
    topologies on algebras with countable vector space basis is
    introduced and applied. Several examples of independent interest
    are investigated as e.g.\ group algebras over finitely generated
    groups and infinite matrices. In the case of the star product on
    $\mathbb{D}_n$ the resulting Fréchet algebra is shown to have many
    nice features: it is a strongly nuclear Köthe space, the symmetry
    group $\mathrm{SU}(1, n)$ acts smoothly by continuous
    automorphisms with an inner infinitesimal action, and evaluation
    functionals at all points of $\mathbb{D}_n$ are continuous
    positive functionals.
\end{abstract}

%
%



%
%

\newpage

%
%

\tableofcontents 

%
%

\section{Introduction}
\label{sec:Introduction}

Deformation quantization as introduced in \cite{bayen.et.al:1978a}
aims at the construction of a quantum observable algebra out of the
observable algebra of a classical mechanical system, usually modeled
by the smooth functions on a Poisson manifold. The idea is to
implement the necessary noncommutativity by means of a deformation of
the commutative product in direction of the Poisson bracket with
deformation parameter $\hbar$ resulting in a \emph{star product}. Here
several different settings are possible: in formal deformation
quantization one follows the algebraic deformation program of
Gerstenhaber \cite{gerstenhaber:1964a} and uses formal power series in
$\hbar$, ignoring questions of convergence completely. In this setting
very strong existence and classification results are available,
cumulating in Kontsevich's formality theorem which guarantees the
existence and provides the full classification of formal star products
on Poisson manifolds in general \cite{kontsevich:2003a}. A gentle
introduction with a large bibliography can be found in the textbook
\cite{waldmann:2007a}.

For many reasons a formal deformation quantization is not believed to
be the end of the story: for applications in quantum theory $\hbar$
has to be treated as a positive number. But also applications of
deformation quantization in noncommutative geometry
\cite{connes:1994a}, where it provides important examples, require a
more analytic framework. Several notions of ``convergent''
deformations are available. Rieffel's notion of strict deformation
quantization \cite{rieffel:1993a, rieffel:1989a}, see also Landsman's
monograph \cite{landsman:1998a}, are located in the realm of
$C^*$-algebras. Here one tries to construct continuous fields of
$C^*$-algebras deforming a given commutative $C^*$-algebra sitting at
a particular value of the underlying space of parameters. Many
examples are known and studied in great detail. Unfortunately, a
general existence and classification as in the formal case is not yet
available and perhaps not even possible: here things usually start to
depend strongly on the example and many completely different
approaches and constructions compete.

While a $C^*$-algebraic formulation has clear motivations from quantum
physics and noncommutative geometry, it also poses several technical
difficulties. Ultimately, they originate in the fact that the passage
from smooth (or even algebraic) differential geometry to the
continuous world of $C^*$-algebras is quite a long way. Thus it might
be reasonable to find intermediate types of deformations where on one
hand the formal character of formal star products is already overcome
but one still has closer contact to the smooth structures present in
differential geometry.

In this work we want to advocate a Fréchet algebraic deformation
quantization where the formal series in $\hbar$ are investigated by
means of a suitable Fréchet topology on a certain class of functions.

To achieve this we first provide a general construction of a locally
convex topology for an algebra with a countable vector space basis,
the completion of it will then be a Fréchet algebra. Whether our
general construction works will depend on some technical details of
the underlying algebra which seem hard to pin down in
general. However, we discuss and illustrate our program in several
examples like polynomial algebras, infinite matrices and group
algebras over finitely generated groups. These examples seem to be
interesting enough to pursue a further investigation. Nevertheless, we
will turn to deformation quantization after a short glance at them.

In principle, there are other constructions of Fréchet topologies for
algebras with countable vector space basis: the construction in
\cite[Prop.~2.1]{cuntz:2005a} can be interpreted as such, see
Remark~\ref{remark:CuntzConstruction}. However, the resulting highly
recursive construction of seminorms seems to be very difficult to
handle. Nevertheless, this suggest that one should not expect any
reasonable uniqueness theorems for such Fréchet topologies unless one
specifies further properties.

In deformation quantization the first example to look at is of course
the Weyl product for the flat phase space $\mathbb{R}^{2n}$. We show
that our general construction coincides with a previous construction
\cite{beiser.roemer.waldmann:2007a} and yields a very explicit
completion of the Weyl algebra to a nuclear Fréchet algebra. The
underlying Fréchet space is a certain Köthe space and hence enjoys
many nice properties.

The main task of this work is devoted to a more nontrivial example
from deformation quantization: the Poincaré disk and its higher
dimensional cousins $\mathbb{D}_n$. This is a first example of a
symplectic manifold with a \emph{curved} Kähler structure: it is a
Hermitian symmetric space of noncompact type with symmetry group
$\mathrm{SU}(1, n)$.  Moreover, the Poincaré disk is the covering
space of the Riemann surfaces of higher genus. Thus an invariant
quantization should ultimately also induce quantizations of these
Riemann surfaces. Finally, the Poincaré disk is a first nontrivial
example of a reduced phase space, allowing to test ideas of the
quantization of phase space reduction also in a convergent framework.

The principal result is that our general construction applies to all
stages of the phase space reduction as introduced in
\cite{bordemann.brischle.emmrich.waldmann:1996a,
  bordemann.brischle.emmrich.waldmann:1996b}: first we start with a
star product $\tildewick$ obtained from a particular equivalence
transformation of the Wick star product on $\mathbb{C}^{n+1}$. The
$\mathrm{U}(1)$-invariant functions are a subalgebra in which we find
a subalgebra with countable vector space basis, essentially given by
polynomials. We obtain a completion to a nuclear Fréchet algebra which
consists of real-analytic functions of a particular type. The former
vector space basis becomes an absolute Schauder basis after
completion.  The passage to the Poincaré disk is a quotient by a
two-sided closed ideal. The resulting nuclear Fréchet algebra
$\complete{\mathcal{A}}_\hbar(\mathbb{D}_n)$ of functions on the disk
can again be characterized very explicitly as a certain Köthe space
with absolute Schauder basis consisting of representative functions
with respect to the $\mathrm{SU}(1, n)$-action on the disk.

After the construction of $\complete{\mathcal{A}}_\hbar(\mathbb{D}_n)$
we determine several further properties: first we show that the
symmetry group $\mathrm{SU}(1, n)$ acts by continuous
automorphisms. The action turns out to be smooth and the corresponding
Lie algebra action is inner via the classical momentum map. Second,
the dependence on the parameter $\hbar$ is shown to be holomorphic for
$\hbar \in \mathbb{C} \setminus \{-\frac{1}{2}, -\frac{1}{4}, \ldots,
0\}$. Third, we consider real $\hbar$ in which case the pointwise
complex conjugation is a continuous $^*$-involution of
$\complete{\mathcal{A}}_\hbar(\mathbb{D}_n)$ for which the evaluation
functionals at points of $\mathbb{D}_n$ are (continuous) positive
functionals. The corresponding GNS construction is determined
explicitly.

This last result will hopefully provide the bridge to compare our
construction to other deformation quantizations of the Poincaré
disk. A well-studied approach uses a Berezin-Toeplitz like
quantization based on coherent states. The star product is then
obtained as asymptotic expansion of a symbol calculus using the
Toeplitz operators on the (anti-)holomorphic sections of higher and
higher tensor powers of certain holomorphic line bundles. Here one
should consult in particular \cite{cahen.gutt.rawnsley:1995a,
  cahen.gutt.rawnsley:1994a, cahen.gutt.rawnsley:1993a,
  cahen.gutt.rawnsley:1990a} and
\cite{bordemann.meinrenken.schlichenmaier:1991a} for a general
background on deformation quantization of Kähler manifolds as well as
\cite{borthwick.lesniewski.upmeier:1993a}. Here our GNS representation
should be compared to the (pseudo-differential) operator approaches.
Alternatively, the algebra
$\complete{\mathcal{A}}_\hbar(\mathbb{D}_n)$ should be related
directly to the algebras constructed and studied in
\cite{bieliavsky.detournay.spindel:2009a, bieliavsky:2002a}. There,
explicit integral formulas for a star product on the Poincaré disk
where given which allow for a $C^*$-algebraic quantization. The
precise relations will be subject to a future project.

The paper is organized as follows: in
Section~\ref{sec:GeneralConstructionFirstVersion} we establish the
general construction in two versions. First we construct a countable
system of seminorms for an algebra with countable vector space basis
in terms of the structure constants.  The difficulty is that the
``seminorms'' may diverge. To avoid this a closer look at the example
will be necessary.  Then we can describe the Fréchet algebra
completion and discuss some general properties of it in
Theorem~\ref{theorem:AniceCompleted}. The second version will be a
finer topology making a given linear functional continuous as
well. Again we have to show by hand that the recursive definition will
produce finite seminorms. General properties of this second version
are then discussed in Theorem~\ref{theorem:Omega}.

In Section~\ref{sec:FirstExamplesPolynomials} we discuss several basic
examples to illustrate our general construction: first we consider
polynomials as well as Laurent polynomials and study some of the
properties of the resulting Fréchet algebras. As a second and now
noncommutative example we consider infinite matrices with finitely
many nonzero entries. Here we discuss the dependence of our general
construction on the chosen vector space basis. The third example is
the group algebra of an infinite but finitely generated group. The
canonical basis given by the group elements has to be rescaled by an
appropriate prefactor in order to make the general construction work:
here the condition of finitely many generators becomes crucial. We
take the factorial of the word length of a group element as
rescaling. Clearly, this example would be interesting to investigate
further. A last example is the Wick star product on flat
$\mathbb{C}^n$. We show that the general construction reproduces an
earlier approach \cite{beiser.roemer.waldmann:2007a} which was
actually the main motivation for the present work. Moreover, we
determine the resulting Fréchet algebra explicitly as a particular
Köthe space. Thereby we show in
Theorem~\ref{theorem:WickAlgebraSubFac} that the Weyl algebra has a
completion which is nuclear and has the monomials as absolute Schauder
basis.

Section~\ref{sec:StarProductOnThePoincareDisc} contains the main
example: first we recall some basic properties of the Poincaré disk to
fix our notation. Next, the construction of the (formal) star product
according to \cite{bordemann.brischle.emmrich.waldmann:1996b,
  bordemann.brischle.emmrich.waldmann:1996a} is recalled. We consider
$\mathrm{U}(1)$-invariant representative functions on
$\mathbb{C}^{n+1}$ with respect to the linear $\mathrm{SU}(1,
n)$-action. A particular vector space basis is chosen and we compute
the structure constants explicitly. The first version of our general
construction is shown to produce the Cartesian product topology on the
span of the basis vectors. Hence the completion is not interesting as
we cannot interpret the elements of it as functions anymore. Thus we
evoke the second version to make the evaluation functionals at all
points continuous as well. In a last step we have to divide by the
(classical) vanishing ideal of a level set of the Hamiltonian of the
$\mathrm{U}(1)$-action in order to get functions on the Poincaré
disk. As we can show the vanishing ideal to be a closed ideal also
with respect to the star product upstairs, we obtain a Fréchet algebra
$\complete{\mathcal{A}}_\hbar(\mathbb{D}_n)$ of functions on
$\mathbb{D}_n$. The main result in Theorem~\ref{theorem:DiscAlgebra}
determines the algebra $\complete{\mathcal{A}}_\hbar(\mathbb{D}_n)$ as
a certain Köthe space with absolute Schauder basis together with some
further properties.

Section~\ref{sec:FurtherProperties} contains some further results on
the algebra $\complete{\mathcal{A}}_\hbar(\mathbb{D}_n)$. We show that
the canonical action of $\mathrm{SU}(1, n)$ on $\mathbb{D}_n$ induces
a smooth action by automorphisms of
$\complete{\mathcal{A}}_\hbar(\mathbb{D}_n)$ with inner Lie algebra
action, see Theorem~\ref{theorem:SUEinsnSymmetrie}. Moreover, for the
allowed values of $\hbar$ we show that
$\complete{\mathcal{A}}_\hbar(\mathbb{D}_n)$ is a holomorphic
deformation in the sense of \cite{pflaum.schottenloher:1998a} in
Theorem~\ref{theorem:HolomorphicProduct}. Finally, we show the
positivity of the evaluation functionals at all points of
$\mathbb{D}_n$ and start investigating the corresponding GNS
representation.

Finally, Appendix~\ref{sec:SubfactorialGrowth} contains some
well-known facts about Köthe spaces which we will need at several
places.

%
%

\medskip

\noindent
\textbf{Acknowledgements:} It is a pleasure to thank Pierre Bieliavsky
for valuable discussion at different stages of this
project. S.B. thanks the Graduiertenkolleg ``Physik an
Hadronenbeschleunigern'' and the Schlieben-Lange-Programm for
financial support.

%
%

\section{The general construction}
\label{sec:GeneralConstructionFirstVersion}

The purpose of this section is to establish a general construction of
a locally convex topology consisting of countably many seminorms on an
algebra $\mathcal{A}$ with a countable vector space basis such that
the product becomes continuous. In general, the construction will
require some additional estimates in order to succeed which we will
discuss in detail.

%
%

\subsection{First version}
\label{subsec:FirstVersion}

To fix our notation we consider a complex algebra $\mathcal{A}$ with a
vector space basis $\{\basis{e}_\alpha\}_{\alpha \in I}$ where $I$ is a
countably infinite index set. The algebra multiplication will be
denoted by $\star$ and gives rise to structure constants
$C^\gamma_{\alpha\beta}$ defined by
\begin{equation}
    \label{eq:StructureConstants}
    \basis{e}_\alpha \star \basis{e}_\beta
    =
    \sum_{\gamma \in I} C^\gamma_{\alpha\beta} \basis{e}_\gamma.
\end{equation}
Note that for given $\alpha$ and $\beta$ we have only \emph{finitely}
many $\gamma$ with $C^\gamma_{\alpha\beta} \ne 0$. An element $a \in
\mathcal{A}$ can thus be written as
\begin{equation}
    \label{eq:aLinearCombination}
    a = \sum_{\gamma \in I} a_\gamma \basis{e}_\gamma
\end{equation}
with only finitely many $a_\gamma$ different from zero. In particular,
the choice of the basis gives the coefficient functionals
\begin{equation}
    \label{eq:CoefficientFuntionals}
    e^\gamma\colon \mathcal{A} \ni a 
    \; \mapsto \;
    a_\gamma \in \mathbb{C}.
\end{equation}
\begin{remark}
    \label{remark:TypeOfAlgebra}%
    For the following we will be interested in \emph{associative}
    algebras over the complex numbers or even in
    $^*$-algebras. However, the construction will work equally well
    over the real numbers and also for other types of algebras like
    Lie algebras. The associativity will not be used in the
    construction at all. The Lie case will clearly be of independent
    interest.
\end{remark}

Based on the structure constants $C^\gamma_{\alpha\beta}$ we can now
state the following (recursive) definition:
\begin{definition}
    \label{definition:TheSeminormStuff}%
    For $m \in \mathbb{N}_0$, $\ell = 0, \ldots, 2^m - 1$, and $\gamma
    \in I$ one defines recursively the maps $h_{m, \ell, \gamma}\colon
    \mathcal{A} \longrightarrow [0, +\infty]$ by setting $h_{0, 0,
      \gamma}(a) = |a_\gamma|$ and
    \begin{equation}
        \label{eq:HRecursion}
        h_{m+1, 2\ell, \gamma}(a)
        =
        \sum_{\alpha \in I}
        \left(h_{m, \ell, \alpha}(a)\right)^2
        \sum_{\beta \in I} \left|C^\gamma_{\alpha\beta}\right|
        \quad
        \textrm{and}
        \quad
        h_{m+1, 2\ell+1, \gamma}(a)
        =
        \sum_{\beta \in I}
        \left(h_{m, \ell, \beta}(a)\right)^2
        \sum_{\alpha \in I} \left|C^\gamma_{\alpha\beta}\right|.
    \end{equation}
    Moreover, we define $\norm{\argument}_{m, \ell, \gamma}\colon
    \mathcal{A} \longrightarrow [0, +\infty]$ by
    \begin{equation}
        \label{eq:Halbnormen}
        \norm{a}_{m, \ell, \gamma}
        = \sqrt[2^m]{h_{m, \ell, \gamma}(a)}.
    \end{equation}
\end{definition}
Before we proceed let us state a few simple remarks. Clearly, the
convergence (divergence) of the series is absolute as they consist of
nonnegative terms only. Moreover, the resulting values depend strongly
on the choice of the basis and, as we shall see later, even the
convergence itself does. This is an unpleasant feature but the
generality of the construction seems to have this price.

If the algebra is commutative, then the index $\ell$ does not play any
role: a simple induction shows that for given $m$ and $\gamma$, the
quantities $h_{m, \ell, \gamma}(a)$ coincide for all $\ell = 0,
\ldots, 2^m-1$ in this case.

Finally, it will be useful to assume that the series
\begin{equation}
    \label{eq:CSeriesConverge}
    C^\gamma_{\alpha, \boldsymbol{\cdot}}
    = \sum_{\beta \in I} \left|C^\gamma_{\alpha\beta}\right|
    \quad
    \textrm{and}
    \quad
    C^\gamma_{\boldsymbol{\cdot}, \beta}
    =
    \sum_{\alpha \in I} \left|C^\gamma_{\alpha\beta}\right|
\end{equation}
are \emph{finite} for all $\alpha, \beta, \gamma \in I$. Otherwise the
corresponding $h_{1, \ell, \gamma}$ will be infinite for $a \ne 0$ and
thus all the following ones as well. Note that this is already a
nontrivial convergence condition for the structure constants. Using
these constants (whether finite or not) we get the following simple
estimate
\begin{equation}
    \label{eq:mKleinermPlusEins}
    \sqrt[2^{m+1}]{C^\gamma_{\alpha, \boldsymbol{\cdot}}}
    \norm{a}_{m, \ell, \alpha}
    \le
    \norm{a}_{m+1, 2\ell, \gamma}
\end{equation}
for all $a \in \mathcal{A}$ and all $\alpha, \gamma \in
I$. Analogously we have
\begin{equation}
    \label{eq:mKleinermPlusEinsII}
    \sqrt[2^{m+1}]{C^\gamma_{\boldsymbol{\cdot}, \beta}}
    \norm{a}_{m, \ell, \beta}
    \le
    \norm{a}_{m+1, 2\ell+1, \gamma}.
\end{equation}

The following lemma explains now our interest in the quantities $h_{m,
  \ell, \gamma}$.
\begin{lemma}
    \label{lemma:AlmostSeminorms}%
    Let $m \in \mathbb{N}_0$, $\ell = 0, \ldots, 2^m-1$ and $\gamma
    \in I$. Then the maps $\norm{\argument}_{m, \ell, \gamma}$ have
    the following properties:
    \begin{lemmalist}
    \item \label{item:Homogeneous} For all $z \in \mathbb{C}$ and $a
        \in \mathcal{A}$ one has $\norm{za}_{m, \ell, \gamma} = |z|
        \norm{a}_{m, \ell, \gamma}$.
    \item \label{item:Subadditive} For all $a, b \in \mathcal{A}$ one
        has $\norm{a + b}_{m, \ell, \gamma} \le \norm{a}_{m, \ell,
          \gamma} + \norm{b}_{m, \ell, \gamma}$.
    \item \label{item:ProductContinuous} For all $a, b \in
        \mathcal{A}$ one has
        \begin{equation}
            \label{eq:ProductContinuous}
            \norm{a \star b}_{m, \ell, \gamma}
            \le
            \norm{a}_{m+1, \ell, \gamma}
            \norm{b}_{m+1, 2^m + \ell, \gamma}.
        \end{equation}
    \end{lemmalist}
\end{lemma}
\begin{proof}
    The first part is clear. For the second, the statement is clear
    for $m = 0$ so we can proceed by induction on $m$. First we get
    \begin{align*}
        \norm{a + b}_{m+1, 2\ell, \gamma}
        &=
        \left(
            \sum\nolimits_{\alpha \in I}
            \left(\norm{a + b}_{m, \ell, \alpha}\right)^{2^{m+1}}
            C^\gamma_{\alpha, \boldsymbol{\cdot}}
        \right)^{\frac{1}{2^{m+1}}} \\
        &\le
        \left(
            \sum\nolimits_{\alpha \in I}
            \left(
                \norm{a}_{m, \ell, \alpha}
                +
                \norm{b}_{m, \ell, \alpha}
            \right)^{2^{m+1}}
            C^\gamma_{\alpha, \boldsymbol{\cdot}}
        \right)^{\frac{1}{2^{m+1}}} \\
        &\le
        \left(
            \sum\nolimits_{\alpha \in I}
            \left(
                \norm{a}_{m, \ell, \alpha}
            \right)^{2^{m+1}}
            C^\gamma_{\alpha, \boldsymbol{\cdot}}
        \right)^{\frac{1}{2^{m+1}}}
        +
        \left(
            \sum\nolimits_{\alpha \in I}
            \left(
                \norm{b}_{m, \ell, \alpha}
            \right)^{2^{m+1}}
            C^\gamma_{\alpha, \boldsymbol{\cdot}}
        \right)^{\frac{1}{2^{m+1}}} \\
        &=
        \norm{a}_{m+1, 2\ell, \gamma}
        +
        \norm{b}_{m+1, 2\ell, \gamma},
    \end{align*}
    where we used first the induction hypothesis and second the
    Minkowski inequality. The case of $\norm{\argument}_{m+1, 2\ell +
      1, \gamma}$ is analogous. For the third part we first note that
    the case $m = 0$ is a simple application of the Cauchy-Schwarz
    inequality. Then we proceed again by induction on $m$ and get for
    even $\ell$
    \begin{align*}
        \norm{a \star b}_{m, \ell, \gamma}
        &=
        \left(
            \sum\nolimits_{\alpha \in I}
            \left(
                \norm{a \star b}_{m-1, \ell/2, \alpha}
            \right)^{2^m}
            C_{\alpha, \boldsymbol{\cdot}}^\gamma
        \right)^{\frac{1}{2^m}} \\
        &\le
        \left(
            \sum\nolimits_{\alpha \in I}
            \left(
                \norm{a}_{m, \ell/2, \alpha}
                \norm{b}_{m, \ell/2 + 2^{m-1}, \alpha}
            \right)^{2^m}
            C_{\alpha, \boldsymbol{\cdot}}^\gamma
        \right)^{\frac{1}{2^m}} \\
        &\le
        \left(
            \left(
                \sum_{\alpha \in I}
                \left(
                    \norm{a}_{m, \ell/2, \alpha}(a)
                \right)^{2^{m+1}}
                C_{\alpha, \boldsymbol{\cdot}}^\gamma
            \right)
            \left(
                \sum_{\alpha \in I}
                \left(
                    \norm{b}_{m, \ell/2 + 2^{m-1}, \alpha}(a)
                \right)^{2^{m+1}}
                C_{\alpha, \boldsymbol{\cdot}}^\gamma
            \right)
        \right)^{\frac{1}{2^{m+1}}} \\
        &=
        \norm{a}_{m+1, \ell, \gamma}
        \norm{b}_{m+1, 2^m + \ell, \gamma},
    \end{align*}
    where we first used the induction hypothesis and then the
    Cauchy-Schwarz inequality. The case with $\ell$ odd is analogous.
\end{proof}
\begin{remark}
    \label{remark:Motivation}%
    The last part of this lemma was the original motivation for
    considering this recursion scheme: when trying to make the product
    continuous for the zeroth seminorms $\norm{\argument}_{0, 0,
      \gamma}$, i.e.\ for the topology of pointwise convergence with
    respect to the coefficient functionals $e^\gamma$, then one is
    lead rather directly to \eqref{eq:HRecursion}.
\end{remark}

Up to now it may well happen that all the higher ``seminorms''
$\norm{a}_{m, \ell, \gamma}$ will produce $+\infty$ for $a \ne 0$. In
this case, our construction stops and we have not succeeded in finding
a nice locally convex topology making the product continuous. This
motivates to consider the following subset
\begin{equation}
    \label{eq:Anice}
    \mathcal{A}_{\mathrm{nice}} =
    \left\{
        a \in \mathcal{A}
        \; \Big| \;
        \norm{a}_{m, \ell, \gamma}(a) < +\infty
        \;
        \textrm{for all}
        \;
        m \in \mathbb{N}_0, \ell = 0, \ldots 2^m-1, \gamma \in I
    \right\}
    \subseteq \mathcal{A}.
\end{equation}
From Proposition~\ref{lemma:AlmostSeminorms} one immediately
obtains the following corollary. Note that already the seminorms
$\norm{\argument}_{0, 0, \gamma}$ are enough to guarantee a Hausdorff
topology.
\begin{corollary}
    \label{corollary:AniceLCA}%
    The subset $\mathcal{A}_{\mathrm{nice}}$ is a subalgebra of
    $\mathcal{A}$ and the maps $\norm{\argument}_{m, \ell, \gamma}$
    are seminorms on $\mathcal{A}_{\mathrm{nice}}$ such that
    $\mathcal{A}_{\mathrm{nice}}$ becomes a Hausdorff locally convex
    algebra.
\end{corollary}

Thus it is reasonable to restrict our attention to
$\mathcal{A}_{\mathrm{nice}}$ and replace $\mathcal{A}$ by this
subalgebra. Note however, that we do not have any a priori guarantee
that $\mathcal{A}_{\mathrm{nice}}$ is different from $\{0\}$.

Having a locally convex algebra we can form its completion. Here we
will get a rather explicit description which will constitute the first
main theorem of this section:
\begin{theorem}
    \label{theorem:AniceCompleted}%
    Let $\mathcal{A} =
    \mathbb{C}\textrm{-}\spann\{\basis{e}_\alpha\}_{\alpha \in I}$ be
    an algebra with countable vector space basis and assume
    $\mathcal{A}_{\mathrm{nice}} = \mathcal{A}$.
    \begin{theoremlist}
    \item \label{item:Completion} The completion
        $\complete{\mathcal{A}}$ of $\mathcal{A}$ becomes a Fréchet
        algebra. The underlying Fréchet space can be described
        explicitly as
        \begin{equation}
            \label{eq:CompleteAnice}
            \complete{\mathcal{A}}
            =
            \left\{
                a = \sum_{\alpha \in I} a_\alpha \basis{e}_\alpha
                \in \prod_{\alpha \in I}\mathbb{C}\basis{e}_\alpha
                \; \bigg| \;
                \norm{a}_{m, \ell, \gamma} < +\infty
                \; \textrm{for all} \;
                m \in \mathbb{N}_0, \ell = 0, \ldots, 2^m - 1, \gamma
                \in I
            \right\},
        \end{equation}
        viewed as subspace of the vector space $\prod_{\alpha \in
          I}\mathbb{C}\basis{e}_\alpha$. The topology of
        $\complete{\mathcal{A}}$ is finer than the Cartesian product
        topology of $\prod_{\alpha \in I} \mathbb{C}\basis{e}_\alpha$.
    \item \label{item:EvaluationContinuous} The evaluation functionals
        $e^\alpha$ extend to continuous linear functionals
        \begin{equation}
            \label{eq:EvaluationFunctionalsContinuous}
            e^\alpha\colon \complete{\mathcal{A}}
            \longrightarrow
            \mathbb{C}. 
        \end{equation}
    \item \label{item:UnconditionalSchauderBasis} The vectors
        $\{\basis{e}_\alpha\}_{\alpha \in I}$ form an unconditional
        Schauder basis of $\complete{\mathcal{A}}$, i.e.\
        \begin{equation}
            \label{eq:SchauderBasis}
            a = \sum_{\alpha \in I} e^\alpha(a) \basis{e}_\alpha
        \end{equation}
        converges unconditionally for all $a \in
        \complete{\mathcal{A}}$.
    \end{theoremlist}
\end{theorem}
\begin{proof}
    The second part is clear from the continuity of the $e^\alpha$ on
    $\mathcal{A}$: they extend continuously to the completion. In
    fact, the continuity of the $e^\alpha$ on $\mathcal{A}$ can be
    seen directly from the estimate (equality)
    \[
    |e^\alpha(a)| = |a_\alpha| = \norm{a}_{0, 0, \alpha},
    \]
    which also shows that the topology of $\complete{\mathcal{A}}$
    will be finer than the Cartesian product topology which is
    determined by the seminorms $\norm{\argument}_{0, 0, \alpha}$
    alone.  For the remaining statements we can essentially follow the
    arguments from \cite[Thm.~3.9 and
    Thm.~3.10]{beiser.roemer.waldmann:2007a}.  From the recursive
    definition it is clear that there are constants $\mu_{m, \ell,
      \gamma, \alpha_1, \ldots, \alpha_s} \ge 0$ with $s = 2^m$ such
    that
    \[
    h_{m, \ell, \gamma}(a)
    =
    \sum_{\alpha_1, \ldots, \alpha_s}
    \mu_{m, \ell, \gamma, \alpha_1, \ldots, \alpha_s}
    |a_{\alpha_1}| \cdots |a_{\alpha_s}|.
    \tag{$*$}
    \]
    The precise form of the $\mu_{m, \ell, \gamma, \alpha_1, \ldots,
      \alpha_s}$ is complicated but irrelevant for the following.  We
    view $a$ as a function $f_a\colon I^s \longrightarrow \mathbb{C}$
    assigning the $s$-tuple of indices $(\alpha_1, \ldots, \alpha_s)$
    the value $a_{\alpha_1} \cdots a_{\alpha_s}$. The coefficients
    $\mu_{m, \ell, \gamma, \alpha_1, \ldots, \alpha_s}$ define a
    weighted counting measure $\mu_{m, \ell, \gamma}$ on $I^s$ such
    that ($*$) simply becomes
    \[
    h_{m, \ell, \gamma}(a) = \norm{f_a}_{\ell^1}
    \quad
    \textrm{and thus}
    \quad
    \norm{a}_{m, \ell, \gamma}
    =
    \sqrt[2^m]{\norm{f_a}_{\ell^1}},
    \]
    where the $\ell^1$-seminorm is defined with respect to the measure
    $\mu_{m, \ell, \gamma}$. For fixed $m$, $\ell$, $\gamma$ we
    consider the subset $\mathcal{A}_{m, \ell, \gamma}$ of those $a
    \in \prod_{\alpha \in I} \mathbb{C}\basis{e}_\alpha$ with $h_{m,
      \ell, \gamma}(a) < \infty$. These correspond precisely to the
    integrable functions $f_a$. Thus the map $a \mapsto f_a$ is
    topological homeomorphism of metric spaces (and even isometric up
    to the root $\sqrt[2^m]{\argument}$) though clearly not a linear
    homeomorphism.  Now let $(a_i)_{i \in \mathbb{N}}$ be a Cauchy
    sequence in $\complete{\mathcal{A}}$. The seminorms
    $\norm{\argument}_{0, 0, \gamma}$ guarantee the Hausdorff property
    and thus there is at most one limit. Moreover, $(a_i)_{i \in
      \mathbb{N}}$ is a Cauchy sequence in $\mathcal{A}_{m, \ell,
      \gamma}$ and hence convergent by the usual completeness of the
    $\ell^1$-spaces. Note that it might happen that the weighted
    counting measures $\mu_{m, \ell, \gamma}$ lead to nonempty subsets
    of $I^s$ of measure zero. Hence the $\ell^1$-spaces
    $\mathcal{A}_{m, \ell, \gamma}$ are not yet Hausdorff but
    nevertheless complete. Thus $a_i \longrightarrow a$ converges to
    some $a \in \mathcal{A}_{m, \ell, \gamma}$ with $a$ being uniquely
    determined on those indices of nonzero measure. But for those the
    evaluation functionals $e^\alpha$ are continuous and hence the
    limits $a$ obtained for different $m$, $\ell$, $\gamma$ coincide
    on those $\alpha$ with nonzero measure. From
    $\complete{\mathcal{A}} = \bigcap_{m, \ell, \gamma}
    \mathcal{A}_{m, \ell, \gamma}$ and the Hausdorff property obtained
    from $m = 0$ we see that $a_i \longrightarrow a \in
    \complete{\mathcal{A}}$ proving the completeness.  Now let $a \in
    \complete{\mathcal{A}}$ be given and let $K \subseteq I$ be a
    finite subset of indices. Let $a_K = \sum_{\alpha \in K} a_\alpha
    \basis{e}_\alpha$. We claim that $a_K \longrightarrow a$ for any
    exhausting sequence of finite subsets of $I$. Indeed, in our
    measure-theoretic picture $a_K$ corresponds to the function
    $f_{a_K} = f_a \chi_{K^s}$ where $\chi_{K^s}$ is the
    characteristic function of $K^s \subseteq I^s$. Thus we can apply
    the dominated convergence theorem to conclude $f_{a_K}
    \longrightarrow f_a$ in the $\ell^1$-seminorm of $\mathcal{A}_{m,
      \ell, \gamma}$ which implies $a_K \longrightarrow a$ in the
    seminorm $\norm{\argument}_{m, \ell, \gamma}$. This shows the
    unconditional convergence of \eqref{eq:SchauderBasis}. Since the
    evaluation functionals are continuous, we have an unconditional
    Schauder basis for $\complete{\mathcal{A}}$. But this finally
    shows that $\mathcal{A} \subseteq \complete{\mathcal{A}}$ is
    dense, completing also the proof of the first part.
\end{proof}
\begin{remark}
    \label{remark:NotMConvex}%
    In general, the topology will \emph{not} be multiplicatively
    convex and hence $\complete{\mathcal{A}}$ will not be a locally
    $m$-convex Fréchet algebra. The relevant index for this failure of
    $m$-convexity is the index $m$ in $\norm{\argument}_{m, \ell,
      \gamma}$. We also see that the index $\ell$ plays a minor role:
    we can easily take the maximum over the finitely many values of
    $\ell$ for a given $m$ and $\gamma$. In particular, the seminorms
    defined by
    \begin{equation}
        \label{eq:SeminormsWithoutL}
        \norm{a}_{m, \gamma}
        =
        \max_{0 \le \ell \le 2^m -1} \norm{a}_{m, \ell, \gamma}
    \end{equation}
    will yield the same topology as the original ones.
\end{remark}

%
%

\subsection{Second version}
\label{subsec:SecondVersion}

While the above construction will already produce interesting examples
by its own, we will need yet a refinement. The idea is that while we
have now a continuous product we also want certain linear functionals
on $\mathcal{A}$ to be continuous. Thus let $\omega\colon \mathcal{A}
\longrightarrow \mathbb{C}$ be a linear functional. Since we have a
vector space basis this is entirely determined by the coefficients
$\omega_\alpha = \omega(\basis{e}_\alpha)$. In general, $\omega$ will
not be continuous with respect to the seminorms $\norm{\argument}_{m,
  \ell, \gamma}$. If we insist on a topology making $\omega$
continuous as well then we have to add at least the seminorm
$\norm{a}_{0, 0, \omega} = |\omega(a)|$. However, adding just this
single seminorm will spoil the continuity of the product in
general. Thus we have to start again a recursion enforcing the
continuity of $\star$. This will lead to the following definition:
\begin{definition}
    \label{definition:OmegaSeminorms}%
    Let $\omega\colon \mathcal{A} \longrightarrow \mathbb{C}$ be a
    linear functional. Then we define the maps $\norm{\argument}_{m,
      \ell, \omega}\colon \mathcal{A} \longrightarrow [0, +\infty]$
    for $a \in \mathcal{A}$ by
    \begin{equation}
        \label{eq:OmegaSeminorms}
        \norm{a}_{m, \ell, \omega}
        =
        \sqrt[2^m]{
          \sum_{\gamma \in I} |\omega_\gamma| h_{m, \ell, \gamma}(a)
        },
    \end{equation}
    where $m \in \mathbb{N}_0$ and $\ell = 0, \ldots, 2^m - 1$.
\end{definition}
\begin{lemma}
    \label{lemma:OmegaSeminorms}%
    Let $\omega\colon \mathcal{A} \longrightarrow \mathbb{C}$ be a
    linear functional. Then for all $m \in \mathbb{N}_0$ and $\ell =
    0, \ldots, 2^m - 1$ the following statements hold:
    \begin{lemmalist}
    \item \label{item:OmegaSeminormHomogeneous} For all $z \in
        \mathbb{C}$ and $a \in \mathcal{A}$ one has $\norm{za}_{m,
          \ell, \omega} = |z| \norm{za}_{m, \ell, \omega}$.
    \item \label{item:OmegaSeminormSubadd} For all $a, b \in
        \mathcal{A}$ one has $\norm{a + b}_{m, \ell, \omega} \le
        \norm{a}_{m, \ell, \omega} + \norm{b}_{m, \ell, \omega}$.
    \item \label{item:OmegaSeminormProduct} For all $a, b \in
        \mathcal{A}$ one has
        \begin{equation}
            \label{eq:OmegaSeminormProduct}
            \norm{a \star b}_{m, \ell, \omega}
            \le
            \norm{a}_{m+1, \ell, \omega}
            \norm{b}_{m+1, 2^m + \ell, \omega}.
        \end{equation}
    \end{lemmalist}
\end{lemma}
\begin{proof}
    Using the properties of the maps $h_{m, \ell, \gamma}$ according
    to Lemma~\ref{lemma:AlmostSeminorms} all the statements are a
    simple verification.
\end{proof}

Of course, it may again well happen that the quantities $\norm{a}_{m,
  \ell, \omega}$ are all $+\infty$ for $a \ne 0$ even though the
$\norm{a}_{m, \ell, \gamma}$ are always finite. Thus we typically add
nontrivial conditions when we require $\norm{a}_{m, \ell, \omega} <
\infty$. As before, we consider only those algebra elements where this
is finite and obtain a subalgebra
$\mathcal{A}_{\omega\textrm{-}\mathrm{nice}}$ of $\mathcal{A}$ inside
the previous $\mathcal{A}_{\mathrm{nice}}$. Moreover, we can even
select an arbitrary family $\Omega = \{\omega\colon \mathcal{A}
\longrightarrow \mathbb{C}\}$ of linear functionals and consider the
seminorms $\norm{\argument}_{m, \ell, \omega}$ for all $\omega \in
\Omega$. We shall refer to this topology as the $\Omega$-nice
topology. This will lead to subalgebras
\begin{equation}
    \label{eq:AOmegaAomegaAniceA}
    \mathcal{A}_{\Omega\textrm{-}\mathrm{nice}}
    =
    \bigcap_{\omega \in \Omega}
    \mathcal{A}_{\omega\textrm{-}\mathrm{nice}}
    \subseteq
    \mathcal{A}_{\omega\textrm{-}\mathrm{nice}}
    \subseteq
    \mathcal{A}_{\mathrm{nice}}
    \subseteq
    \mathcal{A},
\end{equation}
which of course might be trivial, depending on the choice of $\Omega$.
Clearly, the inclusions are continuous and the topologies become
coarser when moving to the right in \eqref{eq:AOmegaAomegaAniceA}.  If
$\Omega$ is at most countably infinite then we still get a countable
set of seminorms. Note that taking $\Omega = \{e^\alpha, \alpha \in
I\}$ just reproduces the seminorms $\norm{\argument}_{m, \ell,
  \gamma}$ we had before since $e^\alpha (\basis{e}_\gamma) =
\delta^\alpha_\gamma$.  Putting things together properly, we get the
analogous statement to Theorem~\ref{theorem:AniceCompleted} also in
this more general situation:
\begin{theorem}
    \label{theorem:Omega}%
    Let $\mathcal{A} =
    \mathbb{C}\textrm{-}\spann\{\basis{e}_\alpha\}_{\alpha \in I}$ be
    an algebra with countable vector space basis. Moreover, let
    $\Omega$ be a family of linear functionals on $\mathcal{A}$ and
    assume that $\mathcal{A}_{\Omega\textrm{-}\mathrm{nice}} =
    \mathcal{A}$.
    \begin{theoremlist}
    \item \label{item:CompletionSecondVersion} The completion
        $\complete{\mathcal{A}}$ of $\mathcal{A}$ becomes a complete
        locally convex algebra which is Fréchet if $\Omega$ is
        countable. The underlying locally convex space can be
        described explicitly as
        \begin{equation}
            \label{eq:CompleteAniceSecondVersion}
            \complete{\mathcal{A}}
            =
            \left\{
                a \in
                \prod_{\alpha \in I}\mathbb{C}\basis{e}_\alpha
                \; \bigg| \;
                \norm{a}_{m, \ell, \gamma},
                \norm{a}_{m, \ell, \omega} < + \infty                
                \; \textrm{for all} \;
                m \in \mathbb{N}_0, \ell = 0, \ldots, 2^m - 1,
                \gamma \in I, \omega \in \Omega
            \right\},
        \end{equation}
        viewed as subspace of $\prod_{\alpha \in I}
        \mathbb{C}\basis{e}_\alpha$ as before.
    \item \label{item:EvaluationContinuousSecondVersion} The
        evaluation functionals $e^\alpha$ as well as the functionals
        $\omega \in \Omega$ extend to continuous linear functionals
        \begin{equation}
            \label{eq:EvaluationFunctionalsContinuousSecondVersion}
            e^\alpha, \omega\colon \complete{\mathcal{A}}
            \longrightarrow
            \mathbb{C}. 
        \end{equation}
    \item \label{item:UnconditionalSchauderBasisSecondVersion} The
        vectors $\{\basis{e}_\alpha\}_{\alpha \in I}$ form an
        unconditional Schauder basis of $\complete{\mathcal{A}}$,
        i.e.\
        \begin{equation}
            \label{eq:SchauderBasisSecondVersion}
            a = \sum_{\alpha \in I} e^\alpha(a) \basis{e}_\alpha
        \end{equation}
        converges unconditionally for all $a \in
        \complete{\mathcal{A}}$.
    \item \label{item:OmegaFiner} If $\Omega' \subseteq \Omega$ then
        $\mathcal{A}_{\Omega'\textrm{-}\mathrm{nice}} = \mathcal{A}$,
        too, and the $\Omega$-nice topology is finer than the
        $\Omega'$-nice topology.
    \end{theoremlist}
\end{theorem}

After this presentation of the general features of the subalgebra
$\mathcal{A}_{\mathrm{nice}}$ as well as
$\mathcal{A}_{\Omega\textrm{-}\mathrm{nice}}$ the remaining but also
quite nontrivial question is how we can guarantee that these
subalgebras are different from $\{0\}$ at all. Unfortunately, we have
no general argument for this: it will boil down to a case by case
study in the examples. This clearly limits the above method, however,
many seemingly unrelated examples turn out to be just the above
construction.

Even if one succeeds in showing that $\mathcal{A}_{\mathrm{nice}} =
\mathcal{A}$ the construction will depend strongly on the chosen
basis: this is really not avoidable as taking the absolute values of
the structure constants in the definition of the seminorms is not at
all well-behaved under a change of the basis. We can hardly expect any
reasonable way to compare the results for different bases. Even worse,
as we shall see in the examples, a simple rescaling of the basis may
change the convergence scheme and also the topology. This can be taken
as advantage in order to cure bad behavior by rescaling.

Since the basis enters in such a crucial way, one should take this
construction only for algebras where one has a quite distinguished
choice of a basis. In fact, this is more often the case than one might
first think.

\begin{remark}
    \label{remark:CuntzConstruction}%
    There are other possibilities of constructing a countable system
    of seminorms on $\mathcal{A}$ for which the product becomes
    continuous: In fact, elaborating on the construction in the proof
    of \cite[Prop.~2.1]{cuntz:2005a} one can again start with the
    evaluation functionals $e^\alpha$ and the corresponding seminorms
    $\norm{\argument}_\alpha = |e^\alpha(\argument)|$ making them
    continuous. Then the recursion in \cite[Prop.~2.1]{cuntz:2005a}
    will produce a countable system of seminorms build on top of the
    $\norm{\argument}_\alpha$ such that the product becomes
    continuous. The completion will be a Fréchet algebra, too, with a
    topology depending on the choice of the basis in a similarly
    ``obscure'' way than our construction.  However, for the examples
    we shall study, our construction will be manageable to yield quite
    explicit properties of the Fréchet algebras in question. We leave
    it as a future task to examine further relations between the two
    approaches.
\end{remark}

%
%

\section{First examples: polynomials, matrices, and group algebras}
\label{sec:FirstExamplesPolynomials}

In this section we collect some simple examples to illustrate the
general method developed in the previous section.

%
%

\subsection{Polynomials}
\label{subsec:Polynomials}

We first consider the algebra of polynomials in one variable
$\mathcal{A} = \mathbb{C}[z]$ with the basis given by the
monomials. For the structure constants we see
\begin{equation}
    \label{eq:StructureConstantsPolynomials}
    z^n z^m = \sum_{k=0}^\infty C_{nm}^k z^k
    \quad
    \textrm{with}
    \quad
    C_{nm}^k = \delta_{k, m+n},
\end{equation}
where $n, m, k \ge 0$. Since $\mathbb{C}[z]$ is commutative we can
ignore the index $\ell$ in this case. We first compute the constants
\eqref{eq:CSeriesConverge} giving
\begin{equation}
    \label{eq:Cdotn}
    C^k_n
    = \sum_{m=0}^\infty C^k_{nm}
    =
    \begin{cases}
        0 & \textrm{for} \; n > k \\
        1 & \textrm{for} \; n \le k.
    \end{cases}
\end{equation}
In particular, $C^k_n < \infty$. From the recursion we see that
\begin{equation}
    \label{eq:hmkPoly}
    h_{m, k}(a)
    = \sum_{n=0}^\infty h_{m-1, n}(a)^2 C^k_n
    = \sum_{n=0}^k h_{m-1, n}(a)^2,
\end{equation}
starting with $h_{0, k}(a) = |a_k|$ for $a = \sum_n a_n z^n$.  Thus
for $h_{m, k}(a)$ only finitely many $h_{m-1, n}(a)$ contribute. By a
simple induction we get the following result:
\begin{proposition}
    \label{proposition:Polynomials}%
    For the polynomials $\mathcal{A} = \mathbb{C}[z]$ the topology of
    the seminorms $\norm{\argument}_{m, k}$ is the topology of the
    Cartesian product. The completion gives the locally
    multiplicatively convex Fréchet algebra $\mathbb{C}[[z]]$.
\end{proposition}

Let us now also implement the second version for the polynomial
algebra. Since the first version gives the formal power series
$\mathbb{C}[[z]]$ the evaluation functionals $\delta_p$ for $p \in
\mathbb{C}$ cannot be continuous unless $p \ne 0$. Thus we want to
enforce their continuity. Since the absolute value of the evaluation
of the monomials at $p$ only depends on $|p|$ we can restrict
ourselves to some positive radius $R > 0$ on which we want to
evaluate. Then the additional seminorms are
\begin{equation}
    \label{eq:NewSeminormsPolyEins}
    \norm{a}_{m, R}
    = \sqrt[2^m]{\sum_{k=0}^\infty R^k h_{m, k}(a)}.
\end{equation}
Now $h_{1, k}(a) \ge |a_n|^2$ for all $k \ge n$ and hence
$\norm{a}_{1, R}$ diverges for $R \ge 1$ unless $a = 0$. In this case,
the second version fails. Consider now the case $R < 1$ and let $a \in
\mathbb{C}[z]$. Then a simple induction shows that $h_{m, k}(a) \le
c_m k^{\alpha_m}$ with some constant $c_m > 0$ depending on $a$. Thus
we see that in this case the series needed for $\norm{a}_{m, R}$
converges and hence the second version works. Moreover, for $|p| \le
R$ we have $|\delta_p(a)| \le \norm{a}_{0, R}$ and hence the topology
we get is finer than the uniform topology on the disk of radius
$R$. After completion, we get a subalgebra of those functions which
are holomorphic on at least the \emph{closed} disk\footnote{If $f$ has
  a convergent Taylor expansion for all $|z| \le R$ then $f$ has in
  fact a holomorphic extension to some slightly larger open disk.} of
radius $R < 1$. In general, not every such function will be in our
completion: take $R < r < 1$ with $r^2 < R$ and consider the
coefficients $a_n = \frac{1}{r^n}$ which define a holomorphic function
$a(z) = \sum_{n=0}^\infty a_n z^n$ on the closed disk of radius
$R$. Then $h_{1, k}(a) \ge \frac{1}{r^{2k}}$ and hence the series
needed for $\norm{a}_{1, R}$ will no longer converge. Hence we have a
proper subalgebra and the topology will be \emph{strictly} finer than
the (Banach) topology of uniform convergence on the closed disk of
radius $R$.
\begin{proposition}
    \label{proposition:PolynomialsSecondVersionEins}%
    For $R \ge 1$ the second version of our construction with respect
    to the basis of the monomials fails for $\mathbb{C}[z]$ while for
    $R < 1$ we get a proper subalgebra of the Banach space of
    holomorphic functions on the closed disk of radius $R$ with a
    strictly finer Fréchet topology. All evaluation functionals for
    points $p$ with $|p| \le R$ are continuous.
\end{proposition}

We illustrate now the dependence on the basis: motivated by the usual
Taylor formula we can also rewrite a polynomial as $a = \sum_n
\tilde{a}_n \frac{z^n}{n!}$. Now the rescaled monomials
$\frac{z^n}{n!}$ will serve as basis. The new structure constants will
be rescaled as well and this will result in the new recursion
\begin{equation}
    \label{eq:NewRecursionForhmkPoly}
    \tilde{h}_{m+1, k}(a)
    =
    \sum_{n=0}^k \tilde{h}_{m, n}(a)^2 \binom{k}{n},
\end{equation}
with $\tilde{h}_{0, k}(a) = |\tilde{a}_k|$. For the first version, the
same argument as for Proposition~\ref{proposition:Polynomials} still
applies and hence we get $\mathbb{C}[[z]]$ as completion. But for the
second version things will change: the evaluation of the basis vectors
on $p \in \mathbb{C} \setminus \{0\}$ are now $\frac{p^n}{n!}$ and
hence the additional seminorms are
\begin{equation}
    \label{eq:NewSeminormsPolyZwei}
    \norm{a}_{m, R}^{\tilde{}}
    =
    \sqrt[2^m]{
      \sum_{k = 0}^\infty \frac{R^k}{k!} \tilde{h}_{m, k}(a)
    },
\end{equation}
where again $R = |p| > 0$. We denote the nice part of $\mathbb{C}[z]$
with respect to the given $R$ by $\mathcal{A}_R$ and its completion by
$\complete{\mathcal{A}}_R$. The following argument shows in particular
that $\mathcal{A}_R = \mathbb{C}[z]$ as vector spaces.

Suppose that $a \in \mathbb{C}[[z]]$ has Taylor coefficients which
satisfy a sub-factorial growth\footnote{Here ``sub-factorial'' is not
  to be confused with the sub-factorial $!n$ which we shall never need
  in this work.}, i.e.\ for all $\epsilon > 0$ there is a constant $c_0
> 0$ with $|\tilde{a}_k| \le c_0 (k!)^\epsilon$. Then a simple
induction shows that for all $\epsilon > 0$ also $h_{m, k}(a)$ can be
bounded by $c_m (k!)^\epsilon$. It will be important to have this
bound not only for one $\epsilon$. Taking now $\epsilon$ sufficiently
small shows that \eqref{eq:NewSeminormsPolyZwei} converges and thus
such a formal series $a$ belongs to the completion
$\complete{\mathcal{A}}_R$, no matter what $R > 0$ is.  Now fix $R >
0$ and suppose that $a \in \mathbb{C}[[z]]$ belongs to
$\complete{\mathcal{A}}_R$. Then we have the estimate
\begin{equation}
    \label{eq:PolynomialEstimatehmk}
    \tilde{h}_{m, k}(a)
    \le
    \frac{k!}{R^k} \left(\norm{a}^{\tilde{}}_{m, R}\right)^{2^m}
\end{equation}
for all $m$ and $k$. By taking always only the term with $n = k$ in
the recursion \eqref{eq:NewSeminormsPolyZwei} it is clear that
$|\tilde{a}_k|^{2^m} \le \tilde{h}_{m, k}(a)$ and hence
\begin{equation}
    \label{eq:CoolEstimateForPolynomialsWithFactorials}
    |\tilde{a}_k|
    \le
    \frac{\sqrt[2^m]{k!}}{R^{\frac{k}{2^m}}}
    \norm{a}^{\tilde{}}_{m, R}.
\end{equation}
Thus $a$ has a sub-factorial growth as above. This results in the
following Proposition:
\begin{proposition}
    \label{proposition:PolynomialsSecondVersionZwei}%
    Let $R > 0$.
    \begin{propositionlist}
    \item \label{item:CoolPolySubFac} $\complete{\mathcal{A}}_R$
        consists of those entire functions $a = \sum_{n = 0}^n
        \tilde{a}_n \frac{z^n}{n!}$ having Taylor coefficients with
        sub-factorial growth, i.e.\ for all $\epsilon > 0$ there exists
        a constant $c > 0$ with $|\tilde{a}_n| \le c (n!)^\epsilon$.
    \item \label{item:FrechetNotDependentOnR} The Fréchet topology of
        $\complete{\mathcal{A}}_R$ does not depend on $R$ and is
        strictly finer than the usual Fréchet topology of the entire
        functions $\mathcal{O}(\mathbb{C})$.
    \item \label{item:PolyAllEvalCont} The evaluation functionals are
        continuous for all $p \in \mathbb{C}$.
    \item \label{item:PolyNotLMC} The algebra
        $\complete{\mathcal{A}}_R$ is not locally multiplicatively
        convex.
    \item \label{item:SubfactorialForPolynomialsNuc} An equivalent
        defining system of seminorms is given by
        \begin{equation}
            \label{eq:SupSubfactorialPoly}
            \norm{a}_\epsilon
            =
            \sup_{n \in \mathbb{N}_0}
            \frac{|\tilde{a}_n|}{(n!)^\epsilon},
        \end{equation}
        where $0 < \epsilon < 1$. As a Fréchet space,
        $\complete{\mathcal{A}}_R$ is isomorphic to the Köthe space
        $\Lambda$ of sequences with sub-factorial growth. It is
        strongly nuclear and the Schauder basis is even absolute.
    \end{propositionlist}
\end{proposition}
\begin{proof}
    We have already shown the first statement. For two different $R$,
    $R'$ we have the same vector space for the completions and clearly
    if $R \le R'$ then the topology of $\mathcal{A}_R$ is coarser than
    that topology of $\mathcal{A}_{R'}$, which one sees directly from
    the seminorms \eqref{eq:NewSeminormsPolyZwei}. But then the two
    Fréchet spaces $\complete{\mathcal{A}}_R$ and
    $\complete{\mathcal{A}}_{R'}$ already coincide by the open mapping
    theorem. In particular, all the seminorms
    $\norm{\argument}^{\tilde{}}_{m, R'}$ will be continuous on
    $\complete{\mathcal{A}}_R$. The seminorms
    $\{\norm{\argument}^{\tilde{}}_{1, R}\}_{R > 0}$ constitute a
    defining set of seminorms for the topology of locally uniform
    convergence in $\mathcal{O}(\mathbb{C})$ and hence the second part
    is shown as well. The third part is clear and the fourth part
    follows from the fact that $\complete{\mathcal{A}}_R$ does not
    have an entire calculus: otherwise it would be equal to
    $\mathcal{O}(\mathbb{C})$. For the last part we first observe that
    \eqref{eq:CoolEstimateForPolynomialsWithFactorials} implies that
    the seminorms $\norm{\argument}_\epsilon$ can be estimated by the
    seminorms $\norm{\argument}_{m, R}^{\tilde{}}$ for appropriate
    $m$. A careful examination of the bound $h_{m, k}(a) \le c_m
    (k!)^\epsilon$ shows also the reverse estimate: however, this is
    also clear by general arguments as $\complete{\mathcal{A}}_R$ is
    clearly a Fréchet space for both topologies and one is finer than
    the other. Hence by the open mapping theorem they coincide. Then
    the remaining statements follow from general properties of the
    Köthe space $\Lambda$, see Appendix~\ref{sec:SubfactorialGrowth}.
\end{proof}

We see that already for this simple example of polynomials we get a
rather rich structure and interesting completions. Moreover, the
dependence on the chosen basis is manifest in this example.

%
%

\subsection{Laurent polynomials}
\label{subsec:LaurentPolynomials}

As a second example we consider the Laurent polynomials $\mathcal{A} =
\mathbb{C}[z, z^{-1}]$. Here the structure constants are similar to
those of $\mathbb{C}[z]$, explicitly given by
\begin{equation}
    \label{eq:LaurentSructureConstants}
    z^n z^m = \sum_{k=-\infty}^\infty C_{nm}^k z^k
    \quad
    \textrm{with}
    \quad
    C_{nm}^k
    = \delta_{n+m, k}
\end{equation}
with $n, m, k \in \mathbb{Z}$ instead of $\mathbb{N}_0$ as above. This
has a nontrivial impact. We have
\begin{equation}
    \label{eq:CdotnLaurent}
    C^k_n = \sum_{m=-\infty}^\infty C^k_{nm} = 1,
\end{equation}
since now there is always precisely one $m$ matching the condition
$n+m = k$. Hence the recursion is
\begin{equation}
    \label{eq:hmkLaurent}
    h_{m, k}(a) = \sum_{n=-\infty}^\infty h_{m-1, n}(a)^2,
\end{equation}
with $h_{0, k}(a) = a_k$ as before. In particular, we have
\begin{equation}
    \label{eq:hEinsk}
    h_{1,k}(a) = \sum_{n=-\infty}^\infty |a_n|^2,
\end{equation}
which is finite for every $a \in \mathbb{C}[z, z^{-1}]$ but
\emph{independent} of $k$. Thus the next iteration gives $h_{2, k}(a)
= +\infty$ unless $a = 0$. Our constructions \emph{fails} in this
case.

In order to fix this divergence we will now rescale the basis
first. Instead of taking the monomials $z^n$ as basis we consider
$\basis{e}_n = \frac{1}{|n|!} z^n$, where $n \in \mathbb{Z}$ as
before. Of course this is sort of arbitrary but here it is again
motivated by the prefactors in the usual Laurent expansion around
$z=0$, i.e.\ we write now
\begin{equation}
    \label{eq:LaurentPolynomial}
    a = \sum_{n \in \mathbb{Z}} \frac{a_n}{|n|!} z^n
\end{equation}
for $a \in \mathbb{C}[z, z^{-1}]$ with only finitely many $a_n$
different from zero. The structure constants (again denoted by
$C^k_{nm}$) are now given by
\begin{equation}
    \label{eq:LaurentRescaledStructureConstants}
    \basis{e}_n \basis{e}_m
    =
    \sum_{k = -\infty}^\infty C^k_{nm} \basis{e}_k
    \quad
    \textrm{with}
    \quad
    C^k_{nm} = \frac{|k|!}{|n|!|m|!} \delta_{n+m, k}.
\end{equation}
The corresponding constants \eqref{eq:CSeriesConverge} are given by
\begin{equation}
    \label{eq:ConstantRescaledForLaurent}
    C^k_n
    = \sum_{m = - \infty}^\infty C^k_{nm}
    = \frac{|k|!}{|n|!|k-n|!} < \infty,
\end{equation}
now depending on both indices $k$ and $n$. The recursion for the
seminorms is changed to
\begin{equation}
    \label{eq:tildehmkLaurent}
    h_{m, k}(a)
    =
    \sum_{n=-\infty}^\infty
    h_{m-1, n}(a)^2 \frac{|k|!}{|n|!|k-n|!}
\end{equation}
with $h_{0, k}(a) = a_k$ according to
\eqref{eq:LaurentPolynomial}. Now we want to show that for $a \in
\mathbb{C}[z, z^{-1}]$ all the quantities $h_{m, k}(a)$ are finite.
In fact, we shall determine the completion directly:
\begin{proposition}
    \label{proposition:CompletionLaurent}%
    We have $\mathcal{A}_{\mathrm{nice}} = \mathcal{A}$. Moreover, a
    formal series $a \in \mathbb{C}[[z, z^{-1}]]$ belongs to the
    completion $\complete{\mathcal{A}}$ iff its Laurent coefficients
    $a_n$ have sub-factorial growth, i.e.\ for all $\epsilon > 0$ there
    is a constant $c > 0$ depending on $a$ with $|a_n| \le c
    (|n|!)^\epsilon$.
\end{proposition}
\begin{proof}
    As an inequality in $[0, +\infty]$ we first prove that for all $a
    \in \mathbb{C}[[z, z^{-1}]]$ we have
    \[
    |a_n|
    \le
    \sqrt[2^m]{\frac{|n|!|n-k|!}{|k|!}} \norm{a}_{m, k},
    \tag{$*$}
    \]
    where $m \ge 1$ and $k \in \mathbb{Z}$. Indeed, for $m = 1$ we
    estimate
    \[
    h_{1, k}(a)
    =
    \sum_{\ell \in \mathbb{Z}}
    |a_\ell|^2 \frac{|k|!}{|\ell|!|k-\ell|!}
    \ge
    |a_n|^2 \frac{|k|!}{|n|!|k-n|!},
    \]
    which gives ($*$). Now we proceed by induction: assuming ($*$) for
    $m$ gives
    \[
    h_{m+1, k}(a)
    \ge
    \sum_{\ell \in \mathbb{Z}}
    |a_\ell|^{2^{m+1}}
    \left(\frac{|\ell|!}{|n|!|n - \ell|!}\right)^2
    \frac{|k|!}{|\ell|!|k-\ell|!}
    \ge
    |a_n|^{2^{m+1}} \frac{|k|!}{|n|!|k-n|!},        
    \]
    which shows ($*$) also for $m+1$. Now taking $k = 0$ gives the
    estimate
    \[
    |a_n| \le \sqrt[2^m]{|n|!} \norm{a}_{m + 1, 0}.
    \tag{$**$}
    \]
    Hence, if $a \in \complete{\mathcal{A}_{\mathrm{nice}}}$,
    i.e. $\norm{a}_{m, k} < \infty$ for all $m$ and $k$, then $a$ has
    sub-factorial growth. Conversely, assume that $a$ has
    sub-factorial growth. We have to show that all the quantities
    $h_{m, k}(a)$ are finite. In fact, we claim that for each $m$ also
    $h_{m, k}(a)$ behaves in a sub-factorial way. For $m = 0$ this is
    the assumption about $a$ itself. Now assume that for all $\epsilon
    > 0$ we have $h_{m, n}(a) \le c_m (|n|!)^\epsilon$. Then we
    estimate for $k \ge 0$ and $\epsilon$ small enough
    \begin{align*}
        h_{m+1, k}(a)
        &=
        \sum_{n=0}^\infty h_{m, -n}(a)^2 \frac{k!}{n!(k+n)!}
        +
        \sum_{n=0}^k h_{m, n}(a)^2 \frac{k!}{n!(k-n)!}
        +
        \sum_{n=k+1}^\infty h_{m, n}(a)^2 \frac{k!}{n!(n-k)!} \\
        &\le
        \sum_{n=0}^\infty c_m^2 (n!)^{2\epsilon} \frac{1}{n!}
        +
        \sum_{n=0}^k c_m^2 (n!)^{2\epsilon} \binom{k}{n}
        +
        \sum_{n=k+1}^\infty
        c_m^2 (n!)^{2\epsilon} \frac{k!}{n!(n-k)!} \\
        &\le
        c
        +
        c_m^2 (k!)^{2\epsilon} 2^k
        +
        c_m^2 (k!)^{2\epsilon}
        \sum_{n=k+1}^\infty
        \frac{(k!)^{1 - 2\epsilon}}{(n!)^{1 - 2\epsilon}}
        \frac{1}{(n-k)!} \\
        &\le
        c
        +
        c' c_m^2 (k!)^{3\epsilon}
        +
        c_m^2 (k!)^{2\epsilon} \E,
    \end{align*}
    where $c$ is the finite value of the first sum and $c'$ is a
    constant such that $2^k \le c' (k!)^\epsilon$. We see that this
    can be estimated by $h_{m+1, k}(a) \le c_{m+1} (k!)^{3\epsilon}$
    with an appropriate $c_{m+1}$. Since $\epsilon > 0$ was arbitrary
    we get again a sub-factorial growth, proving our claim. The case
    $k < 0$ is analogous. Hence all the quantities $h_{m, k}(a)$ are
    finite and $a$ belongs to the completion.
\end{proof}

We shall now show that the elements of $\complete{\mathcal{A}}$ can
still be evaluated at points $p \in \mathbb{C} \setminus \{0\}$. Hence
we can still interpret them as \emph{functions}.  We are able to show
the following result:
\begin{proposition}
    \label{proposition:LaurentOnC}%
    Let $a \in \complete{\mathcal{A}}$.
    \begin{propositionlist}
    \item \label{item:EvaluationContinuousLaurent} For $p \in
        \mathbb{C} \setminus \{0\}$ the evaluation $a(p)$ is
        well-defined and yields a continuous character
        \begin{equation}
            \label{eq:ContinuousCharLaurent}
            \delta_p\colon \complete{\mathcal{A}} \ni a
            \; \mapsto \; a(p)
            \in \mathbb{C}.
        \end{equation}
    \item \label{item:WeHaveHolFunLaurent} Any $a \in
        \complete{\mathcal{A}}$ can be viewed as holomorphic function
        $a \in \mathcal{O}(\mathbb{C} \setminus \{0\})$ via
        \eqref{eq:ContinuousCharLaurent}.
    \item \label{item:ContinuousInclusionLaurent} The map
        $\complete{\mathcal{A}} \longrightarrow \mathcal{O}(\mathbb{C}
        \setminus \{0\})$ is a continuous injective algebra
        homomorphism with dense image.
    \item \label{item:LaurentNotLMC} The topology of
        $\complete{\mathcal{A}}$ is strictly finer than the locally
        uniform topology of $\mathcal{O}(\mathcal{C} \setminus \{0\})$
        and it is not locally multiplicatively convex.
    \end{propositionlist}
\end{proposition}
\begin{proof}
    Let $a = \sum_{n \in \mathbb{Z}} \frac{a_n}{|n|!} z^n \in
    \complete{\mathcal{A}}$ be given and $p \in \mathbb{C} \setminus
    \{0\}$. Then we consider the convergence of the series $a(p) =
    \sum_{n \in \mathbb{Z}} \frac{a_n}{|n|!}  p^n$. Using the estimate
    for the Laurent coefficients $a_n$ as in ($**$) from the proof of
    Proposition~\ref{proposition:CompletionLaurent} for $m = 1$ gives
    \[
    |a(p)|
    \le
    \sum_{n \in \mathbb{Z}}
    \frac{|a_n|}{|n|!} |p|^n
    \le
    \norm{a}_{2, 0}
    \sum_{n \in \mathbb{Z}}
    \frac{|p|^n}{\sqrt{|n|!}}
    =
    c_{|p|} \norm{a}_{2, 0},
    \]
    which is the desired continuity estimate for $\delta_p$ since the
    remaining series $c_{|p|}$ converges for all $p \ne 0$. Clearly,
    the evaluation is a character of the algebra $\mathcal{A}$ and
    hence also for $\complete{\mathcal{A}}$ by continuity. This shows
    the first part and the second is clear as we have convergence for
    all $p \in \mathbb{C} \setminus \{0\}$. In fact, this is also
    clear from Proposition~\ref{proposition:CompletionLaurent}. For
    the next part, recall that the locally uniform topology of
    $\mathcal{O}(\mathbb{C} \setminus \{0\})$ can be obtained e.g.\
    from the seminorms $\norm{\argument}_R$ with $R > 1$ where
    \[
    \norm{f}_R = \sum_{n \in \mathbb{Z}} \frac{|f_n|}{|n|!} R^n,
    \]
    where $f(z) = \sum_{n \in \mathbb{Z}} \frac{f_n}{|n|!} z^n$ is the
    convergent Laurent expansion of $f$. Then the above computation
    shows $\norm{a}_R \le c_R \norm{a}_{2, 0}$ which is the continuity
    of the inclusion. The image is dense as already $\mathbb{C}[z,
    z^{-1}]$ is dense. The last part is clear as
    $\complete{\mathcal{A}}$ does not coincide with
    $\mathcal{O}(\mathbb{C} \setminus \{0\})$ and it has no entire
    calculus.
\end{proof}
\begin{remark}
    \label{remark:SubFacorms}%
    Again, a simple verification shows that the seminorms
    $\norm{a}_\epsilon = \sup_n \frac{|a_n|}{(n!)^\epsilon}$ for
    $\epsilon > 0$ will produce an equivalent system of
    seminorms. Hence, as a Fréchet space, also here
    $\complete{\mathcal{A}}$ is just the (strongly nuclear) Köthe
    space $\Lambda$ of sequences with sub-factorial growth and the
    Schauder basis is absolute. Note that indexing sequences by $n \in
    \mathbb{Z}$ does not cause any difficulties here.
\end{remark}

%
%

\subsection{Infinite matrices}
\label{subsec:InfiniteMatrices}

The third example is the noncommutative and nonunital algebra
$\mathcal{A} = M_\infty(\mathbb{C})$ of infinite matrices with only
finitely many nonzero entries. Let $E_{ij}$ be the matrix with $1$ in
the $(i,j)$-th position and zeros everywhere else. Then $\mathcal{A} =
\mathbb{C}\textrm{-}\spann\{E_{ij}\}_{i, j \in \mathbb{N}}$ and the
matrix multiplication gives the structure constants
\begin{equation}
    \label{eq:StructureConstantMatrix}
    E_{ij} E_{kl}
    =
    \sum_{r, s = 1}^\infty C^{(r, s)}_{(i, j) (k, l)} E_{rs}
    \quad
    \textrm{with}
    \quad
    C^{(r, s)}_{(i, j), (k, l)} = \delta_{ir} \delta_{jk} \delta_{ls}
\end{equation}
as usual. The corresponding constants from \eqref{eq:CSeriesConverge}
are now given by
\begin{equation}
    \label{eq:ConstantsForMatrix}
    C^{(r, s)}_{(i, j), \boldsymbol{\cdot}}
    =
    \sum_{k, l = 1}^\infty C^{(r, s)}_{(i, j), (k, l)}
    =
    \sum_{k, l = 1}^\infty \delta_{ir} \delta_{jk} \delta_{ls}
    =
    \delta_{ir}
\end{equation}
as well as
\begin{equation}
    \label{eq:ConstantsForMatrixII}
    C^{(r, s)}_{\boldsymbol{\cdot}, (k, l)}
    =
    \sum_{i, j = 1}^\infty C^{(r, s)}_{(i, j), (k, l)}
    =
    \sum_{i, j = 1}^\infty \delta_{ir} \delta_{jk} \delta_{ls}
    =
    \delta_{ls}.
\end{equation}
Note that now we indeed have a noncommutative algebra and hence two
types of such constants.  We get the recursion
\begin{equation}
    \label{eq:hmlMatrixFirstVersion}
    h_{m+1, 2\ell, (r, s)}(A)
    =
    \sum_{i, j = 1}^\infty
    h_{m, \ell, (i, j)}(A)^2 C^{(r, s)}_{(i, j), \boldsymbol{\cdot}}
    =
    \sum_{j = 1}^\infty h_{m, \ell, (r, j)}(A)^2 
\end{equation}
and analogously for the odd case
\begin{equation}
    \label{eq:hmlMatrixFirstVersionOddCase}
    h_{m+1, 2\ell+1, (r, s)}(A)
    =
    \sum_{k, l = 1}^\infty
    h_{m, \ell, (k, l)}(A)^2 C^{(m, n)}_{\boldsymbol{\cdot}, (k, l)}
    =
    \sum_{k = 1}^\infty h_{m, \ell, (k, s)}(A)^2.
\end{equation}
We see that the even case of $2\ell$ does not depend on the index $s$
while the odd case $2\ell +1$ is independent on $r$. After the second
iteration we get infinite sums over constants and hence $h_{2, \ell,
  (r, s)}(A) = + \infty$ for all $A \ne 0$. Again, we conclude that
the method fails for this choice of a basis.

Thus we rescale the basis as we did already in the Laurent case: we
shall discuss two different options here. First we consider the new
basis
\begin{equation}
    \label{eq:MatrixBasisFactorial}
    \hat{E}_{ij} = \frac{1}{\sqrt{i!j!}} E_{ij}.
\end{equation}
Consequently, the new structure constants with respect to this basis
are given by
\begin{equation}
    \label{eq:MatrixSCFactorial}
    \hat{C}_{(i, j), (k, l)}^{(r, s)}
    =
    \frac{\sqrt{r!s!}}{\sqrt{i!j!k!l!}} C^{(r, s)}_{(i, j), (k, l)}
    =
    \frac{1}{j!} \delta_{ir} \delta_{jk} \delta_{ls},
\end{equation}
and hence
\begin{equation}
    \label{eq:FunnyConstantMatrixFactorial}
    \hat{C}^{(r, s)}_{(i, j), \boldsymbol{\cdot}}
    =
    \frac{1}{j!} \delta_{ir}
    \quad
    \textrm{and}
    \quad
    \hat{C}^{(r, s)}_{\boldsymbol{\cdot}, (k, l)}
    =
    \frac{1}{k!} \delta_{ls}.
\end{equation}
The recursive definition of the seminorms is now changed into
\begin{equation}
    \label{eq:hmlMatrixSecondVersion}
    \hat{h}_{m+1, 2\ell, (r, s)}(A)
    =
    \sum_{j = 1}^\infty
    \frac{1}{j!} \hat{h}_{m, \ell, (r, j)}(A)^2 
    \quad
    \textrm{and}
    \quad
    \hat{h}_{m+1, 2\ell+1, (r, s)}(A)
    =
    \sum_{k = 1}^\infty
    \frac{1}{k!} \hat{h}_{m, \ell, (k, s)}(A)^2
\end{equation}
with $\hat{h}_{0, 0, (i, j)}(A) = |\hat{A}_{ij}| = \sqrt{i!j!}
|A_{ij}|$ as starting point. As usual $\norm{A}^{\hat{\;}}_{m, \ell,
  (r, s)} = \sqrt[2^m]{\hat{h}_{m, \ell, (r, s)}(A)}$. Note that for
even $\ell$, the seminorms do not depend on $s$ while for odd $\ell$
they do not depend on $r$. Thus there is a certain redundancy.  One
can show that the recursion will work for this choice of the basis.
\begin{proposition}
    \label{proposition:MatricesFirstVersion}%
    Let $M_\infty(\mathbb{C})$ be endowed with the seminorms
    $\norm{\argument}^{\hat{\;}}_{m, \ell, (r, s)}$ for $r, s \in
    \mathbb{N}$.
    \begin{propositionlist}
    \item \label{item:MatricesFirstWorks} For all $A \in
        M_\infty(\mathbb{C})$ we have $\norm{A}^{\hat{\;}}_{m, \ell,
          (r, s)} < \infty$.
    \item \label{item:MatricesFirstContainsStuff} The completion of
        $M_\infty(\mathbb{C})$ to a Fréchet algebra contains at least
        those $A$ with coefficients $\hat{A}_{rs}$ having
        sub-factorial growth with respect to $r+s$.
    \item \label{item:TraceContinuous} The trace functional $\tr\colon
        M_\infty(\mathbb{C}) \longrightarrow \mathbb{C}$ extends to a
        continuous linear trace functional on the completion.
    \end{propositionlist}
\end{proposition}
\begin{proof}
    Let $\hat{A}_{rs}$ have sub-factorial growth, i.e.\ for all
    $\epsilon > 0$ there is a constant $c > 0$ with $|\hat{A}_{rs}|
    \le c ((r+s)!)^\epsilon$. Equivalently, we can replace $(r+s)!$
    also by $\max(r, s)!$ or by $r!s!$. We claim that the recursion
    yields $\hat{h}_{m, \ell, (r, s)}(A)$ still having sub-factorial
    growth for all $m, \ell$. In particular, $\norm{A}^{\hat{\;}}_{m,
      \ell, (r, s)} < \infty$. Indeed, this is a simple induction
    following directly from \eqref{eq:hmlMatrixSecondVersion}. This
    proves the first and second part. For the third, we note that a
    straightforward application of Hölder's inequality gives
    \begin{align*}
        |\tr(A)|
        &\le \sum_{r=1}^\infty \frac{|\hat{A}_{rr}|}{r!} \\
        &\le
        \left(
            \sum_{r=1}^\infty \frac{|\hat{A}_{rr}|^4}{(r!)^3}
        \right)^{\frac{1}{4}}
        \left(
            \sum_{r=1}^\infty \frac{1}{(r!)^{\frac{1}{3}}}
        \right)^{\frac{3}{4}} \\
        &\le
        c
        \left(
            \sum_{s, r = 1}^\infty
            \frac{|\hat{A}_{sr}|^4}{s!(r!)^2}
        \right)^{\frac{1}{4}} \\
        &\le
        c
        \left(
            \sum_{s=1}^\infty
            \frac{1}{s!}
            \left(
                \sum_{r=1}^\infty
                \frac{|\hat{A}_{sr}|^2}{r!}
            \right)^2
        \right)^{\frac{1}{4}} \\
        &=
        c \norm{A}^{\hat{\;}}_{2, 1 , (i, j)},
    \end{align*}
    where $c$ is the numerical constant coming from the second series
    in the second step and $(i, j)$ are arbitrary. This proves the
    last part.
\end{proof}

Alternatively, we can make use of a polynomial rescaling. We consider
the basis of $M_\infty(\mathbb{C})$ defined by
\begin{equation}
    \label{eq:SecondMatrixBasis}
    \tilde{E}_{ij} = \frac{1}{ij} E_{ij}
\end{equation}
leading to the structure constants $\tilde{C}_{(i,j), (k, l)}^{(r, s)}
= \frac{1}{j^2} \delta_{ir}\delta_{jk}\delta_{ls}$ and hence
\begin{equation}
    \label{eq:PolyMatrixConstants}
    \tilde{C}_{(i, j), \boldsymbol{\cdot}}^{(r, s)}
    =
    \frac{1}{j^2} \delta_{ir}
    \quad
    \textrm{and}
    \quad
    \tilde{C}_{\boldsymbol{\cdot}, (k, l)}^{(r, s)}
    =
    \frac{1}{k^2} \delta_{ls}.    
\end{equation}
Accordingly, the recursion for the seminorms is now given by
\begin{equation}
    \label{eq:MatrixRecursionSecond}
    \tilde{h}_{m + 1, 2\ell, (r, s)}(A)
    =
    \sum_{j=1}^\infty \frac{1}{j^2} \tilde{h}_{m, \ell, (r, j)}(A)^2
    \quad
    \textrm{and}
    \quad
    \tilde{h}_{m + 1, 2\ell + 1, (r, s)}(A)
    =
    \sum_{k=1}^\infty \frac{1}{k^2} \tilde{h}_{m, \ell, (k, s)}(A)^2
\end{equation}
with starting point $\tilde{h}_{0, 0, (i, j)}(A) = |\tilde{A}_{ij}| =
ij |A_{ij}|$.
\begin{proposition}
    \label{proposition:MatricesSecondVersion}
    Let $M_\infty(\mathbb{C})$ be endowed with the seminorms
    $\norm{\argument}^{\tilde{\;}}_{m, \ell, (r, s)}$ for $r, s \in
    \mathbb{N}$.
    \begin{propositionlist}
    \item \label{item:SecondMatrixWorks} For all $A \in
        M_\infty(\mathbb{C})$ we have $\norm{A}^{\tilde{\;}}_{m, \ell,
          (r, s)} < \infty$.
    \item \label{item:CompletionBigger} The completion of
        $M_\infty(\mathbb{C})$ to a Fréchet algebra contains at least
        those $A$ with coefficients $\tilde{A}_{rs}$ being bounded.
    \item \label{item:Trace} The trace functional is not continuous
        and the completion contains matrices not being trace-class.
    \end{propositionlist}
\end{proposition}
\begin{proof}
    The convergence of the series $\sum_{n=1}^\infty \frac{1}{n^2}$
    allows to show inductively that for bounded $|\tilde{A}_{rs}|$
    also all the quantities $h_{m, \ell, (r, s)}(A)$ stay bounded with
    respect to $(r, s)$. This proves the first and second part. For
    the third, we consider the matrix $A$ with coefficients
    $\tilde{A}_{rs} = r\delta_{rs}$. We claim that also this matrix is
    in the completion (though it does not have bounded
    coefficients). Indeed, we have $\tilde{h}_{1, 0, (r, s)}(A) =
    \frac{1}{r}$ and $\tilde{h}_{1, 1, (r, s)}(A) = \frac{1}{s}$ which
    is now bounded with respect to $(r, s)$. Hence by the induction
    from the second part, all higher $\tilde{h}_{m, \ell, (r, s)}(A)$
    will be bounded as well. Clearly, $A$ is not trace-class. However,
    the series $A = \sum_{r, s = 1}^\infty r \delta_{rs}
    \tilde{E}_{rs} = \sum_{r=1}^\infty \frac{1}{r} E_{rr}$ converges
    according to the general fact that the $\tilde{E}_{rs}$ form a
    Schauder basis. Thus $\tr$ cannot be continuous.
\end{proof}
\begin{remark}
    \label{remark:MatrizenZweiVerschiedeneVersionen}%
    Here we see that the two rescalings yield two \emph{different}
    Fréchet algebras. Furthermore, one can show that the topology of
    the first version is strictly finer than the well-known topology
    of fast decreasing matrices (the Fréchet space being the Schwarz
    space), which then is strictly finer than the topology of the
    second version. More details on this example and a further
    discussion can be found in \cite[Sect.~4.3]{beiser:2011a}.
\end{remark}

%
%

\subsection{Group algebras of finitely generated groups}
\label{subsec:GroupAlgebrasFinitelyGeneratedGroups}

Generalizing the example of the Laurent polynomials, the fourth
example is based on a countable but infinite group $G$ and the
corresponding group algebra $\mathcal{A} = \mathbb{C}[G]$. For $G =
\mathbb{Z}$ we are back at the second example.

Since we have already seen in the case $G = \mathbb{Z}$ that a
rescaling of the canonical basis might be necessary, we immediately
start by considering the basis $\basis{e}_g = \frac{1}{\mathsf{c}(g)}
g$ for $g \in G$ where $\mathsf{c}(g) > 0$ will be a rescaling. We
will require $\mathsf{c}(g) = \mathsf{c}(g^{-1})$ for all $g \in G$ as
well as $\mathsf{c}(e) = 1$.  The structure constants are now given by
\begin{equation}
    \label{eq:StructureConstantsGroupCase}
    \basis{e}_g \basis{e}_h
    =
    \sum_{k \in G} C^k_{g, h} \basis{e}_k
    \quad
    \textrm{with}
    \quad
    C^k_{g, h}
    =
    \frac{\mathsf{c}(k)}{\mathsf{c}(g)\mathsf{c}(h)} \delta_{gh, k}.
\end{equation}
The constants \eqref{eq:CSeriesConverge} needed for the construction
of the topology on $\mathbb{C}[G]$ are given by
\begin{equation}
    \label{eq:ConstantsForGroupCase}
    C^k_{g, \boldsymbol{\cdot}}
    =
    \frac{\mathsf{c}(k)}{\mathsf{c}(g)\mathsf{c}(g^{-1}k)}
    \quad
    \textrm{and}
    \quad
    C^k_{\boldsymbol{\cdot}, h}
    =
    \frac{\mathsf{c}(k)}{\mathsf{c}(kh^{-1})\mathsf{c}(h)}
    =
    \frac{\mathsf{c}(k^{-1})}{\mathsf{c}(hk^{-1})\mathsf{c}(h^{-1})}
    =
    C^{k^{-1}}_{h^{-1}, \boldsymbol{\cdot}}
\end{equation}
according to our convention $\mathsf{c}(g) =
\mathsf{c}(g^{-1})$. Writing $a = \sum_{g \in G} a_g \basis{e}_g$ we
have $\norm{a}_{0, 0, g} = |a_g|$ as usual. For the higher seminorms
we have the recursion
\begin{equation}
    \label{eq:hmPlusEinsGeradeGroup}
    h_{m+1, 2\ell, k}(a)
    =
    \sum_{g \in G} h_{m, \ell, g}(a)^2
    \frac{\mathsf{c}(k)}{\mathsf{c}(g)\mathsf{c}(g^{-1}k)}
\end{equation}    
and by \eqref{eq:ConstantsForGroupCase}
\begin{equation}
    \label{eq:hmPlusEinsUngeradeGroup}
    h_{m+1, 2\ell + 1, k^{-1}}(a)
    =
    \sum_{g \in G} h_{m, \ell, g}(a)^2
    C^{k^{-1}}_{\boldsymbol{\cdot}, g}
    =
    \sum_{g \in G} h_{m, \ell, g}(a)^2 C^k_{g, \boldsymbol{\cdot}}
    =
    h_{m+1, 2\ell, k}(a).
\end{equation}

While it is tempting to choose $\mathsf{c}(g) = 1$ for all $g$, the
following lemma shows that this choice will \emph{not} lead to
$\mathcal{A}_{\mathrm{nice}} = \mathcal{A}$.
\begin{lemma}
    \label{lemma:NecessaryCondGroup}%
    A necessary condition for $\mathcal{A}_{\mathrm{nice}} =
    \mathcal{A}$ is $\frac{1}{\mathsf{c}} \in \ell^2(G)$.
\end{lemma}
\begin{proof}
    First we note that $h_{1, 0, k}(a) \ge |a_g|^2
    \frac{\mathsf{c}(k)}{\mathsf{c}(g)\mathsf{c}(g^{-1}k)}$ for all
    $g$. Hence for $m = 2$ we get
    \[
    h_{2, 0, h}(a)
    \ge
    \sum_{k \in G}
    |a_g|^4
    \frac{\mathsf{c}(k)^2}{\mathsf{c}(g)^2\mathsf{c}(g^{-1}k)^2}
    \frac{\mathsf{c}(h)}{\mathsf{c}(k)\mathsf{c}(k^{-1}h)}
    =
    |a_g|^4
    \frac{\mathsf{c}(h)}{\mathsf{c}(g)^2}
    \sum_{k \in G}
    \frac{\mathsf{c}(k)}{\mathsf{c}(g^{-1}k)^2\mathsf{c}(k^{-1}h)}.
    \]
    Thus a necessary condition for $h_{2, 0, h}(a) < \infty$ for all
    $a \in \mathcal{A}$ is the convergence
    \[
    \sum_{k \in G}
    \frac{\mathsf{c}(k)}{\mathsf{c}(g^{-1}k)^2\mathsf{c}(k^{-1}h)}
    < \infty
    \]
    for all $g, h \in G$. Taking $g = h = e$ gives the summability of
    $\frac{1}{\mathsf{c}^2}$ as claimed.
\end{proof}

However, this will be just a necessary condition, the higher seminorms
will yield more conditions in general. While it is not yet clear
whether we can actually find a suitable $\mathsf{c}$, the next lemma
shows that the growth of the coefficients will be limited, in the same
spirit as we have seen that for the polynomials and the Laurent
polynomials:
\begin{lemma}
    \label{lemma:GrowthCoefficientsGroupCase}%
    For every $m \ge 0$ and $a \in \mathcal{A}$ one has
    \begin{equation}
        \label{eq:GrowthLimitGroupCase}
        |a_g| \le \sqrt[2^m]{\mathsf{c}(g)} \norm{a}_{m+1, 0, e}.
    \end{equation}
\end{lemma}
\begin{proof}
    We claim that $h_{m, 0, k}(a) \ge |a_g|^{2^m}
    \frac{\mathsf{c}(k)}{\mathsf{c}(g)\mathsf{c}(g^{-1}k)}$ which
    gives \eqref{eq:GrowthLimitGroupCase} for $k = e$.  Indeed, the
    case $m = 0$ was already obtained in the proof of
    Lemma~\ref{lemma:NecessaryCondGroup}. By induction on $m$ we have
    \[
    h_{m+1, 0, k}(a)
    \ge
    \sum_{h \in G}
    |a_g|^{2^{m+1}}
    \frac{\mathsf{c}(h)^2}{\mathsf{c}(g)^2\mathsf{c}(g^{-1}h)^2}
    \frac{\mathsf{c}(k)}{\mathsf{c}(h)\mathsf{c}(h^{-1}k)}
    \ge
    |a_g|^{2^{m+1}}
    \frac{\mathsf{c}(k)}{\mathsf{c}(g)\mathsf{c}(g^{-1}k)}.
    \]    
\end{proof}
\begin{lemma}
    \label{lemma:InvolutionAntipodeContinuous}%
    For every rescaling with $\mathcal{A}_{\mathrm{nice}} =
    \mathcal{A}$ the canonical $^*$-involution and the canonical
    antipode $S$ of $\mathbb{C}[G]$ are continuous.
\end{lemma}
\begin{proof}
    Recall that by definition $\basis{e}_g^* = \basis{e}_{g^{-1}}$
    since we assume $\mathsf{c}(g) = \mathsf{c}(g^{-1})$. Hence
    $(a^*)_g = \cc{a_{g^{-1}}}$. Thus $\norm{a^*}_{0, g} =
    \norm{a}_{0, g^{-1}}$ for all $g \in G$. For the higher seminorms
    we have to take care of even and odd $\ell$: the effect of $^*$ is
    to flip between even and odd $\ell$ as well as between $k$ and
    $k^{-1}$.  The details require some notational effort using the
    binary expansion of $\ell$ which we leave out. The argument for
    the antipode $S(a) = \sum_{g \in G} a_g \basis{e}_{g^{-1}}$ is
    identical.
\end{proof}

After these generalities we are now looking for a rescaling
$\mathsf{c}$ such that we can guarantee $\mathcal{A}_{\mathrm{nice}} =
\mathcal{A}$. The idea is to look also for a class of functions
$\mathcal{F} \subseteq \mathrm{Fun}(G, [0,\infty))$ subject to the
following conditions:
\begin{enumerate}
\item \label{item:FGrowthCond} For every $\chi \in \mathcal{F}$ there
    is another $\psi \in \mathcal{F}$ and $c > 0$ such that for all $h
    \in G$
    \begin{equation}
        \label{eq:GrowthConditionF}
        \sum_{g \in G} 
        \frac{\chi(g)^2}
        {\mathsf{c}(g) \mathsf{c}(g^{-1}h)} 
        \le
        c \: \frac{\psi(h)}{\mathsf{c}(h)}.
    \end{equation}
\item \label{item:FhatEins} There is a $\chi \in \mathcal{F}$ with
    $\chi(g) > 0$ for all $g \in G$.
\item \label{item:FClosedUnderInverses} For $\chi \in \mathcal{F}$
    also $\chi^{-1} \in \mathcal{F}$ where $\chi^{-1}(g) =
    \chi(g^{-1})$.
\end{enumerate}
Suppose we have such a matching pair $(\mathsf{c}, \mathcal{F})$ then
one has the following statement:
\begin{lemma}
    \label{lemma:MatchingPairGroupAlgebra}%
    Suppose $\mathsf{c}$ is a rescaling and $\mathcal{F}$ is a class
    of functions such that the above compatibility holds. Then for
    $\mathcal{A} = \mathbb{C}[G]$ we have:
    \begin{lemmalist}
    \item \label{item:AniceGroupAlgebra} $\mathcal{A}_{\mathrm{nice}}
        = \mathcal{A}$.
    \item \label{item:CompletionLarge} If $a \in \mathbb{C}[[G]]$ is a
        formal series $a = \sum_{g \in G} a_g \basis{e}_g$ such that
        there is a $\chi \in \mathcal{F}$ with $|a_g| \le c \chi(g)$
        for some $c > 0$ then $a \in \complete{\mathcal{A}}$.
    \end{lemmalist}
\end{lemma}
\begin{proof}
    Suppose that $a \in \mathbb{C}[[G]]$ satisfies the estimate from
    the second part. We claim that for all $m$ and $\ell$ there exists
    a $c > 0$ and a $\chi \in \mathcal{F}$ (depending on $m$ and
    $\ell$) such that $h_{m, \ell, k}(a) \le c \chi(k)$ for all $k \in
    G$. Indeed, for $m = 0$ this is the assumption about $a$. Then by
    induction
    \[
    h_{m+1, 2\ell, k}(a)
    =
    \sum_{g \in G} h_{m, \ell, g}(a)^2
    \frac{\mathsf{c}(k)}{\mathsf{c}(g)\mathsf{c}(g^{-1}k)}
    \le
    \sum_{g \in G} c^2 \chi(g)^2
    \frac{\mathsf{c}(k)}{\mathsf{c}(g)\mathsf{c}(g^{-1}k)}
    \le
    c^2 c' \psi(k),
    \]
    where $\psi$ and $c'$ are chosen according to
    \eqref{eq:GrowthConditionF}. For $2\ell +1$ instead of $2\ell$ we
    can proceed analogously using the property $\chi^{-1} \in
    \mathcal{F}$ whenever $\chi \in \mathcal{F}$. This establishes the
    claim.  In particular, all the $h_{m, \ell, k}(a)$ are
    \emph{finite} for such a $a \in \mathbb{C}[[G]]$. Since we have at
    least one nontrivial $\chi \in \mathcal{F}$ with $\chi(g) > 0$
    for all $g \in G$, it follows that the finite sums $a \in
    \mathbb{C}[G]$ can clearly be estimated by such a $\chi$. Hence
    $\mathcal{A}_{\mathrm{nice}} = \mathcal{A}$ follows, proving the
    first part. The second was already shown on the way.
\end{proof}

In general it is not clear whether we can find such a rescaling
$\mathsf{c}$ and a sufficiently interesting class of functions
$\mathcal{F}$ meeting the above criteria. However, if we make the
additional assumption that $G$ is \emph{finitely generated} then we
can construct such a matching pair.  We choose a set of generators
$g_1, \ldots, g_N \in G$. This allows to define the word length
functional
\begin{equation}
    \label{eq:Length}
    \length\colon G \longrightarrow \mathbb{N}_0
\end{equation}
in the usual way, i.e. $\length(g)$ is the minimum of the number of
generators and their inverses needed to obtain $g$ as a product of
generators. By convention we set $\length(e) = 0$. The crucial
properties of this functional are now
\begin{equation}
    \label{eq:LengthProperties}
    \length(gh) \le \length(g) + \length(h),
    \quad
    \length(e) = 0,
    \quad
    \textrm{and}
    \quad
    \length(g^{-1}) = \length(g)
\end{equation}
for all $g, h \in G$.  As usual, $\length$ depends of course on the
choice of the generators.
\begin{proposition}
    \label{proposition:GroupAlgebra}%
    Assume $G$ is a finitely generated group with word length
    functional as above and fix $\epsilon > 0$.  Let
    $\mathsf{c}_\epsilon(g) = (\length(g)!)^\epsilon$ and
    \begin{equation}
        \label{eq:ClassFForGroupAlgebra}
        \mathcal{F} = \left\{
            \chi_R\colon g \mapsto 
            \chi_R(g) = R^{\length(g)}
            \; \Big| \;
            R > 0
        \right\}.
    \end{equation}
    Then the pair $\mathsf{c}_\epsilon$ and $\mathcal{F}$ satisfies
    the conditions \refitem{item:FGrowthCond},
    \refitem{item:FhatEins}, and \refitem{item:FClosedUnderInverses}
    needed for Lemma~\ref{lemma:MatchingPairGroupAlgebra}.
\end{proposition}
\begin{proof}
    The second and third condition is trivial. It remains to check the
    first. Let $R > 0$ be given. For a fixed $\ell$ we have at most
    $(2N)^\ell$ group elements $g$ with $\length(g) = \ell$, where
    $2N$ is the number of generators and their inverses.  Indeed, the
    bound is exceeded for the free group in $N$ letters and any other
    finitely generated group has possibly less thanks to possible
    relations. From this we see that
    \[
    \sum_{g \in G} \frac{R^{\length(g)}}{(\length(g)!)^\epsilon}
    =
    \sum_{\ell = 0}^\infty \sum_{\length(g) = \ell}
    \frac{R^\ell}{(\ell!)^\epsilon}
    \le
    \sum_{\ell = 0}^\infty \frac{(2N)^\ell R^\ell}{(\ell!)^\epsilon}
    \]
    converges.  Hence $\lexp_\epsilon(z) = \sum_{g \in G}
    \frac{z^{\length(g)}}{(\length(g)!)^\epsilon}$ defines an entire
    function for $z \in \mathbb{C}$.  For $h \in G$ and $R > 0$ we
    have
    \begin{align*}
        &\sum_{g \in G} \left(R^{\length(g)}\right)^2
        \left(
            \frac{\length(h)!}{\length(g)!\length(g^{-1}h)!}
        \right)^\epsilon \\
        &=
        \left(
            \sum_{\length(g) \ge \length(h)}
            +
            \sum_{\substack{
                \length(g) < \length(h) \\
                \length(g^{-1}h) \ge \length(h)}
            }
            +
            \sum_{\substack{
                \length(g) < \length(h) \\
                \length(g^{-1}h) < \length(h)}
            }
        \right)
        R^{2\length(g)}
        \left(
            \frac{\length(h)!}{\length(g)!\length(g^{-1}h)!}
        \right)^\epsilon \\
        &\stackrel{(a)}{\le}
        \sum_{\length(g) \ge \length(h)}
        \frac{R^{2\length(g)}}{(\length(g^{-1}h)!)^\epsilon}
        +
         \sum_{\substack{
            \length(g) < \length(h) \\
            \length(g^{-1}h) \ge \length(h)}
        }
        \frac{R^{2\length(g)}}{(\length(g)!)^\epsilon}
        +
        \sum_{\substack{
            \length(g) < \length(h) \\
            \length(g^{-1}h) < \length(h)}
        }
        R^{2\length(g)}
        \left(
            \frac{\left(\length(g) + \length(g^{-1}h)\right)!}
            {\length(g)! \length(g^{-1}h)!}
        \right)^\epsilon
        \\
        &\le
        \sum_{k \in G}
        \frac{R^{2\length(k^{-1}h)}}{(\length(k)!)^\epsilon}
        +
        \sum_{g \in G}
        \frac{R^{2\length(g)}}{(\length(g)!)^\epsilon}
        +
        \sum_{\substack{
            \length(g) < \length(h) \\
            \length(g^{-1}h) < \length(h)}
        }
        R^{2\length(g)}
        \binom{\length(g) + \length(g^{-1}h)}{\length(g)}^\epsilon
        \\
        &\stackrel{(b)}{\le}
        \sum_{k \in G}
        \frac{(1 + R)^{2\length(k^{-1}h)}}{(\length(k)!)^\epsilon}
        +
        \lexp_\epsilon(R^2)
        +
        \sum_{\substack{
            \length(g) < \length(h) \\
            \length(g^{-1}h) < \length(h)}
        }
        (1 + R)^{2\length(h)}
        2^{2\epsilon\length(h)}
        \\
        &\stackrel{\mathclap{\eqref{eq:LengthProperties}}}{\le}
        \quad
        (1 + R)^{2\length(h)}
        \lexp_\epsilon\left((1+R)^2\right)
        +
        \lexp_\epsilon(R^2)
        +
        (2N)^{\length(h)}
        (1 + R)^{2\length(h)}
        2^{2\epsilon\length(h)},
    \end{align*}
    where in $(a)$ we use $\length(h) \le \length(g) +
    \length(g^{-1}h)$, see \eqref{eq:LengthProperties}, and in $(b)$
    we use that (rough) estimate that $\binom{n+m}{m} \le 2^{n+m}$ for
    $n = \length(g)$ and $m = \length(g^{-1}h)$ both being $\le
    \length(h)$. Thus the binomial coefficient is $\le
    2^{2\length(h)}$. But the last estimate shows that we have again
    an at most exponential growth in $\length(h)$, which establishes
    the second property.
\end{proof}

Thanks to the proposition, there always exists a good choice (in fact
many) of a rescaling $\mathsf{c}$ for a finitely generated group.  For
the fixed choice of generators we write $\mathcal{A}_\epsilon$ for
$\mathbb{C}[G]$ equipped with the above topology. The corresponding
Fréchet $^*$-algebra $\complete{\mathcal{A}_\epsilon}$ has several
nice properties.
\begin{proposition}
    \label{proposition:GroupAlgebraFurtherProperties}%
    Let $G$ be a finitely generated infinite group with chosen word
    length functional $\length$ and let $\epsilon > 0$. Then for the
    corresponding Fréchet algebra $\complete{\mathcal{A}_\epsilon}$
    one has:
    \begin{propositionlist}
    \item \label{item:AepsilonInellEins}
        $\complete{\mathcal{A}_\epsilon}$ is continuously included
        into $\ell^1(G)$.
    \item \label{item:CharactersContinuous} Every complex-valued group
        character $\chi$ of $G$ extends continuously to an algebra
        character $\chi$ on $\complete{\mathcal{A}_\epsilon}$.
    \end{propositionlist}
\end{proposition}
\begin{proof}
    Let $a \in \complete{\mathcal{A}_\epsilon}$. Then
    Lemma~\ref{lemma:GrowthCoefficientsGroupCase} yields
    \[
    \norm{a}_{\ell^1(G)}
    =
    \sum_{g \in G}
    \frac{|a_g|}{(\length(g)!)^\epsilon}
    \le
    \norm{a}_{m+1, 0, e}
    \sum_{g \in G}
    \frac{\sqrt[2^m]{(\length(g)!)^\epsilon}}
    {(\length(g)!)^\epsilon}
    =
    c_m \norm{a}_{m+1, 0, e}
    \]
    with a finite $c_m > 0$ for $m \ge 1$.  This shows the first
    part. For the second, we note that for a character $\chi$ we have
    the estimate $|\chi(g)| \le R^{\length(g)}$ where $R$ is the
    maximum of $|\chi|$ on the generators. But then
    \[
    |\chi(a)|
    \le
    \sum_{g \in G} \frac{|a_g|}{(\length(g)!)^\epsilon} |\chi(g)|
    \le
    \norm{a}_{m+1, 0, e}
    \sum_{g \in G}
    \frac{(\length(g)!)^{\frac{\epsilon}{2^m}} R^{\length(g)}}
    {(\length(g))!}
    =
    c_{m, R} \norm{a}_{m+1, 0, e}
    \]
    with a finite constant $c_{m, R} > 0$ for $m \ge 1$. This shows
    the continuity of $\chi$.
\end{proof}
The image of $\complete{\mathcal{A}_\epsilon}$ in $\ell^1(G)$ is a
dense subalgebra which is different from $\ell^1(G)$ since we have the
nontrivial growth conditions \eqref{eq:GrowthLimitGroupCase} not
necessary for an arbitrary element of $\ell^1(G)$. Hence the Fréchet
topology of $\complete{\mathcal{A}_\epsilon}$ is strictly finer than
the Banach topology of $\ell^1(G)$.
\begin{remark}
    \label{remark:GroupAlgebra}%
    As our primary interest is a different one, we leave now the group
    algebra case with several open questions: first, one would like to
    know how sensitive the topology of $\mathcal{A}_\epsilon$ depends
    on the choice of the generators as well as on the parameter
    $\epsilon$. At the present state it seems to be rather difficult
    to approach this question. In particular, it would be nice to see
    if one can prove a similar statement as in the Laurent series
    case: the completion consists of those formal series with
    sub-factorial growth with respect to the length
    functional. Second, in view of
    Lemma~\ref{lemma:InvolutionAntipodeContinuous}, it would be
    interesting to find a reasonable topology on $\mathcal{A}_\epsilon
    \tensor \mathcal{A}_\epsilon $ such that also the canonical
    coproduct determined by $\Delta(g) = g \tensor g$ is continuous,
    leading to a locally convex Hopf algebra structure.
\end{remark}

%
%

\subsection{The Wick star product on $\mathbb{C}^n$}
\label{subsec:WickStarProductFlatCase}

Recall that the Wick star product on $\mathbb{C}^n$ is explicitly
given by
\begin{equation}
    \label{eq:WickStarProduct}
    f \starwick g
    =
    \sum_{N = 0}^\infty \frac{(2\hbar)^{|N|}}{N!}
    \frac{\partial^{|N|} f}{\partial z^N}
    \frac{\partial^{|N|} g}{\partial \cc{z}^N},
\end{equation}
where $z^1, \ldots, z^n$ are the usual complex coordinates on
$\mathbb{C}^n$ and we make use of the standard multiindex notation: we
use a multiindex $N = (N_1, \ldots, N_n) \in \mathbb{N}_0^n$ and set
$N! = N_1! \cdots N_n!$ as well as $|N| = N_1 + \cdots + N_n$. Sums,
min and max, and partial derivatives and powers of $z$ and $\cc{z}$
for multiindices are defined componentwise as usual.

Usually, the Wick star product is treated as a \emph{formal} star
product with $\hbar$ being a formal parameter and $f, g \in
C^\infty(\mathbb{C}^n)[[\hbar]]$, see
e.g. \cite[Sect.~5.2.4]{waldmann:2007a} or
\cite{bordemann.waldmann:1997a} for a detailed exposition. However, it
is clear from the definition that the Wick star product makes perfect
sense for polynomials $f$ and $g$ and $\hbar$ an arbitrary complex
number. For physical reasons, we choose $\hbar > 0$. Then
$\mathbb{C}[z, \cc{z}]$ endowed with the Wick star product is just the
usual Weyl algebra since we have the canonical commutation relations
$[z^k, \cc{z}^\ell]_{\starwick} = 2\hbar \delta_{k\ell}$. In fact,
using the complex conjugation as $^*$-involution we get the more
familiar Hermitian generators $q^k = \frac{1}{2}(z^k + \cc{z}^k)$ and
$p_k = \frac{1}{2\I}(z^k - \cc{z}^k)$.

The monomials in $z$ and $\cc{z}$ constitute a vector space basis of
all the polynomials which we shall use now to implement our general
construction. In fact, we use the rescaled monomials
\begin{equation}
    \label{eq:eIJWick}
    \basis{e}_{I, J} = \frac{z^I \cc{z}^J}{I!J!(2\hbar)^{|I| + |J|}}.
\end{equation}

We will not carry out the details of the computation but the procedure
is clear: one has to compute the structure constants of $\starwick$
with respect to this basis and set up the recursion for the seminorms
$\norm{\argument}_{m, \ell, (I, J)}$ as defined in
Section~\ref{subsec:FirstVersion}. One can check that the seminorms
are indeed all finite on polynomials and hence we get a locally convex
algebra structure on $\mathbb{C}[z, \cc{z}]$ which we shall denote by
$\mathcal{A}_{\scriptscriptstyle{\mathrm{Wick}}}$.

Moreover, we can require that an evaluation functional $\delta_p$ for
some $p \in \mathbb{C}^n$ is continuous and hence implement the second
version of our general construction as in
Section~\ref{subsec:SecondVersion}. Again, one shows that this works
for any given $p$ and yields a locally convex algebra structure on
$\mathbb{C}[z, \cc{z}]$ which we shall denote by $\mathcal{A}_p$.

Finally, in \cite{beiser.roemer.waldmann:2007a} a more ad-hoc
recursive construction of a locally convex topology for the Weyl
algebra $(\mathbb{C}[z, \cc{z}], \starwick)$ was given, depending also
on the choice of a point $p$. We denote the resulting locally convex
algebra by $\mathcal{A}_{\hbar, p}$. In fact, this work was the
starting point for our more general construction. Even though the
precise recursion is not the same as the one from our general
construction, one can show the following result
\cite[Sect.~4.5]{beiser:2011a}:
\begin{theorem}
    \label{theorem:AllWickStuffTheSame}%
    For all $p \in \mathbb{C}^n$ the locally convex algebras
    $\mathcal{A}_{\scriptscriptstyle{\mathrm{Wick}}}$, $\mathcal{A}_p$
    and $\mathcal{A}_{h, p}$ are the same. The completion yields a
    Fréchet algebra with unconditional Schauder basis given by the
    $\basis{e}_{I, J}$.
\end{theorem}
In \cite{beiser.roemer.waldmann:2007a} many additional properties of
the resulting Fréchet algebra have been studied. In particular, the
completion contains also the exponentials of $z$ and $\cc{z}$ and
hence the generators of the $C^*$-algebraic version of the Weyl
algebra. Moreover, the Heisenberg group acts on
$\complete{\mathcal{A}}_{\scriptscriptstyle{\mathrm{Wick}}}$ by inner
$^*$-automorphisms. Finally, every evaluation functional is a
continuous positive functional and the corresponding GNS
representation gives the usual representation on the Bargmann-Fock
space.
\begin{remark}[Weyl algebra]
    \label{remark:WeylAlgebra}%
    The Weyl algebra has of course been studied by many people for
    many different aspects, e.g.\ in the work of Cuntz
    \cite{cuntz:2005a} the finest locally convex topology on the Weyl
    algebra was used to define and study its bivariant
    $K$-theory. Note however that this topology is of course very far
    from being Fréchet.
\end{remark}

We can even determine the Fréchet space
$\complete{\mathcal{A}}_{\scriptscriptstyle{\mathrm{Wick}}}$
explicitly. Let $a \in
\complete{\mathcal{A}}_{\scriptscriptstyle{\mathrm{Wick}}}$ be given
as Taylor series
\begin{equation}
    \label{eq:TaylorInWickAlgebra}
    a
    = \sum_{I, J} a_{I, J} \basis{e}_{I, J}
    = \sum_{I, J}
    \frac{a_{I, J}}{I!J!(2\hbar)^{|I| + |J|}} z^I \cc{z}^J.
\end{equation}
In the notation of \cite{beiser.roemer.waldmann:2007a} we have an
estimate of the Taylor coefficient $a_{I, J}$ in terms of the zeroth
seminorms of the form
\begin{equation}
    \label{eq:EstimateTaylorWickErste}
    |a_{I, J}| \le  \frac{1}{\sqrt{(2\hbar)^{|I| + |J|}}}
    \norm{a}_{0, 0, I, J}^{0, \hbar},
\end{equation}
see \cite[Def.~3.1]{beiser.roemer.waldmann:2007a}. From that recursive
definition of the quantities $h_{m, \ell, I, J}^{0, \hbar}$ one
immediately sees that we have an estimate of the form $h^{0,
  \hbar}_{m, \ell, I, J}(a) \ge h^{0, \hbar}_{m-1, \ell/2, I, J}(a)^2$
for even $\ell$ and $h^{0, \hbar}_{m, \ell, I, J} \ge h^{0,
  \hbar}_{m-1, \ell/2, J, I}(a)^2$ for odd $\ell$. Thus, taking $\ell
= 0$ for simplicity, we arrive at the estimate $h^{0, \hbar}_{0, 0, I,
  J}(a) \le \sqrt[2^m]{h^{0, \hbar}_{m, 0, I, J}(a)}$ resulting in the
estimate
\begin{equation}
    \label{eq:TaylorWickNextEstimate}
    |a_{I, J}|
    \le \sqrt[2^{m+1}]{h^{0, \hbar}_{m, 0, I, J}(a)}
    = \norm{a}^{0, \hbar}_{m, 0, I, J}
    \le
    \sqrt[2^{m+3}]{I!} \sqrt[2^{m+2}]{J!}
    \norm{a}^{0, \hbar}_{m+2, 1, 0, 0},
\end{equation}
where for the last we used \cite[Prop.~3.3,
6.)]{beiser.roemer.waldmann:2007a}. Thus the Taylor coefficients have
sub-factorial growth with respect to $|I| + |J|$ or $\max(|I|, |J|)$,
respectively. Conversely, and more simple, one can use the explicit
recursive definition of the seminorms to see that sub-factorial growth
with respect to $|I| + |J|$ is also sufficient for a formal Taylor
series to belong to the completion. This is done analogously to
\cite[Thm.~3.6, 4.)]{beiser.roemer.waldmann:2007a}, where an
exponential growth was shown to be sufficient. More details on the
algebra $\complete{\mathcal{A}}_{\scriptscriptstyle{\mathrm{Wick}}}$
can be found in \cite[Sect.~4.5]{beiser:2011a}. We summarize this as
follows:
\begin{theorem}
    \label{theorem:WickAlgebraSubFac}%
    The completion
    $\complete{\mathcal{A}}_{\scriptscriptstyle{\mathrm{Wick}}}$ is a
    strongly nuclear Fréchet algebra with absolute Schauder basis
    given by the Taylor monomials $\{\basis{e}_{I, J}\}_{I, J \in
      \mathbb{N}_0^n}$ and it consists of real-analytic functions with
    Taylor coefficients having sub-factorial growth with respect to
    $|I| + |J|$.  The underlying Fréchet space is isomorphic to a
    Köthe space of sub-factorial growth.
\end{theorem}

%
%

\section{The star product on the Poincaré disk}
\label{sec:StarProductOnThePoincareDisc}

We come now to our main example, the star product on the Poincaré disk
and on its higher dimensional cousins.  The star product we are
interested in appeared in the literature in the work
\cite{moreno.ortega-navarro:1983b}, see also
\cite[Sect.~9]{cahen.gutt.rawnsley:1994a} for an explicit formula for
a particular class of functions. We shall rely on the construction
from \cite{bordemann.brischle.emmrich.waldmann:1996a,
  bordemann.brischle.emmrich.waldmann:1996b}, which works both for
$\mathbb{CP}^n$ and the disk. It is inspired by classical phase space
reduction and yields very explicit and manageable formulas.

%
%

\subsection{The geometry of the disk}
\label{subsec:GeometryDisc}

In this subsection we recall some basic features of the geometry of
the disk $\mathbb{D}_n$ to fix our notation.

On $\mathbb{C}^{n+1}$ we consider a modified symplectic structure by
means of a different pseudo-Kähler metric: let $\metric = \diag(-1,
+1, \ldots, +1)$ and consider the function
\begin{equation}
    \label{eq:y}
    y(z)
    = - \metric(z, \cc{z})
    = |z^0|^2 - |z^1|^2 - \cdots - |z^n|^2,
\end{equation}
where $z^0, \ldots, z^n$ are the usual holomorphic coordinate
functions on $\mathbb{C}^{n+1}$. Everything we will do will take place
in the open cone
\begin{equation}
    \label{eq:TheCone}
    C^+_{n+1} = \{z \in \mathbb{C}^{n+1} \; | \; y(z) > 0\}.
\end{equation}
The disk is now obtained as a phase space reduction with respect to
the $\mathrm{U}(1)$-action induced by the Hamiltonian flow of
$y$. Indeed, the flow of $y$ with respect to the Poisson bracket
corresponding to $\metric$ is easily shown to be the
$\mathrm{U}(1)$-action $z \mapsto \E^{\I \varphi} z$ where the point
$z$ becomes simply multiplied by the phase. Clearly, $C^+_{n+1}$ is
invariant under this action. More precisely, every level set of
constant $y$ is invariant. Thus first restricting to such a ``mass
shell'', say for $y = 1$, and then dividing by the remaining
$\mathrm{U}(1)$-symmetry gives again a symplectic manifold, the disk
$\mathbb{D}_n$.  To handle the $\mathrm{U}(1)$-symmetry, we shall use
the holomorphic and anti-holomorphic Euler operators $E = \sum_{i=0}^n
z^i \frac{\partial}{\partial z^i}$ and $\cc{E}$. Then a function is
$\mathrm{U}(1)$-invariant iff $E f = \cc{E} f$.

The disk $\mathbb{D}_n$ can alternatively be viewed as an open subset
of $\mathbb{CP}^n$ where we just take those complex lines in
$\mathbb{C}^{n+1}$ which lie in $C^+_{n+1}$, i.e.\
\begin{equation}
    \label{eq:DiscInCPn}
    \mathbb{D}_n =
    \{
    [z] \in \mathbb{CP}^n \; | \; z \in C^+_{n+1}
    \}.
\end{equation}
This makes $\mathbb{D}_n$ a complex manifold. Since for every $[z] \in
\mathbb{D}_n$ we have $z^0 \ne 0$, the (local) holomorphic coordinates
$v^i = \frac{z^i}{z^0}$ of $\mathbb{CP}^n$ are \emph{globally} defined
holomorphic coordinates on $\mathbb{D}_n$. This shows that
$\mathbb{D}_n$ is diffeomorphic to $\mathbb{R}^{2n}$ and holomorphic
to the open unit ball in $\mathbb{C}^n$ determined by the inequality
$|v^1|^2 + \cdots + |v^n|^2 < 1$ as $y(z) > 0$. The particular case of
$n = 1$ gives the usual unit disk $\mathbb{D}_1 = \{v \in \mathbb{C}
\; | \; |v| < 1\}$.  Moreover, it can be shown that the reduced
symplectic form and the complex structure fit together nicely and
yield a (now actually positive definite) Kähler manifold of constant
negative holomorphic curvature. The canonical projection
\begin{equation}
    \label{eq:ProjectionOntoDisc}
    \pi\colon C_{n+1}^+ \longrightarrow \mathbb{D}_n
\end{equation}
is a holomorphic surjection. It allows to pull back functions $u \in
C^\infty(\mathbb{D}_n)$ to $C^\infty(C_{n+1}^+)$ which will be called
\emph{homogeneous} functions on $C_{n+1}^+$. Indeed, the functions of
the form $f = \pi^* u$ with $u \in C^\infty(\mathbb{D}_n)$ are
precisely the functions invariant under the multiplicative action of
$\mathbb{C} \setminus \{0\}$. With other words, they are
$\mathrm{U}(1)$-invariant and constant in direction $y$. Equivalently,
they satisfy $E f = 0 = \cc{E} f$.

Finally, $\mathbb{D}$ is the noncompact dual of $\mathbb{CP}^n$ as a
Hermitian symmetric space: the function $y$ is clearly invariant under
the usual linear group action of $\mathrm{SU}(1, n)$ on
$\mathbb{C}^{n+1}$. Hence it also acts on the complex lines in
$C^+_{n+1}$, even in a transitive manner. The stabilizer of the
standard point $(1, 0, \ldots, 0)$ in the level set $y = 1$ is now
$\mathrm{S}(\mathrm{U}(1) \times \mathrm{U}(n))$ and hence the last
version of the disk is
\begin{equation}
    \label{eq:DiscAsSymmetricSpace}
    \mathbb{D}_n
    =
    \frac{\mathrm{SU}(1, n)}
    {\mathrm{S}(\mathrm{U}(1) \times \mathrm{U}(n))}.
\end{equation}
The symmetry respects both the complex structure and the reduced
symplectic structure. The symmetry is even Hamiltonian with an
equivariant momentum map given by
\begin{equation}
    \label{eq:MomentumMapsuEinsn}
    J_\xi([z])
    =
    \frac{\I}{2}
    \frac{
      \sum_{i, j, k = 0}^n
      \cc{z}^i \metric^{ij}\xi_{jk} z^k
    }{y(z)}
    \quad
    \textrm{for}
    \quad
    \xi \in \mathfrak{su}(1, n),    
\end{equation}
where we use $\mathfrak{su}(1, n) \subseteq M_{n+1}(\mathbb{C})$. For
the case $n = 1$ the group $\mathrm{SU}(1, 1) =
\mathrm{SL}_2(\mathrm{R})$ acts by Möbius transformations on
$\mathbb{D}_1$. Moreover, in this particular case
\eqref{eq:DiscAsSymmetricSpace} simplifies to $\mathbb{D}_1 =
\mathrm{SU}(1, 1)\big/ \mathrm{U}(1)$ and the level set $y = 1$ can be
identified with $\mathrm{SU}(1, 1)$: for a given point $p$ in the mass
shell $y = 1$ there is precisely one group element $U \in
\mathrm{SU}(1, 1)$ which moves the point $(1, 0)$ to $p$. In higher
dimensions this uniqueness fails. These statements are of course all
very much standard and can be found e.g.\ in textbooks like
\cite[Chap.~VIII and Chap.~IX]{helgason:2000a}.

%
%

\subsection{Construction of the star product}
\label{subsec:ConstructionStarProduct}

In a next step we recall the basic features of the star product on the
disk and review its explicit construction from
\cite{bordemann.brischle.emmrich.waldmann:1996a,
  bordemann.brischle.emmrich.waldmann:1996b}.

Matching to the modified symplectic structure on $\mathbb{C}^{n+1}$ we
consider the corresponding Wick star product
\begin{equation}
    \label{eq:WickOnPseudoCnPlusEins}
    f \starwick g
    =
    \sum_{r=0}^\infty \frac{(2 \lambda)^r}{r!}
    \sum_{\substack{
        i_1, \ldots, i_r = 0 \\
        j_1, \ldots, j_r = 0}
    }^n
    \metric^{i_1 j_1} \cdots \metric^{i_r j_r}
    \frac{\partial^r f}{\partial z^{i_1} \cdots z^{i_r}}
    \frac{\partial^r g}{\partial \cc{z}^{j_1} \cdots \cc{z}^{j_r}},
\end{equation}
again first in the formal power series setting, i.e.\ for $f, g \in
C^\infty(\mathbb{C}^{n+1})[[\lambda]]$. We note that the polynomials
in $z$ and $\cc{z}$ form a subalgebra where we can substitute
$\lambda$ by $\hbar$ without difficulties.
\begin{remark}
    \label{remark:WickInvariant}%
    The Wick star product is clearly invariant under $\mathrm{SU}(1,
    n)$: for $U \in \mathrm{SU}(1, n)$ we have for all $f, g \in
    C^\infty(\mathbb{C}^{n+1})[[\lambda]]$
    \begin{equation}
        \label{eq:WickInvariant}
        (U^*f) \starwick (U^*g) = 
        U^*(f \starwick g).
    \end{equation}
    It is even strongly invariant with respect to the usual momentum
    map, i.e.\ for
    \begin{equation}
        \label{eq:SUEinsnMomentumMap}
        J_\xi(z)
        =
        \frac{\I}{2} \sum_{i, j, k} \cc{z}^k \metric^{ki} \xi^{ij} z^j
    \end{equation}
    with $\xi \in \mathfrak{su}(1, n)$ we have
    \begin{equation}
        \label{eq:JxiMomentumMap}
        J_\xi \starwick f - f \starwick J_\xi
        =
        \I\lambda \{J_\xi, f\}
        \quad
        \textrm{and}
        \quad
        J_\xi \starwick J_\eta - J_\eta \starwick J_\xi
        =
        \I\lambda J_{[\xi, \eta]},
    \end{equation}
    where $f \in C^\infty(\mathbb{C}^{n+1})[[\lambda]]$ and $\xi, \eta
    \in \mathfrak{su}(1, n)$. It is this invariance which we will need
    later.
\end{remark}

We call a function $f \in C^\infty(C_{n+1}^+)$ \emph{radial} if it
depends only on the ``radius'' $y$, i.e.\ if there is a smooth
function $\varrho:(0, \infty) \longrightarrow \mathbb{C}$ with $f =
\varrho \circ y$. It will be useful to consider the following global
vector field $\frac{\partial}{\partial y} = \frac{1}{2y}(E + \cc{E})$
on $C_{n+1}^+$. On radial functions it is indeed just differentiation
with respect to $y$. With this vector field we have the following
formula
\begin{equation}
    \label{eq:WickRadialInvariant}
    R \starwick F
    =
    \sum_{r=0}^\infty \frac{(2\lambda)^r}{r!} y^r
    \frac{\partial^r R}{\partial y^r}
    \frac{\partial^r F}{\partial y^r}
    =
    F \starwick R,
\end{equation}
for the Wick star product of a radial function $R$ and a
$\mathrm{U}(1)$-invariant function $F$. In particular, the radial
functions constitute a commutative subalgebra.

In a next step one makes use of an equivalence transformation $S$ from
the Wick star product to a new one, denoted by $\tildewick$, i.e.\
\begin{equation}
    \label{eq:TildeWick}
    f \tildewick g = S (S^{-1} f \starwick S^{-1} g).
\end{equation}
The aim is that for $\tildewick$ the product of a radial function and
an arbitrary $\mathrm{U}(1)$-invariant function becomes not only
commutative but \emph{pointwise}.  As usual the equivalence $S$ is a
formal series of differential operators starting with the identity. We
require that it contains only powers of $\frac{\partial}{\partial
  y}$. Hence it is completely determined by its symbol
$\hat{S}(\alpha, y) = \E^{-\alpha y} S \E^{\alpha y}$.  The required
property gives a functional equation for the symbol which can be
solved in an essentially unique way. The resulting formal differential
operator has the following properties which we recall from
\cite{bordemann.brischle.emmrich.waldmann:1996a}. For all $r \in
\mathbb{N}$ we have
\begin{equation}
    \label{eq:SPowersOfy}
    S 1 = 1,
    \quad
    S y^r
    =
    y^r \prod_{k=0}^{r-1} \left(1 + k \frac{2\lambda}{y}\right),
    \quad
    \textrm{and}
    \quad
    S y^{-r}
    =
    y^{-r} \prod_{k=1}^{r} \left(1 - k\frac{2\lambda}{y}\right)^{-1}.
\end{equation}
We see that we can substitute $\lambda$ by $\hbar$. Moreover, we can
rewrite the action of $S$ on $y^r$ for $r \ge 0$ by means of the
\emph{Pochhammer symbols} (or raising factorials) in the following way
\begin{equation}
    \label{eq:Syr}
    S y^r = (2\hbar)^r \Pochhammer{\frac{y}{2\hbar}}_r,
\end{equation}
where as usual $(\alpha)_r = \alpha (\alpha + 1) \cdots (\alpha +
(r-1))$.

Since we are interested in evaluating this on $y = 1$ later on, we
have to take care of the zeros of the Pochhammer symbols. We will call
$\hbar$ an \emph{allowed} value if
\begin{equation}
    \label{eq:Allowedhbar}
    2\hbar \in 
    \mathbb{C} \setminus 
    \left\{0, -1, -\frac{1}{2}, -\frac{1}{3}, \ldots\right\}.
\end{equation}
Equivalently, $\Pochhammer{\frac{1}{2\hbar}}_\alpha \ne 0$ for all
$\alpha \in \mathbb{N}_0$.  In
\cite{bordemann.brischle.emmrich.waldmann:1996a,
  bordemann.brischle.emmrich.waldmann:1996b} the case of
$\mathbb{CP}^n$ was considered. Here the critical values of $\hbar$
were $+ \frac{1}{2r}$ on the positive half axis. This is the reason
why we will deal with the disk $\mathbb{D}_n$ instead: all positive
$\hbar$ will be allowed.

According to \cite{bordemann.brischle.emmrich.waldmann:1996a,
  bordemann.brischle.emmrich.waldmann:1996b}, the explicit formula for
$\tildewick$ is now
\begin{equation}
    \label{eq:TildeWickExplicit}
    F \tildewick G
    =
    \sum_{r=0}^\infty \frac{1}{r!} \left(\frac{2\lambda}{y}\right)^r
    \prod_{k=1}^{r}\left(1 - k \frac{2\lambda}{y}\right)^{-1}
    y^r
    \sum_{\substack{i_1, \ldots, i_r = 0 \\ j_1, \ldots, j_r = 0}}^n
    \metric^{i_1 j_1} \cdots \metric^{i_r j_r}
    \frac{\partial^r F}
    {\partial z^{i_1} \cdots \partial z^{i_r}}
    \frac{\partial^r G}
    {\partial \cc{z}^{j_1} \cdots \partial \cc{z}^{j_r}}
\end{equation}
for $\mathrm{U}(1)$-invariant functions $F, G \in
C^\infty(C_{n+1}^+)[[\lambda]]$. Moreover, since the radial functions
behave like scalars by the very design of $\tildewick$, one can show
that the ideal generated with respect to $y-1$ \emph{coincides} with
the classical vanishing ideal of $y = 1$. This allows to set $y = 1$
in \eqref{eq:TildeWickExplicit} for $\mathrm{U}(1)$-invariant
functions to obtain an explicit star product $\stardisk$ on the disk
$\mathbb{D}_n$.
\begin{remark}
    \label{remark:SymmetryOfTildewick}%
    The equivalence transformation $S$ is clearly $\mathrm{SU}(1,
    n)$-invariant. Thus the new star product $\tildewick$ is again a
    $\mathrm{SU}(1, n)$-invariant star product. Moreover, since $Sy =
    y$ and $\frac{1}{y}J_\xi$ is homogeneous we conclude that $SJ_\xi
    = J_\xi$ for all $\xi \in \mathfrak{su}(1, n)$. It follows easily
    that $\tildewick$ is again strongly invariant under
    $\mathrm{SU}(1, n)$ with the same (quantum) momentum map $J_\xi$
    as already $\starwick$, see also
    \cite[Lem.~5]{bordemann.brischle.emmrich.waldmann:1996b} for
    details.
\end{remark}

%
%

\subsection{The basis and the structure constants}
\label{subsec:BasisStructureConstants}

In a next step we shall consider a subalgebra of the formal star
product where we have convergence for $\hbar$ in the set
\eqref{eq:Allowedhbar} and a countable vector space basis.  Let $P, Q
\in \mathbb{N}_0^n$ be multiindices of length $n$ and let $\alpha \in
\mathbb{N}_0$ with $|P|, |Q| \le \alpha$. Then we consider the
functions
\begin{equation}
    \label{eq:ePQalpha}
    \basis{e}_{P, Q, \alpha}(z)
    =
    (z^0)^{\alpha - |P|} z^P (\cc{z}^0)^{\alpha - |Q|} \cc{z}^Q,
\end{equation}
where we use the multiindex notation $z^P = (z^1)^{P_1} \cdots
(z^n)^{P_n}$ as usual. Clearly, $\basis{e}_{P, Q, \alpha}$ is
$\mathrm{U}(1)$-invariant. In fact, such a monomial is homogeneous of
degree $\alpha$ for both Euler operators $E$ and $\cc{E}$, i.e.\ we
have $E\basis{e}_{P, Q, \alpha} = \alpha \basis{e}_{P, Q, \alpha} =
\cc{E} \basis{e}_{P, Q, \alpha}$. Note that we view the monomials as
functions on $C_{n+1}^+$ from now on. Note also that every
$\mathrm{U}(1)$-invariant polynomial is a linear combination of these
monomials which explains our choice. In the following, we call $(P, Q,
\alpha)$ an \emph{index triple} if $P$ and $Q$ are multiindices with
$|P|, |Q| \le \alpha$.
\begin{lemma}
    \label{lemma:eWicke}%
    For all index triples $(P, Q, \alpha)$ and $(R, S, \beta)$ we have
    \begin{equation}
        \label{eq:eWicke}
        \begin{split}
            &\basis{e}_{P, Q, \alpha}
            \starwick
            \basis{e}_{R, S, \beta}
            = 
            \sum_{k=0}^{\min(\alpha - |P|, \beta - |S|)}
            \sum_{K = 0}^{\min(P, S)}
            \frac{(2\hbar)^{k+|K|}}{k! K!}
            (-1)^k \\
            &\qquad\qquad\qquad
            \frac{(\alpha - |P|)!}{(\alpha - |P| - k)!}
            \frac{P!}{(P - K)!}
            \frac{(\beta - |S|)!}{(\beta - |S| - k)!}
            \frac{S!}{(S - K)!}
            \basis{e}_{
              P + R - K, Q + S - K, \alpha + \beta - k - |K|
            }, 
        \end{split}
    \end{equation}
    where the minimum of multiindices is taken componentwise as usual.
\end{lemma}
\begin{proof}
    This is just an exercise in differentiation.
\end{proof}

We take now the monomials and apply the equivalence $S$ to them. A
priori, this will give a formal power series in $\lambda$ but from the
explicit formula we see that the expression will make sense for all
$\hbar \ne 0$. The following result is clear from \eqref{eq:Syr}:
\begin{lemma}
    \label{lemma:SofMonomial}%
    For all index triples $(P, Q, \alpha)$ and $\hbar \ne 0$ we have
    \begin{equation}
        \label{eq:SMonomial}
        S \basis{e}_{P, Q, \alpha}
        =
        (2 \hbar)^\alpha
        \Pochhammer{\frac{y}{2\hbar}}_\alpha
        \frac{\basis{e}_{P, Q, \alpha}}{y^\alpha}.
    \end{equation}
\end{lemma}
\begin{proof}
    Indeed, since $S$ is the identity on pullbacks $f = \pi^*u$ of
    functions on $\mathbb{D}_n$, we have by \eqref{eq:Syr}
    \[
    S \basis{e}_{P, Q, \alpha}
    =
    S \left(
        y^\alpha \frac{\basis{e}_{P, Q, \alpha}}{y^\alpha}
    \right)
     =
    S(y^\alpha) \frac{\basis{e}_{P, Q, \alpha}}{y^\alpha}
    =
    (2 \hbar)^\alpha
    \Pochhammer{\frac{y}{2\hbar}}_\alpha
    \frac{\basis{e}_{P, Q, \alpha}}{y^\alpha}.
    \]
\end{proof}

For $\hbar \ne 0$ the functions obtained from the monomials after
applying $S$ are still linearly independent:
\begin{lemma}
    \label{lemma:LinearIndependent}%
    Let $\hbar \ne 0$. Then the functions
    \begin{equation}
        \label{eq:LinearIndependent}
        \left\{
            \Pochhammer{\frac{y}{2\hbar}}_\alpha
            \frac{\basis{e}_{P, Q, \alpha}}{y^\alpha}.
        \right\}_{
          \alpha \in \mathbb{N}_0,
          P, Q  \in \mathbb{N}_0^n,
          \alpha \ge \max(|P|, |Q|)
        }
    \end{equation}
    are linearly independent in $C^\infty(C_{n+1}^+)$.
\end{lemma}
\begin{proof}
    First it is clear that the monomials $\basis{e}_{P, Q, \alpha}$
    are linearly independent. Second, the Pochhammer symbol
    $\Pochhammer{\frac{y}{2\hbar}}_\alpha$ has leading term
    $\left(\frac{y}{2\hbar}\right)^\alpha$ for $y \longrightarrow
    +\infty$ while $\frac{\basis{e}_{P, Q, \alpha}}{y^\alpha}$ stays
    bounded. Thus we can recover $\alpha$ from the asymptotics of
    $\Pochhammer{\frac{y}{2\hbar}}_\alpha \frac{\basis{e}_{P, Q,
        \alpha}}{y^\alpha}$. From these two facts the statement
    follows easily.
\end{proof}
We anticipate here that evaluating at $y = 1$ will give additional
linear dependencies if $\hbar$ is \emph{not} an allowed value.

In view of the product formula it will be advantageous to rescale the
functions slightly. From a more physical point of view, we can make
them dimensionless in order to get dimensionless structure constants,
i.e.\ independent of $\hbar$. This gives the final definition: for an
index triple $(P, Q, \alpha)$ and $\hbar \ne 0$ we consider the
functions
\begin{equation}
    \label{eq:fPQalphaDef}
    \basis{f}_{P, Q, \alpha}
    =
    \frac{1}{(2\hbar)^\alpha}
    \frac{1}{P!(\alpha - |P|)!Q!(\alpha - |Q|)!}
    S \basis{e}_{P, Q, \alpha}
    =
    \frac{1}{P!(\alpha - |P|)!Q!(\alpha - |Q|)!}
    \Pochhammer{\frac{y}{2\hbar}}_\alpha
    \frac{\basis{e}_{P, Q, \alpha}}{y^\alpha}.
\end{equation}
They are still linearly independent and hence we consider their span
inside $C^\infty(C_{n+1}^+)$ which we will denote by
\begin{equation}
    \label{eq:Aoben}
    \mathcal{A}_\hbar(C_{n+1}^+)
    =
    \mathbb{C}\textrm{-}\spann
    \left\{
        \basis{f}_{P, Q, \alpha}
        \; \big| \;
        \alpha \in \mathbb{N}_0,
        P, Q  \in \mathbb{N}_0^n,
        \alpha \ge \max(|P|, |Q|)
    \right\},
\end{equation}
for which the $\basis{f}_{P, Q, \alpha}$ will be a countable vector
space basis. We write $a = \sum_{(P, Q, \alpha)} a_{P, Q, \alpha}
\basis{f}_{P, Q, \alpha}$ for $a \in \mathcal{A}_\hbar(C_{n+1}^+)$ as
usual.  In fact, $\mathcal{A}_\hbar(C_{n+1}^+)$ will be a subalgebra
with respect to $\tildewick$:
\begin{proposition}
    \label{proposition:StructureConstantsAbove}%
    Let $\hbar \ne 0$.
    \begin{propositionlist}
    \item \label{item:AlgebraForTildeWick}
        $\mathcal{A}_\hbar(C_{n+1}^+)$ is a unital subalgebra with
        respect to $\tildewick$.
    \item \label{item:AlgebraForRealHbar} For real $\hbar$ the algebra
        $\mathcal{A}_\hbar(C_{n+1}^+)$ is a $^*$-algebra with respect
        to the pointwise complex conjugation and one has
        \begin{equation}
            \label{eq:ccfPQalpha}
            \cc{\basis{f}_{P, Q, \alpha}}
            =
            \basis{f}_{Q, P, \alpha}.
        \end{equation}
    \item \label{item:StructureConstants} The structure constants of
        $\mathcal{A}_\hbar(C_{n+1}^+)$ with respect to the basis
        \eqref{eq:fPQalphaDef} are explicitly given by
        \begin{equation}
            \label{eq:StructureConstantsTildeWick}
            \begin{split}
                C_{(P, Q, \alpha), (R, S, \beta)}^{(I, J, \gamma)}
                &=
                \frac{
                  (-1)^{\alpha + \beta - \gamma  - |P| - |R| + |I|}
                  \epsilon(P, Q, \alpha, R, S, \beta, I, J, \gamma)
                }
                {
                  (\alpha + \beta - \gamma - |P| - |R| + |I|)!
                  (P + R - I)!
                } \\
                &\qquad\times\quad
                \binom{I}{R}
                \binom{J}{Q}
                \binom{\gamma - |I|}{\beta - |R|}
                \binom{\gamma - |J|}{\alpha - |Q|}
                \delta_{P + R - I, Q + S - J},
            \end{split}
        \end{equation}
        where the quantity
        \begin{equation}
            \epsilon(P, Q, \alpha, R, S, \beta, I, J, \gamma)
            =
            \sum_{k=0}^{\min(\alpha - |P|, \beta - |S|)}
            \sum_{K = 0}^{\min(P, S)}
            \delta_{I, P + R - K}
            \delta_{\gamma, \alpha + \beta - k - |K|}.
        \end{equation}
        takes only the values $0$ or $1$. In particular, the structure
        constant will be zero if one of the conditions
        \begin{equation}
            \label{eq:NecessaryInequalities}
            R \le I,
            \quad
            Q \le J,
            \quad
            \max(\alpha, \beta) \le \gamma \le \alpha + \beta
        \end{equation}
        is violated.
    \end{propositionlist}
\end{proposition}
\begin{proof}
    First we note that $\basis{f}_{0, 0, 0} = 1$. If $\hbar =
    \cc{\hbar}$ then $S$ commutes with the pointwise complex
    conjugation and hence $\tildewick$ is a Hermitian star product
    since $\starwick$ is obviously a Hermitian star product. To show
    that $\mathcal{A}_\hbar(C_{n+1}^+)$ is a subalgebra, we have to
    compute the product of two elements of the basis. Using
    \eqref{eq:fPQalphaDef} as well as Lemma~\ref{lemma:eWicke} we have
    \begin{align*}
        &\basis{f}_{P, Q, \alpha} \tildewick \basis{f}_{R, S, \beta} \\
        &=
        S\left(
            S^{-1} \basis{f}_{P, Q, \alpha}
            \starwick
            S^{-1} \basis{f}_{R, S, \beta}
        \right) \\
        &=
        S\left(
            \frac{1}{(2\hbar)^\alpha}
            \frac{1}{P!(\alpha - |P|)!Q!(\alpha - |Q|)!}
            \frac{1}{(2\hbar)^\beta}
            \frac{1}{R!(\beta - |R|)!S!(\beta - |S|)!}
            \basis{e}_{P, Q, \alpha}
            \starwick
            \basis{e}_{R, S, \beta}
        \right) \\
        &=
        \frac{1}{Q!(\alpha - |Q|)!}
        \frac{1}{R!(\beta - |R|)!}
        \sum_{k=0}^{\min(\alpha - |P|, \beta - |S|)}
        \sum_{K = 0}^{\min(P, S)}
        \frac{(2\hbar)^{k+|K| - \alpha - \beta}}{k! K!}
        (-1)^k \\
        &\qquad\times\quad
        \frac{1}{(\alpha - |P| - k)!}
        \frac{1}{(P - K)!}
        \frac{1}{(\beta - |S| - k)!}
        \frac{1}{(S - K)!}
        S \basis{e}_{P + R - K, Q + S - K, \alpha + \beta - k - |K|} \\
        &= 
        \sum_{k=0}^{\min(\alpha - |P|, \beta - |S|)}
        \sum_{K = 0}^{\min(P, S)}
        \frac{(-1)^k}{k!K!} \\
        &\qquad\times\quad
        \frac{
          (P + R - K)!
          (\alpha + \beta - k - |P| - |R|)!
          (Q + S - K)!
          (\alpha + \beta - k - |Q| - |S|)!
        }
        {
          (\alpha - |P| - k)!
          (P - K)!
          Q!
          (\alpha - |Q|)!
          R!
          (\beta - |R|)!
          (\beta - |S| - k)!
          (S - K)!
        } \\
        &
        \qquad\times\quad
        \basis{f}_{P + R - K, Q + S - K, \alpha + \beta - k - |K|} \\
        &=
        \sum_{k=0}^{\min(\alpha - |P|, \beta - |S|)}
        \sum_{K = 0}^{\min(P, S)}
        \frac{(-1)^k}{k!K!}
        \basis{f}_{P + R - K, Q + S - K, \alpha + \beta - k - |K|} \\
        &\qquad\times\quad
        \binom{P + R - K}{R}
        \binom{Q + S - K}{Q}
        \binom{\alpha + \beta - k - |P| - |R|}{\beta - |R|}
        \binom{\alpha + \beta - k - |Q| - |S|}{\alpha - |Q|}.
    \end{align*}
   This shows that $\mathcal{A}_\hbar(C_{n+1}^+)$ is indeed a
   subalgebra, completing the proof of the first part. Moreover, we
   can read off the structure constants of $\tildewick$ from this
   formula and get
   \begin{align*}
       &C_{(P, Q, \alpha), (R, S, \beta)}^{(I, J, \gamma)} \\
       &=
       \sum_{k=0}^{\min(\alpha - |P|, \beta - |S|)}
       \sum_{K = 0}^{\min(P, S)}
       \frac{(-1)^k}{k!K!}
       \delta_{I, P + R - K}
       \delta_{J, Q + S - K}
       \delta_{\gamma, \alpha + \beta - k - |K|} \\
       &\qquad\times\quad
       \binom{P + R - K}{R}
       \binom{Q + S - K}{Q}
       \binom{\alpha + \beta - k - |P| - |R|}{\beta - |R|}
       \binom{\alpha + \beta - k - |Q| - |S|}{\alpha - |Q|} \\
       &=
       \frac{(-1)^{\alpha + \beta - \gamma  - |P| - |R| + |I|}}
       {(\alpha + \beta - \gamma - |P| - |R| + |I|)! (P + R - I)!}
       \binom{I}{R}
       \binom{J}{Q}
       \binom{\gamma - |I|}{\beta - |R|}
       \binom{\gamma - |J|}{\alpha - |Q|}
       \delta_{P + R - I, Q + S - J} \\
       &\qquad\times\quad
       \sum_{k=0}^{\min(\alpha - |P|, \beta - |S|)}
       \sum_{K = 0}^{\min(P, S)}
       \delta_{I, P + R - K}
       \delta_{\gamma, \alpha + \beta - k - |K|}.
   \end{align*}
   We have three $\delta$'s and only two summations. Hence one
   $\delta$ will survive. Implementing the conditions from the
   $\delta$'s we have a remaining summation over $k$ and $K$ with two
   of the $\delta$'s. This sum will either give $0$ or $1$ depending
   on whether in the allowed ranges we find the correct $k$ and $K$ or
   not.  If the summation would be unrestricted we would get $1$
   always.  Let us abbreviate this last sum by
   \[
   \epsilon(P, Q, \alpha, R, S, \beta, I, J, \gamma)
   =
   \sum_{k=0}^{\min(\alpha - |P|, \beta - |S|)}
   \sum_{K = 0}^{\min(P, S)}
   \delta_{I, P + R - K}
   \delta_{\gamma, \alpha + \beta - k - |K|}.
   \]
   In view of the $\delta$ in front of the sum we can replace
   $\delta_{I, P + R - K}$ also by $\delta_{J, Q + S - K}$. From this
   we get the two necessary conditions $Q \le J$ and $R \le I$ in
   order to get a nonzero structure constant. The condition on
   $\gamma$ is more delicate to evaluate. The reason is that the
   minimum $\min(P, S)$ might have strictly smaller length than the
   minimum of the lengths $\min(|P|, |S|)$. Nevertheless, we get an
   estimate that the structure constant is certainly zero unless
   \[
   \max(\alpha, \beta) \le \gamma \le \alpha + \beta.
   \tag{$*$}
   \]
   Note however, that there will be situations where the structure
   constants still vanish, even though ($*$) is satisfied.
\end{proof}

The important feature is that for a given index triple $(I, J,
\gamma)$ we only have \emph{finitely} many $(P, Q, \alpha)$ and $(R,
S, \beta)$ such that the corresponding structure constant is
nonzero. This follows from the estimate $\gamma \ge \max(\alpha,
\beta)$ and the conditions $\gamma \ge |I|, |J|$.


Note also that the rescaling resulted in structure constants
\emph{not} depending on the deformation parameter $\hbar$
anymore. Thus the recursion for the seminorms will not contain
$\hbar$. However, for a given function $a \in
\mathcal{A}_\hbar(C_{n+1}^+)$ the seminorms \emph{do} depend on
$\hbar$ and so does the topology, as the basis vectors $\basis{f}_{P,
  Q, \alpha}$ do. We will have to come back to this $\hbar$-dependence
in Subsection~\ref{subsec:DependenceOnhbar}.

The Wick star product $\starwick$ and also $\tildewick$ are Hermitian
star products if one treats the deformation parameter as a real
quantity. Since we have absorbed $\hbar$ into the definition of the
basis, we get a symmetry of the structure constants originating from
the complex conjugation being a $^*$-involution for real $\hbar$.
\begin{proposition}
    \label{proposition:SymmetryOfStructureConstants}%
    For all index triples the structure constants $C^{(I, J,
      \gamma)}_{(P, Q, \alpha), (R, S, \beta)}$ are real and one has
    \begin{equation}
        \label{eq:SymmetryStructureConstants}
        C^{(I, J, \gamma)}_{(S, R, \beta), (Q, P, \alpha)}
        =
        C^{(J, I, \gamma)}_{(P, Q, \alpha), (R, S, \beta)}.
    \end{equation}
\end{proposition}

We conclude this subsection by the following observation. The
restrictions \eqref{eq:NecessaryInequalities} on $\gamma$ lead to a
filtration of the algebra $\mathcal{A}_\hbar(C_{n+1}^+)$. We set
\begin{equation}
    \label{eq:AgammaDef}
    \mathcal{A}^\gamma_\hbar(C_{n+1}^+)
    =
    \bigoplus_{I, J} \mathbb{C} \basis{f}_{I, J, \gamma}
    \quad
    \textrm{and}
    \quad
    \mathcal{A}^{(\gamma)}_\hbar(C_{n+1}^+)
    =
    \bigoplus_{\alpha = 0}^\gamma \mathcal{A}^\gamma_\hbar(C_{n+1}^+),
\end{equation}
which are \emph{finite-dimensional} subspaces with
$\mathcal{A}^{(\gamma)}_\hbar(C_{n+1}^+) \subseteq
\mathcal{A}^{(\gamma + 1)}_\hbar(C_{n+1}^+)$. The following is then an
immediate consequence of
Proposition~\ref{proposition:StructureConstantsAbove}.
\begin{corollary}
    \label{corollary:FilteredAlgebra}%
    Let $\hbar \ne 0$.  The algebra $\mathcal{A}_\hbar(C_{n+1}^+)$ is
    filtered via the subspaces
    $\mathcal{A}^{(\gamma)}_\hbar(C_{n+1}^+)$, i.e.\
    \begin{equation}
        \label{eq:Filtered}
        \mathcal{A}_\hbar(C_{n+1}^+)
        =
        \bigcup_{\gamma = 0}^\infty
        \mathcal{A}^{(\gamma)}_\hbar(C_{n+1}^+)
        \quad
        \textrm{and}
        \quad
        \mathcal{A}^{(\alpha)}_\hbar(C_{n+1}^+)
        \tildewick
        \mathcal{A}^{(\beta)}_\hbar(C_{n+1}^+)
        \subseteq
        \mathcal{A}^{(\alpha + \beta)}_\hbar(C_{n+1}^+).
    \end{equation}
\end{corollary}

%
%

\subsection{First version}
\label{subsec:FirstVersionDisc}

To proceed we need to compute the constants \eqref{eq:CSeriesConverge}
from the structure constants as usual. Even though one can compute
them explicitly, the following properties and rough estimates will be
all we need:
\begin{lemma}
    \label{lemma:ConstantsForTheDisc}%
    Let $(I, J, \gamma)$ be an index triple.
    \begin{lemmalist}
    \item \label{item:DiscConstantsFinite} For all index triples $(P,
        Q, \alpha)$ the constant $C^{(I, J, \gamma)}_{(P, Q, \alpha),
          \boldsymbol{\cdot}}$ is finite and zero for $\alpha >
        \gamma$. Analogously, for all index triples $(R, S, \beta)$
        the constant $C^{(I, J, \gamma)}_{\boldsymbol{\cdot}, (R, S,
          \beta)}$ is finite and zero for $\beta > \gamma$.
    \item \label{item:GrowthDiscConstants} For all index triples $(P,
        Q, \alpha)$ and $(R, S, \beta)$ one has
        \begin{equation}
            \label{eq:EstimateDiscConstants}
            \sum_{I, J}
            C^{(I, J, \gamma)}_{(P, Q, \alpha), \boldsymbol{\cdot}}
            \le
            (\gamma + 1)^{4n+1} 4^\gamma
            \quad
            \textrm{and}
            \quad
            \sum_{I, J}
            C^{(I, J, \gamma)}_{\boldsymbol{\cdot}, (R, S, \beta)}
            \le
            (\gamma + 1)^{4n+1} 4^\gamma.
        \end{equation}
    \item \label{item:ConstantsAtLeastOne} For all index triples $(P,
        Q, \alpha)$ and all $\gamma \ge \alpha$ one has
        \begin{equation}
            \label{eq:ConstantsAtLeastOne}
            C^{(P, Q, \gamma)}_{(P, Q, \alpha), \boldsymbol{\cdot}}
            \ge 1
            \quad
            \textrm{and}
            \quad
            C^{(P, Q, \gamma)}_{\boldsymbol{\cdot}, (P, Q, \alpha)}
            \ge 1.
        \end{equation}
    \end{lemmalist}
\end{lemma}
\begin{proof}
    With the explicit formula from
    Proposition~\ref{proposition:StructureConstantsAbove} this is now
    a simple argument. First recall that in an index triple $(I, J,
    \gamma)$ there are only finitely many allowed $I$ and $J$ for a
    fixed $\gamma$ since $|I|, |J| \le \gamma$. Thus we only have to
    take care of the indices $\alpha$, $\beta$, and $\gamma$.  For a
    given $\gamma$ there are only finitely many $\alpha$ and $\beta$
    such that $\max(\alpha, \beta) \le \gamma \le \alpha + \beta$ can
    hold. In particular, $\alpha, \beta \le \gamma$ is a necessary
    condition. But then the summation over $\beta$ gives a finite
    constant $C^{(I, J, \gamma)}_{(P, Q, \alpha), \boldsymbol{\cdot}}$
    while the summation over $\alpha$ gives a finite $C^{(I, J,
      \gamma)}_{\boldsymbol{\cdot}, (R, S, \beta)}$. In addition, the
    conditions $\alpha \le \gamma$ and $\beta \le \gamma$,
    respectively, still persist giving the first statement. For the
    second we have
    \begin{align*}
        \sum_{I, J}
        C^{(I, J, \gamma)}_{(P, Q, \alpha), \boldsymbol{\cdot}}
        &=
        \sum_{I, J}
        \sum_{(R, S, \beta)}
        C^{(I, J, \gamma)}_{(P, Q, \alpha), (R, S, \beta)} \\
        &\le
        \sum_{I, J}
        \sum_{(R, S, \beta)}
        \frac{
          \binom{I}{R}
          \binom{J}{Q}
          \binom{\gamma - |I|}{\beta - |R|}
          \binom{\gamma - |J|}{\alpha - |Q|}
        }{(\alpha + \beta - \gamma - |P| - |R| + |I|)!(P + R - I)!}
        \\
        &\le
        \sum_{I, J}
        \sum_{(R, S, \beta)}
        4^\gamma
        \le
        (\gamma + 1)^{4n+1} 4^\gamma,
    \end{align*}
    since each index $I_\ell$ in $I$ with $\ell = 1, \ldots, n$ runs
    at most from $0$ to $\gamma$ and analogously for $J$, $R$, and
    $S$. Also $\beta$ runs at most from $0$ to $\gamma$. The other
    estimate is analogous. For the third estimate we note that the
    product of $\basis{f}_{P, Q, \alpha}$ with $\basis{f}_{0, 0,
      \beta}$ gives a contribution for $\basis{f}_{P, Q, \alpha +
      \beta}$ with a prefactor given by $\binom{\alpha + \beta -
      |P|}{\beta}\binom{\alpha + \beta - |Q|}{\beta} \ge 1$. This
    corresponds to the summation indices $k = 0$ and $K = 0$ in the
    computations in the proof of
    Proposition~\ref{proposition:StructureConstantsAbove}. Hence
    $C^{(P, Q, \alpha + \beta)}_{(P, Q, \alpha), (0, 0, \beta)} =
    C^{(P, Q, \alpha + \beta)}_{(0, 0, \beta), (P, Q, \alpha)} \ge
    1$. But then the third claim is clear.
\end{proof}

In principle, we can even compute the constants explicitly, using the
explicit formulas obtained in
Proposition~\ref{proposition:StructureConstantsAbove}. However, to
determine the topology of $\mathcal{A}_\hbar(C_{n+1}^+)$, the above
simple counting is already sufficient, a situation very similar to the
case of the polynomials in Subsection~\ref{subsec:Polynomials}.
\begin{proposition}
    \label{proposition:FirstTopologyTooCoarse}%
    Let $\hbar \ne 0$.  The topology of $\mathcal{A}_\hbar(C_{n+1}^+)$
    according to the construction as in
    Theorem~\ref{theorem:AniceCompleted} is the Cartesian product
    topology inherited from $\prod_{(P, Q, \alpha)}
    \mathbb{C}\mathsf{f}_{P, Q, \alpha}$.
\end{proposition}
\begin{proof}
    Thanks to Lemma~\ref{lemma:ConstantsForTheDisc} we have for the
    recursive definition of the $h_{m, \ell, (I, J, \gamma)}$ the
    \emph{finite} sums
    \[
    h_{m+1, 2\ell, (I, J, \gamma)}(a)
    =
    \sum_{(P, Q, \alpha)}^{\textrm{finite}}
    h_{m, \ell, (P, Q, \alpha)}(a)^2
    C^{(I, J, \gamma)}_{(P, Q, \alpha), \boldsymbol{\cdot}}
    \]
    and
    \[
    h_{m+1, 2\ell+1, (I, J, \gamma)}(a)
    =
    \sum_{(R, S, \beta)}^{\textrm{finite}}
    h_{m, \ell, (R, S, \beta)}(a)^2
    C^{(I, J, \gamma)}_{\boldsymbol{\cdot}, (R, S, \beta)}.
    \]
    A simple induction shows that in $h_{m, \ell, (I, J, \gamma)}(a)$
    only \emph{finitely} many coefficients $|a_{P, Q, \alpha}|$ of $a$
    contribute. Hence, up to a constant, we can estimate $\norm{a}_{m,
      \ell, (I, J, \gamma)}$ by the maximum of those finitely many
    $|a_{P, Q, \alpha}|$ which constitutes a continuous seminorm of
    the Cartesian product. Thus the Cartesian product topology is
    finer in this case. But it is coarser in general according to
    Theorem~\ref{theorem:AniceCompleted}, \refitem{item:Completion}.
\end{proof}


To conclude this subsection we collect a few more technical properties
of the seminorms which we shall need later on. Since for a given
$\gamma$ we have only finitely many $I$ and $J$ with $|I|, |J| \le
\gamma$ we can consider the new combination
\begin{equation}
    \label{eq:Dischmellgamma}
    h_{m, \ell, \gamma}(a)
    =
    \sum_{I, J}
    h_{m, \ell, (I, J, \gamma)}(a).
\end{equation}
We have corresponding seminorms $\norm{a}_{m, \ell, \gamma} =
\sqrt[2^m]{h_{m, \ell, \gamma}(a)}$. Clearly, they will produce the
Cartesian product topology as well.
\begin{lemma}
    \label{lemma:hmlgammaDiscStuff}%
    Let $m \in \mathbb{N}_0$, $\ell = 0, \ldots, 2^m-1$, and
    $\gamma \in \mathbb{N}_0$.
    \begin{lemmalist}
    \item \label{item:EstimateSumhmellSquares}
        \begin{equation}
            \label{eq:EstimateSumhmellSquares}
            \sum_{\alpha = 0}^\gamma
            h_{m, \ell, \alpha}(a)^2
            \le
            (\gamma + 1)^{2n}
            h_{m+1, 2\ell, \gamma}(a)
            \quad
            \textrm{and}
            \quad
            \sum_{\alpha = 0}^\gamma
            h_{m, \ell, \alpha}(a)^2
            \le
            (\gamma + 1)^{2n}
            h_{m+1, 2\ell+1, \gamma}(a).
        \end{equation}
    \item \label{item:SumPShmellPAaSquares}
        \begin{equation}
            \label{eq:SumPQhmlPAaSquares}
            \sum_{I, J}
            h_{m, \ell, (I, J, \gamma)}(a)^2
            \le
            h_{m, \ell, \gamma}(a)^2.
        \end{equation}
    \item \label{item:NormOnLessGamma} If $m \ge 1$ then
        $\norm{\argument}_{m, \ell, \gamma}$ is a norm on
        $\mathcal{A}_\hbar^{(\gamma)}(C_{n+1}^+)$ and identically zero
        on the complement $\bigoplus_{(P, Q, \alpha), \alpha > \gamma}
        \mathbb{C} \basis{f}_{P, Q, \alpha}$.
    \end{lemmalist}
\end{lemma}
\begin{proof}
    For the first part we use $C^{(P, Q, \gamma)}_{(P, Q, \alpha),
      \boldsymbol{\cdot}} \ge 1$ according to
    Lemma~\ref{lemma:ConstantsForTheDisc},
    \refitem{item:ConstantsAtLeastOne}. Thus we get
    \begin{align*}
        \sum_{\alpha = 0}^\gamma
        h_{m, \ell, \alpha}(a)^2
        &=
        \sum_{\alpha = 0}^\gamma
        \left(
            \sum_{|P|, |Q| \le \alpha}
            h_{m, \ell, (P, Q, \alpha)}(a)
        \right)^2\\
        &\le
        (\gamma + 1)^{2n}
        \sum_{\alpha = 0}^\gamma
        \sum_{|P|, |Q| \le \alpha}
        h_{m, \ell, (P, Q, \alpha)}(a)^2 \\
        &\le
        (\gamma + 1)^{2n}
        \sum_{\alpha = 0}^\gamma
        \sum_{|P|, |Q| \le \alpha}
        h_{m, \ell, (P, Q, \alpha)}(a)^2
        C^{(P, Q, \gamma)}_{(P, Q, \alpha), \boldsymbol{\cdot}} \\
        &\le
        (\gamma + 1)^{2n}
        \sum_{|I|, |J| \le \gamma}
        \sum_{\alpha = 0}^\gamma
        \sum_{|P|, |Q| \le \alpha}
        h_{m, \ell, (P, Q, \alpha)}(a)^2
        C^{(I, J, \gamma)}_{(P, Q, \alpha), \boldsymbol{\cdot}} \\
        &=
        (\gamma + 1)^{2n}
        h_{m+1, 2\ell, \gamma}(a),
    \end{align*}
    where in the first estimate we use Hölder's inequality for
    \emph{finite} sums and the rough estimate that all indices run at
    most from $0$ to $\gamma$. In the second we used
    Lemma~\ref{lemma:ConstantsForTheDisc},
    \refitem{item:ConstantsAtLeastOne}.  Analogously, we can take
    $C^{(P, Q, \gamma)}_{\boldsymbol{\cdot}, (P, Q, \alpha)}$ in the
    second estimate, proving the first part.  The second part is
    trivial since all terms in \eqref{eq:Dischmellgamma} are
    nonnegative.  For the last we note that $\norm{a}_{0, 0, \alpha}
    = \sum_{P, Q} |a_{P, Q, \alpha}|$ is a norm on the span of the
    $\basis{f}_{P, Q, \alpha}$ for fixed $\alpha$. But then the first
    part shows that $h_{1, 0, \gamma}(a)$ being zero implies $h_{0, 0,
      \alpha}(a) = 0$ for all $\alpha \le \gamma$. Hence it follows
    that $\norm{\argument}_{1, 0, \gamma}$ is a norm on the span of
    all $\basis{f}_{P, Q, \alpha}$ with $\alpha \le \gamma$ and
    analogously for $\norm{\argument}_{1, 1, \gamma}$. For the higher
    $m$ the norm property follows again by the first part by
    induction.  Moreover, a simple induction using
    Lemma~\ref{lemma:ConstantsForTheDisc},
    \refitem{item:DiscConstantsFinite}, shows that for $h_{m, \ell,
      \gamma}(a)$ one never uses coefficients $|a_{P, Q, \alpha}|$
    with $\alpha > \gamma$.
\end{proof}

%
%

\subsection{Making  the evaluation functionals continuous}
\label{subsec:EvaluationFunctionalsContinuous}

We want to refine the topology now in such a way that the evaluation
functionals at an arbitrary point in $C_{n+1}^+$ become continuous.
This will allow us to interpret the elements of the completion still
as \emph{functions} on $C_{n+1}^+$, a feature which we do not want to
loose.

The following technical lemma will help us to show that the completion
will not be too small. The situation is very similar to the polynomial
algebra from Section~\ref{subsec:Polynomials}: again a sub-factorial
behavior with respect to the index $\gamma$ will reproduce in the
recursion of the $h_{m, \ell, (I, J, \gamma)}$.
\begin{lemma}
    \label{lemma:GrowthhmlSubfactorialDisc}%
    Let $\hbar \ne 0$ and $a = \sum_{I, J, \gamma} a_{I, J, \gamma}
    \basis{f}_{I, J, \gamma} \in \prod_{I, J, \gamma} \mathbb{C}
    \basis{f}_{I, J, \gamma}$ be an element in the Cartesian product
    with sub-factorial growth with respect to $\gamma$, i.e.\ for all
    $\epsilon > 0$ there is a constant $c_0 > 0$ such that
    \begin{equation}
        \label{eq:aSubfactorialInGamma}
        |a_{I, J, \gamma}| \le c_0 (\gamma!)^\epsilon
    \end{equation}
    for all index triples $(I, J, \gamma)$. Then for all $m \in
    \mathbb{N}_0$ and $\ell = 0, \ldots, 2^m -1$ and all $\epsilon >
    0$ we have a constant $c_m > 0$ with
    \begin{equation}
        \label{eq:EstimateGrowthhml}
        h_{m, \ell, (I, J, \gamma)}(a) \le 
        c_m (\gamma!)^\epsilon.
    \end{equation}
\end{lemma}
\begin{proof}
    Again, we prove this by induction on $m$. For $m = 0$ this is
    precisely the assumption on $a$. From
    Lemma~\ref{lemma:ConstantsForTheDisc},
    \refitem{item:GrowthDiscConstants}, we know that the constants
    $C^{(I, J, \gamma)}_{(P, Q, \alpha), \boldsymbol{\cdot}}$ as well
    as $C^{(I, J, \gamma)}_{\boldsymbol{\cdot}, (R, S, \beta)}$ can be
    estimated by some $c^\gamma$. Hence for all $\epsilon > 0$
    \[
    h_{m+1, 2\ell, (I, J, \gamma)}(a)
    =
    \sum_{(P, Q, \alpha)}
    h_{m, \ell, (P, Q, \alpha)}(a)^2
    C_{(P, Q, \alpha), \boldsymbol{\cdot}}^{(I, J, \gamma)}
    \le
    \sum_{(P, Q, \alpha)}
    c_m^2 (\alpha!)^{2\epsilon}
    c^\gamma
    \le
    c_{m+1} (\gamma!)^{3\epsilon},
    \]
    since the sum is finite and $\alpha \le \gamma$. Thus also
    $h_{m+1, 2\ell, (I, J, \gamma)}(a)$ has sub-factorial growth. The
    case $2\ell + 1$ is analogous.
\end{proof}
Note that even though the topology induced by the $h_{m, \ell,
  \gamma}$ is the (rather trivial) Cartesian product topology, the
above statement is nontrivial and would immediately fail if the $h_{m,
  \ell, (I, J, \gamma)}$ are multiplied by a suitable
$\gamma$-dependent factor. Moreover, since for a given $\gamma$ we
have only finitely many $I$ and $J$ also the quantities $h_{m, \ell,
  \gamma}(a)$ have sub-factorial growth for $a$ having sub-factorial
growth, i.e.\ for all $\epsilon > 0$ we have a (different) constant
$c_m > 0$ with
\begin{equation}
    \label{eq:SubfactorialGrowthAlsoForhmlgamma}
    h_{m, \ell, \gamma}(a) \le c_m (\gamma!)^\epsilon.
\end{equation}

Let $w \in C_{n+1}^+$. Then we consider the evaluation functional
\begin{equation}
    \label{eq:omegaEvaluation}
    \delta_w\colon
    \mathcal{A}_\hbar(C_{n+1}^+)
    \ni a \; \mapsto \; a(w) \in \mathbb{C}.
\end{equation}
This gives us new seminorms $\norm{\argument}_{m, \ell, \delta_w}$ on
$\mathcal{A}_\hbar(C_{n+1}^+)$ which have shown to be finite by
hand. According to the general construction from
Subsection~\ref{subsec:SecondVersion} they are given by
\begin{equation}
    \label{eq:NewSeminormsViaPoints}
    \norm{a}_{m, \ell, \delta_w}
    =
    \sqrt[2^m]{
      \sum_{(I, J, \gamma)}
      |\basis{f}_{I, J, \gamma}(w)|
      h_{m, \ell, (I, J, \gamma)}(a)
    }.
\end{equation}
From the explicit form of the functions $\basis{f}_{I, J, \gamma}$ the
evaluation at $w$ gives
\begin{equation}
    \label{eq:BasisAtwEvaluation}
    \left|\delta_w\left(\basis{f}_{I, J, \gamma}\right)\right|
    =
    \frac{1}{I!(\gamma - |I|)!J!(\gamma - |J|)!}
    \left|\Pochhammer{\frac{y(w)}{2\hbar}}_\gamma\right|
    \frac{|w^0|^{2\gamma - |I| - |J|} |w|^{I + J}}{y(w)^\gamma}.
\end{equation}
To elaborate further on the new seminorms we recall some basic
estimate on the Pochhammer symbols. Let $z \in \mathbb{C}$ then
\begin{equation}
    \label{eq:EstimatePochhammerSymbols}
    \mathsf{a} \mathsf{b}^\gamma
    \le
    \frac{1}{\gamma!}\left|\Pochhammer{z}_\gamma\right|
    \le
    \mathsf{c}^\gamma
\end{equation}
for constants $\mathsf{b}, \mathsf{c} > 0$ and $\mathsf{a} \ge 0$
depending on $z$. Note that $\mathsf{c}$ can be chosen locally
uniformly in $z$. Moreover, if in addition $-z$ is \emph{not} in
$\mathbb{N}_0$ then also $\mathsf{a} > 0$ and $\mathsf{a}$ and
$\mathsf{b}$ can be chosen locally uniformly in $z$ as well. Clearly,
for $-z \in \mathbb{N}_0$ the Pochhammer symbol is $0$ for large
enough $\gamma$.
\begin{lemma}
    \label{lemma:EquivalentSeminormSystem}%
    Let $\hbar \ne 0$.  The system of seminorms
    $\{\norm{\argument}_{m, \ell, \delta_w}\}_{w \in C_{n+1}^+}$ is
    equivalent to the system of seminorms $\{\norm{\argument}_{m,
      \ell, R}\}_{R > 0}$ where
    \begin{equation}
        \label{eq:NewSeminormsWithR}
        \norm{a}_{m, \ell, R}
        =
        \sqrt[2^m]{
          \sum_{\gamma = 0}^\infty
          \frac{R^\gamma}{\gamma!}
          h_{m, \ell, \gamma}(a)
          }.
    \end{equation}
    On those $a \in \prod_{I, J, \gamma} \mathbb{C}\basis{f}_{I, J,
      \gamma}$ with sub-factorial growth as in
    Lemma~\ref{lemma:GrowthhmlSubfactorialDisc} they take finite
    values. Moreover, it is equivalent to the system
    $\{\norm{\argument}_{m, \ell, \delta_w}\}$ where we only take $w
    \in C_{n+1}^+$ with $y(w) = 1$, provided $\hbar$ is an allowed
    value.
\end{lemma}
\begin{proof}
    Of course, we prove the mutual estimates as inequalities in $[0,
    +\infty]$ first and then deduce their finiteness. In fact,
    $\norm{a}_{m, \ell, R} < \infty$ is trivial for those $a$ with
    sub-factorial growth according to
    \eqref{eq:SubfactorialGrowthAlsoForhmlgamma}. To prove the
    equivalence of the two systems we first consider
    \begin{align*}
        \left(\norm{a}_{m, \ell, \delta_w}\right)^{2^m}
        &\le
        \sum_{(I, J, \gamma)}
        \frac{1}{\gamma!^2}
        \binom{\gamma}{\gamma - |I|}\frac{|I|!}{I!}
        \binom{\gamma}{\gamma - |J|}\frac{|J|!}{J!}
        \left|\Pochhammer{\frac{y(w)}{2\hbar}}_\gamma\right|
        \frac{\supnorm{w}^{2\gamma}}{y(w)^\gamma}
        h_{m, \ell, (I, J, \gamma)}(a) \\
        &\stackrel{\eqref{eq:EstimatePochhammerSymbols}}{\le}
        \sum_{\gamma = 0}^\infty
        \frac{1}{\gamma!}
        \mathsf{c}^\gamma
        (n+1)^{2\gamma}
        \frac{\supnorm{w}^{2\gamma}}{y(w)^\gamma}
        \sum_{|I|, |J| \le \gamma}
        h_{m, \ell, (I, J, \gamma)}(a) \\
        &=
        \sum_{\gamma = 0}^\infty
        \frac{R^\gamma}{\gamma!}
        h_{m, \ell, \gamma}(a),
    \end{align*}
    where we have used the standard estimate for the
    $(n+1)$-multinomial coefficients and where 
    \[
    R = \mathsf{c} (n+1)^2
    \frac{\supnorm{w}^2}{y(w)}
    \tag{$*$}
    \]
    with $\mathsf{c}$ begin the constant in
    \eqref{eq:EstimatePochhammerSymbols} for $z =
    \frac{y(w)}{2\hbar}$.  For the reverse estimate let $R > 0$ be
    given. Then we have for an arbitrary point $w$ with $-
    \frac{y(w)}{2\hbar} \not\in \mathbb{N}_0$
    \begin{align*}
        \sum_{\gamma = 0}^\infty
        \frac{R^\gamma}{\gamma!}
        h_{m, \ell, \gamma}(a)
        &\stackrel{\eqref{eq:EstimatePochhammerSymbols}}{\le}
        \frac{1}{\mathsf{a}}
        \sum_{\gamma = 0}^\infty
        \frac{R^\gamma}{\gamma!}
        \frac{1}{\mathsf{b}^\gamma}
        \left|
            \frac{1}{\gamma!}
            \Pochhammer{\frac{y(w)}{2\hbar}}_\gamma
        \right|
        \sum_{|I|, |J| \le \gamma}
        h_{m, \ell, (I, J, \gamma)}(a) \\
        &\le
        \frac{1}{\mathsf{a}}
        \sum_{\gamma = 0}^\infty
        \frac{R^\gamma}{\gamma!}
        \frac{1}{\mathsf{b}^\gamma}
        \left|
            \frac{1}{\gamma!}
            \Pochhammer{\frac{y(w)}{2\hbar}}_\gamma
        \right|
        \sum_{|I|, |J| \le \gamma}
        \binom{\gamma}{\gamma - |I|}\frac{|I|!}{I!}
        \binom{\gamma}{\gamma - |J|}\frac{|J|!}{J!}
        h_{m, \ell, (I, J, \gamma)}(a) \\
        &=
        \frac{1}{\mathsf{a}}
        \sum_{(I, J, \gamma)}
        \frac{1}{I!(\gamma - |I|)!J!(\gamma - |J|)!}
        \left|
            \Pochhammer{\frac{y(w)}{2\hbar}}_\gamma
        \right|
        \frac{R^\gamma}{\mathsf{b}^\gamma}
        h_{m, \ell, (I, J, \gamma)}(a).
    \end{align*}        
    We choose the point $w$ such that in addition to the requirement
    $- \frac{y(w)}{2\hbar} \not\in \mathbb{N}_0$ we have $|w^i| \ge
    \sqrt{\frac{R}{b y(w)}}$ for all $i = 0, \ldots, n$.  Note that we
    always can fulfill this additional requirement since we have
    points $w$ with arbitrarily large positive coefficients $w^i$ but
    fixed $y(w)$. Then we can estimate further and get
    \begin{align*}
        \sum_{\gamma = 0}^\infty
        \frac{R^\gamma}{\gamma!}
        h_{m, \ell, \gamma}(a)
        &\le
        \frac{1}{\mathsf{a}}
        \sum_{(I, J, \gamma)}
        \frac{1}{I!(\gamma - |I|)!J!(\gamma - |J|)!}
        \left|
            \Pochhammer{\frac{y(w)}{2\hbar}}_\gamma
        \right|
        \frac{
          |w^0|^{\gamma - |I|} |w|^I
          |\cc{w}^0|^{\gamma - |J|} |\cc{w}|^J
        }
        {y(w)^\gamma}
        h_{m, \ell, (I, J, \gamma)}(a) \\
        &=
        \frac{1}{\mathsf{a}}
        \norm{a}_{m, \ell, \delta_w}^{2^m}.
    \end{align*}
    If $\hbar$ is an allowed value this shows that even the points
    with $y(w) = 1$ will suffice.
\end{proof}
\begin{corollary}
    \label{corollary:LocallyUniformEstimate}%
    Let $\hbar \ne 0$.  The estimate of $\norm{a}_{m, \ell, \delta_w}$
    by some $\norm{a}_{m, \ell, R}$ is locally uniform in $w$, i.e.\
    for every compact subset $K \subseteq C_{n+1}^+$ there is a $R >
    1$ with
    \begin{equation}
        \label{eq:LocallyUniformEstimate}
        \sup_{w \in K} \norm{a}_{m, \ell, \delta_w}
        \le
        \norm{a}_{m, \ell, R}.
    \end{equation}
\end{corollary}
\begin{proof}
    Indeed, from ($*$) in the above proof this is clear by taking $R$
    to be the maximum value of $\mathsf{c} (n+1)^2 \supnorm{w}^2$ for
    $w \in K$.
\end{proof}

From now on we endow the subalgebra $\mathcal{A}_\hbar(C_{n+1}^+)$ of
$\prod_{I, J, \gamma} \mathbb{C}\basis{f}_{I, J, \gamma}$ with the
topology arsing from the additional seminorms $\norm{\argument}_{m,
  \ell, \delta_w}$ for $w \in C_{n+1}^+$ which we know to be finite
according to the lemma. Equivalently, we can add all the seminorms
$\norm{\argument}_{m, \ell, R}$ with $R > 0$.  The completion
$\complete{\mathcal{A}}_\hbar(C_{n+1}^+)$ is the locally convex
algebra according to the general construction from
Theorem~\ref{theorem:Omega} where we choose the set of all evaluation
functionals for the construction of the additional seminorms.

We collect these results and the explicit characterization of the
completion in the following theorem:
\begin{theorem}
    \label{theorem:AlgebraUpstairs}%
    Let $\hbar \ne 0$. Then $\mathcal{A}_\hbar(C_{n+1}^+)$ has the
    following properties:
    \begin{theoremlist}
    \item \label{item:FrakAhbarLC} $\mathcal{A}_\hbar(C_{n+1}^+)$ is a
        Hausdorff locally convex algebra.
    \item \label{item:DeltasContinuous} Every evaluation functional
        $\delta_w$ for all $w \in C_{n+1}^+$ is continuous.
    \item \label{item:FrakAStillFrechet} The completion
        $\complete{\mathcal{A}}_\hbar(C_{n+1}^+)$ is a Fréchet
        algebra.
    \item \label{item:FrakAIsSubfactorialStuff} The completion
        $\complete{\mathcal{A}}_\hbar(C_{n+1}^+)$ coincides with those
        $a \in \prod_{(I, J, \gamma)} \mathbb{C} \basis{f}_{I, J,
          \gamma}$ having sub-factorial growth
        \eqref{eq:aSubfactorialInGamma} with respect to $\gamma$. An
        equivalent system of seminorms is given by
        $\{\norm{\argument}_\epsilon\}_{0 < \epsilon < 1}$ where
        \begin{equation}
            \label{eq:YetAnotherEquivalentSystemOfSeminorms}
            \norm{a}_\epsilon
            =
            \sup_{(I, J, \gamma)}
            \frac{|a_{I, J, \gamma}|}{(\gamma!)^\epsilon}.
        \end{equation}
    \item \label{item:AlgebraObenIsKoetheSpace} As Fréchet space
        $\complete{\mathcal{A}}_\hbar(C_{n+1}^+)$ is isomorphic to a
        Köthe space of sub-factorial growth. It is strongly nuclear
        and the Schauder basis $\{\basis{f}_{I, J, \gamma}\}$ is
        absolute.
    \item \label{item:CompletionDiscIsContinuousFun} The completion
        $\complete{\mathcal{A}}_\hbar(C_{n+1}^+)$ can be continuously
        included into the continuous functions $C^0(C_{n+1}^+)$.
    \end{theoremlist}
\end{theorem}
\begin{proof}
    The first and second part is clear by construction and
    Lemma~\ref{lemma:EquivalentSeminormSystem}. By the same lemma we
    note that the topology is determined by countably many seminorms
    $\norm{\argument}_{m, \ell, R_n}$ with an increasing sequence $R_n
    \longrightarrow \infty$. Hence the completion is Fréchet. For the
    fourth part we already know that the elements with sub-factorial
    growth belong to the completion, this was also mentioned in
    Lemma~\ref{lemma:EquivalentSeminormSystem}. To show the converse,
    let $a \in \complete{\mathcal{A}}_\hbar(C_{n+1}^+)$ be given. From
    the estimate \eqref{eq:EstimateSumhmellSquares} we get by
    induction that
    \[
    |a_{I, J, \gamma}|^{2^m}
    \le
    (\gamma + 1)^{k_m} h_{m, 0, \gamma}(a)
    \]
    with some universal exponent $k_m > 0$ depending only on the
    dimension and $m$. Thus the definition of the seminorm
    $\norm{\argument}_{m, \ell, R}$ as in \eqref{eq:NewSeminormsWithR}
    gives the estimate
    \[
    |a_{I, J, \gamma}|^{2^m}
    \le
    (\gamma + 1)^{k_m}
    h_{m, 0, \gamma}(a)
    \le
    (\gamma + 1)^{k_m}
    \frac{\gamma!}{R^\gamma}
    \norm{a}_{m, 0, R}^{2^m},
    \tag{$*$}
    \]
    for all $R > 0$ and all $m \in \mathbb{N}_0$.  Since $(\gamma +
    1)^{k_m}$ and $R^{-\gamma}$ clearly grow sub-factorially we can
    take the $2^m$-th root of this inequality to see that $|a_{I, J,
      \gamma}|$ grows sub-factorially, too. Now we note that
    $\norm{a}_\epsilon$ is clearly a well-defined seminorm for those
    $a$ with sub-factorial growth.  Moreover, ($*$) shows that the
    seminorm $\norm{a}_\epsilon$ with $\epsilon = \frac{1}{2^m}$ can
    be estimated by $\norm{a}_{m, 0, R}$ for $R > 1$. Clearly, those
    $\epsilon$ are sufficient to conclude that the original topology
    is finer than the one induced by the system of all the
    $\norm{\argument}_\epsilon$. For the converse estimates one
    carefully examines the proof of
    Lemma~\ref{lemma:GrowthhmlSubfactorialDisc} to conclude that the
    constant $c_m$ in \eqref{eq:EstimateGrowthhml} can be chosen to be
    a numerical factor times the $2^m$-th power of
    $\norm{a}_{\epsilon'}$ for an appropriate $\epsilon' > 0$. From
    this we get the reverse estimate of $\norm{a}_{m, \ell, R}$ by
    some $\norm{a}_\epsilon$. It is also clear from general arguments
    as on a Fréchet space any coarser Fréchet topology is necessarily
    the same by the open mapping theorem. This proves the fourth
    part. Then the fifth part is a simple consequence of
    \eqref{eq:YetAnotherEquivalentSystemOfSeminorms} and the general
    facts from Appendix~\ref{sec:SubfactorialGrowth}.  The last part
    is clear by Corollary~\ref{corollary:LocallyUniformEstimate}.
\end{proof}
\begin{remark}
    \label{remark:NotSoTrivial}%
    While the first version gave us just the uninteresting Cartesian
    product topology, the second is now much more subtle. In
    particular, the precise form of the $h_{m, \ell, \gamma}$ enter
    the game as we still have to guarantee the continuity of the
    product via Theorem~\ref{theorem:Omega}.
\end{remark}

The functions in $\complete{\mathcal{A}}_\hbar(C_{n+1}^+)$ can be
characterized further: we consider the diagonal map
\begin{equation}
    \label{eq:DiagonalCnPlusEins}
    \Delta\colon
    C_{n+1}^+ \ni z
    \; \mapsto \;
    (z, z) \in C_{n+1}^+ \times C_{n+1}^+.  
\end{equation}
Moreover, denote by $\HolAntiHol(C_{n+1}^+ \times C_{n+1}^+)$ the
functions which are holomorphic in the first and anti-holomorphic in
the second variables. Then we have for every index triple $(P, Q,
\alpha)$ the function $\hat{\basis{e}}_{P, Q, \alpha} \in
\HolAntiHol(C_{n+1}^+ \times C_{n+1}^+)$ given by $\hat{\basis{e}}_{P,
  Q, \alpha}(u, v) = (u^0)^{\alpha - |P|} u^P (\cc{v}^0)^{\alpha -
  |Q|} \cc{v}^Q$ such that $\Delta^* \hat{\basis{e}}_{P, Q, \alpha} =
\basis{e}_{P, Q, \alpha}$. Similarly, $\hat{y}(u, v) = u^0\cc{v}^0 -
u^1 \cc{v}^1 - \cdots - u^n \cc{v}^n$ satisfies $\Delta^* \hat{y} =
y$. Since $\hat{y}$ is still different from $0$ on $C_{n+1}^+ \times
C_{n+1}^+$ we have also $\hat{\basis{f}}_{P, Q, \alpha} \in
\HolAntiHol(C_{n+1}^+ \times C_{n+1}^+)$ by the analogous formulas.
Thus every element $a \in \mathcal{A}_\hbar(C_{n+1}^+)$ has a
(necessarily unique) extension $\hat{a}$ to a function in
$\HolAntiHol(C_{n+1}^+ \times C_{n+1}^+)$, i.e. $\Delta^* \hat{a} =
a$. This still holds for the completion:
\begin{proposition}
    \label{proposition:HolAntiHolExtension}%
    Under the identification of
    $\complete{\mathcal{A}}_\hbar(C_{n+1}^+)$ with functions on
    $C_{n+1}^+$, any function $a \in
    \complete{\mathcal{A}}_\hbar(C_{n+1}^+)$ has an extension $\hat{a}
    \in \HolAntiHol(C_{n+1}^+ \times C_{n+1}^+)$. In particular, $a$
    is real-analytic on $C_{n+1}^+$.
\end{proposition}
\begin{proof}
    First we note that for a basis vector we get the estimate
    \[
    \left|
        \hat{\basis{f}}_{P, Q, \alpha}(u, v)
    \right|
    \le
    \frac{1}{\alpha!}
    (n+1)^{2\alpha}
    \mathsf{c}(u, v)^\alpha
    \frac{\supnorm{u}^\alpha \supnorm{v}^\alpha}
    {|\hat{y}(u, v)|^\alpha},
    \]
    where $\mathsf{c}(u, v)$ is the constant in the estimate
    \eqref{eq:EstimatePochhammerSymbols} of the Pochhammer symbol
    $\Pochhammer{\frac{\hat{y}(u, v)}{2\hbar}}_\gamma$. Since this
    constant can be chosen to be locally uniform, for every compact
    subset $K \subseteq C_{n+1}^+ \times C_{n+1}^+$ we can define
    \[
    R = 
    \max_{(u, v) \in K}
    (n+1)^2
    \mathsf{c}(u, v)
    \frac{\supnorm{u}\supnorm{v}}{|\hat{y}(u, v)|}.
    \]
    On $\mathcal{A}_\hbar(C_{n+1}^+)$ we can define the evaluation
    functional $\delta_{(u, v)}$ of $a$ as $\delta_{(u, v)}(a) =
    \hat{a}(u, v)$ for $u, v \in C_{n+1}^+$. We claim that this is
    continuous. Indeed, writing $a = \sum_{(P, Q, \alpha)} a_{P, Q,
      \alpha} \basis{f}_{P, Q, \alpha}$ we estimate
    \[
    \left|\delta_{(u, v)}(a)\right|
    \le
    \sum_{(P, Q, \alpha)}
    |a_{P, Q, \alpha}|
    \left|
        \hat{\basis{f}}_{P, Q, \alpha}(u, v)
    \right|
    \le
    \sum_{\alpha = 0}^\infty
    h_{0, 0, \alpha}(a)
    \frac{R^\alpha}{\alpha!}
    =
    \norm{a}_{0, 0, R}
    \tag{$*$}
    \]
    for all $(u, v) \in K$. Thus $\delta_{(u, v)}$ is continuous and
    extends continuously to $\complete{\mathcal{A}}_\hbar(C_{n+1}^+)$
    and hence we can define $\hat{a}$ pointwise as $\hat{a}(u, v) =
    \delta_{(u, v)}(a)$. Since $a = \sum_{(P, Q, \alpha)} a_{P, Q,
      \alpha} \basis{f}_{P, Q, \alpha}$ converges for every $a \in
    \complete{\mathcal{A}}_\hbar(C_{n+1}^+)$ we have, thanks to ($*$),
    a locally uniform approximation of $\hat{a}$ by elements of
    $\HolAntiHol(C_{n+1}^+ \times C_{n+1}^+)$. Hence also $\hat{a} \in
    \HolAntiHol(C_{n+1}^+ \times C_{n+1}^+)$.
\end{proof}

Note that we get a rather strong regularity property for the functions
$a \in \complete{\mathcal{A}}_\hbar(C_{n+1}^+)$: not every
real-analytic function on $C_{n+1}^+$ has such an extension. In
general, only an extension in a neighborhood of the diagonal is
possible. Moreover, from the last estimate ($*$) in the proof we note
that the topology of $\complete{\mathcal{A}}_\hbar(C_{n+1}^+)$ is
finer than the canonical topology of $\HolAntiHol(C_{n+1}^+ \times
C_{n+1}^+)$, i.e.\ the locally uniform one. Thus we get a nontrivial
refinement of Theorem~\ref{theorem:AlgebraUpstairs},
\refitem{item:CompletionDiscIsContinuousFun}. Finally, the
characterization of the growth as in
Theorem~\ref{theorem:AlgebraUpstairs},
\refitem{item:FrakAIsSubfactorialStuff}, gives another (and more
direct) way to prove the proposition, however, the continuity
statement of the evaluation functionals $\delta_{(u, v)}$ requires the
above argument.

%
%

\subsection{The algebra on the disk}
\label{subsec:AlgebraOnDisc}

In a last step we have to pass from $C_{n+1}^+$ to the disk
$\mathbb{D}_n$. In the formal power series setting the idea of
\cite{bordemann.brischle.emmrich.waldmann:1996a,
  bordemann.brischle.emmrich.waldmann:1996b} was to show that the
classical vanishing ideal of the hypersurface $y = 1$ coincides with
the two-sided ideal generated by $y - 1$ with respect to $\tildewick$,
simply because $\tildewick$ with a radial function is just the
pointwise product.

Before we can set $y = 1$ in our situation, we have to check a few
compatibilities: since we have now only a \emph{subalgebra} of
$C^\infty(C_{n+1}^+)[[\lambda]]$ it is not clear whether the ideal
generated by $y -1$ is still the vanishing ideal or not. Here we will
need that $\hbar$ is an allowed value:
\begin{lemma}
    \label{lemma:VanishingIdeal}%
    Let $\hbar$ be an allowed value.  In
    $\mathcal{A}_\hbar(C_{n+1}^+)$ the vanishing ideal of the
    hypersurface $y = 1$ coincides with the ideal generated by $y -
    1$.
\end{lemma}
\begin{proof}
    First we note that since $y-1$ is a radial function, the ideal
    generated by $y-1$ with respect to $\tildewick$ is just the
    classical ideal generated by $y-1$. Hence we only have to care
    about the pointwise product. Clearly, the ideal generated by $y-1$
    is contained in the vanishing ideal. To show the reverse inclusion
    we argue as follows: Let $\gamma$ be given and let $a \in
    \mathcal{A}^{(\gamma)}_\hbar(C_{n+1}^+)$ be in the vanishing
    ideal. Then $a = \sum_{(P, Q, \alpha)} a_{P, Q, \alpha}
    \basis{f}_{P, Q, \alpha}$ satisfies
    \[
    \sum_{(P, Q, \alpha)}
    a_{P, Q, \alpha}
    \frac{1}{P!(\alpha - |P|)!Q!(\alpha - |Q|)!}
    \Pochhammer{\frac{1}{2\hbar}}_\alpha
    \basis{e}_{P, Q, \alpha}(w)
    =
    0
    \]
    for every $w \in C_{n+1}^+$ with $y(w) = 1$. By assumption
    $\Pochhammer{\frac{1}{2\hbar}}_\alpha \ne 0$. This shows that the
    \emph{polynomial}
    \[
    \tilde{a}
    =
    \sum_{(P, Q, \alpha)}
    a_{P, Q, \alpha}
    \frac{1}{P!(\alpha - |P|)!Q!(\alpha - |Q|)!}
    \Pochhammer{\frac{1}{2\hbar}}_\alpha
    \basis{e}_{P, Q, \alpha}
    \]
    vanishes on the hypersurface $y = 1$. Since $\tilde{a}$ is a
    $\mathrm{U}(1)$-invariant polynomial of degree at most $2\gamma$
    we find a polynomial $\tilde{b}$ of degree at most $2(\gamma - 1)$
    with $\tilde{a} = (y-1)\tilde{b}$. Moreover, $\tilde{b}$ is
    \emph{unique} and still $\mathrm{U}(1)$-invariant. Denote the
    dimension of the space of $\mathrm{U}(1)$-invariant polynomials of
    degree at most $2\gamma$ by $n_\gamma$. Then this shows that the
    vanishing ideal intersected with
    $\mathcal{A}^{(\gamma)}_\hbar(C_{n+1}^+)$ is
    $n_{\gamma-1}$-dimensional. But also
    $\mathcal{A}^{(\gamma)}_\hbar(C_{n+1}^+)$ is
    $n_\gamma$-dimensional and the map $b \mapsto (y-1)b$ is still
    injective for $b \in \mathcal{A}^{(\gamma -
      1)}_\hbar(C_{n+1}^+)$. Hence the space of elements of the form
    $(y-1)b$ with $b \in \mathcal{A}^{(\gamma - 1)}_\hbar(C_{n+1}^+)$
    is $n_{\gamma - 1}$-dimensional, too. Since the latter is
    contained in the former, both have to coincide. Note however, that
    the map $a \mapsto \tilde{a}$ is not at all an algebra morphism,
    we only use it to compare the sizes of the ideals.
\end{proof}

This shows now that the ideal generated by $y-1$ is already a
\emph{closed} ideal in $\mathcal{A}_\hbar(C_{n+1}^+)$ since all the
evaluation functionals $\delta_w$ with $w$ in the hypersurface $y = 1$
are continuous. Hence the intersection of their kernels is closed in
the (noncomplete) algebra $\mathcal{A}_\hbar(C_{n+1}^+)$. Moreover,
after completion to $\complete{\mathcal{A}}_\hbar(C_{n+1}^+)$ the
closure of the vanishing ideal is still just the vanishing ideal,
i.e.\
\begin{equation}
    \label{eq:ClosureVanishingIdeal}
    \mathcal{J}_{y=1}
    =
    \left(
        (y - 1)\tildewick \mathcal{A}_\hbar(C_{n+1}^+)
    \right)^\cl
    =
    \left(
        \bigcap_{w} \ker \delta_w
    \right)^\cl
    =
    \bigcap_{w} \ker \complete{\delta_w}
    \subseteq
    \complete{\mathcal{A}}_\hbar(C_{n+1}^+) ,
\end{equation}
where $w$ runs through the points of the $y = 1$ hypersurface and
$\complete{\delta_w}$ denotes the extension of $\delta_w$ to the
completion. Thus the closure is still an ideal with respect to
$\tildewick$ and we can now divide by this ideal.  Since
$\complete{\mathcal{A}}_\hbar(C_{n+1}^+)$ has the interpretation of a
space of functions on $C_{n+1}^+$, the quotient by the vanishing ideal
has the interpretation of a space of functions on $\mathbb{D}_n$. This
will now be the motivation for the following definition:
\begin{definition}[The algebra
    $\complete{\mathcal{A}}_\hbar(\mathbb{D}_n)$] 
    \label{definition:AlgebraOnTheDisc}%
    Let $\hbar$ be an allowed value.  The quotient Fréchet algebra
    $\complete{\mathcal{A}}_\hbar(C_{n+1}^+) \big/ \mathcal{J}_{y=1}$
    is denoted by $\complete{\mathcal{A}}_\hbar(\mathbb{D}_n)$ and its
    multiplication will be denoted by $\stardisk$.
\end{definition}
Analogously, we set $\mathcal{A}_\hbar(\mathbb{D}_n) =
\mathcal{A}_\hbar(C_{n+1}^+) \big/ (y-1)\tildewick
\mathcal{A}_\hbar(C_{n+1}^+)$. Then the completion of
$\mathcal{A}_\hbar(\mathbb{D}_n)$ is indeed
$\complete{\mathcal{A}}_\hbar(\mathbb{D}_n)$ justifying our notation.

Since the ideal generated by $y-1$ is precisely the vanishing ideal,
we can just set $y$ equal to $1$ in all our above formulas to get the
corresponding formulas for the algebra on the disk. In particular, the
basis vectors $\basis{f}_{P, Q, \alpha}$ behave as follows: recall
that the canonical coordinates on the disk $\mathbb{D}_n$ are given by
$v^i = \frac{z^i}{z^0}$ for $i = 1, \ldots, n$. Then the function on
$\mathbb{D}_n$ corresponding to the equivalence class of
$\basis{f}_{P, Q, \alpha}$ is explicitly given by
\begin{equation}
    \label{eq:ClassOfBasisFunction}
    [\basis{f}_{P, Q, \alpha}](v)
    =
    \frac{1}{P!(\alpha - |P|)!Q!(\alpha - |Q|)!}
    \Pochhammer{\frac{1}{2\hbar}}_\alpha
    \frac{v^P \cc{v}^Q}{(1 - |v|^2)^\alpha},
\end{equation}
and their product is given by the very same formula
\eqref{eq:StructureConstantsTildeWick}. However, now these functions
are no longer linearly independent as we have used the relation $y =
1$. It was this class of functions for which the star product on the
disk was given in \cite[Sect.~4]{cahen.gutt.rawnsley:1994a}.
\begin{lemma}
    \label{lemma:BasisOnTheDisc}%
    Let $\hbar$ be an allowed value. Then the functions
    \begin{equation}
        \label{eq:BasisOnTheDisc}
        \basis{f}_{P, Q}(v)
        =
        [\basis{f}_{P, Q, \alpha}](v)
        =
        \frac{1}{P!(\alpha - |P|)!Q!(\alpha - |Q|)!}
        \Pochhammer{\frac{1}{2\hbar}}_\alpha
        \frac{v^P\cc{v}^Q}{(1 - |v|^2)^\alpha}
    \end{equation}
    for $P, Q \in \mathbb{N}_0^n$ and $\alpha = \max(|P|, |Q|)$ form a
    vector space basis of $\mathcal{A}_\hbar(\mathbb{D}_n)$.
\end{lemma}
\begin{proof}
    First we note the obvious relation $\frac{1}{1 - |v|^2} = 1 +
    \sum_{i=1}^n \frac{v^i \cc{v}^i}{1 - |v|^2}$ which gives for all
    $\alpha \ge 1$ the relation
    \[
    \frac{1}{\left(1 - |v|^2\right)^\beta}
    \frac{v^R \cc{v}^S}{(1 - |v|^2)^{\max(|R|, |S|)}}
    =
    \sum_{|I| = 0}^\beta
    \binom{\beta}{|I|}
    \frac{|I|!}{I!}
    \frac{v^{R+I} \cc{v}^{S+I}}
    {(1 - |v|^2)^{|I| + \max(|R|, |S|)}}.
    \tag{$*$}
    \]
    Clearly, each term on the right hand side is in the linear span of
    the functions $\basis{f}_{P, Q}$ showing that the linear span of
    the $\basis{f}_{P, Q}$ is all of
    $\mathcal{A}_\hbar(\mathbb{D}_n)$. To show their linear
    independence we have to use that $\hbar$ is an allowed value: we
    note that the asymptotic behavior of $\basis{f}_{P, Q}$ for $|v|
    \longrightarrow 1$ allows to recover $\alpha = \max(|P|,
    |Q|)$. Then the linear independence of the monomials $v^P
    \cc{v}^Q$ for fixed $\max(|P|, |Q|)$ gives immediately the linear
    independence of the $\basis{f}_{P, Q}$.
\end{proof}
Note finally that the non-allowed values of $\hbar$ would yield a
vanishing Pochhammer symbol in \eqref{eq:ClassOfBasisFunction} for all
large enough $\alpha$ and hence a finite-dimensional quotient, a case
which was discussed in detail in
\cite{bordemann.brischle.emmrich.waldmann:1996b}.

We summarize now some first properties of the algebra on the disk:
\begin{theorem}[The algebra on the disk]
    \label{theorem:DiscAlgebra}%
    Let $\hbar$ be an allowed value.
    \begin{theoremlist}
    \item \label{item:DiscAlgebraFrechet} The algebra
        $\complete{\mathcal{A}}_\hbar(\mathbb{D}_n)$ on the disk is a
        unital Fréchet algebra with respect to $\stardisk$.
    \item \label{item:AdiskInAhatdisk}
        $\mathcal{A}_\hbar(\mathbb{D}_n)$ is a dense subalgebra of
        $\complete{\mathcal{A}}_\hbar(\mathbb{D}_n)$.
    \item \label{item:DiscPointsContinuous} The evaluation functionals
        $\delta_v\colon \complete{\mathcal{A}}_\hbar(\mathbb{D}_n)
        \longrightarrow \mathbb{C}$ defined by
        \begin{equation}
            \label{eq:EvaluationOnDisc}
            \delta_v ([a]) = \delta_w(a),
        \end{equation}
        where $w \in C_{n+1}^+$ is a preimage of $v \in \mathbb{D}_n$
        with $y(w) = 1$, are continuous linear functionals.
    \item \label{item:DiscFunctionsRealAnalytic} The elements $[a] \in
        \complete{\mathcal{A}}_\hbar(\mathbb{D}_n)$ can be identified
        with certain real-analytic functions on $\mathbb{D}_n$ having
        an extension to $\mathbb{D}_n \times \mathbb{D}_n$ begin
        holomorphic in the first and anti-holomorphic in the second
        variables. The topology of
        $\complete{\mathcal{A}}_\hbar(\mathbb{D}_n)$ is finer than the
        locally uniform topology of $\HolAntiHol(\mathbb{D}_n \times
        \mathbb{D}_n)$.
    \item \label{item:SubalgebraOfDisc}
        $\mathcal{A}_\hbar(\mathbb{D}_n)$ inherits the filtration of
        $\mathcal{A}_\hbar(C_{n+1}^+)$ from
        Corollary~\ref{corollary:FilteredAlgebra}.
    \item \label{item:SchauderBasisForDisc} The functions
        $\basis{f}_{P, Q} \in
        \complete{\mathcal{A}}_\hbar(\mathbb{D}_n)$ form an
        unconditional Schauder basis.
    \item \label{item:SubfactorialGroth} A formal series $\sum_{P, Q}
        a_{P, Q} \basis{f}_{P, Q}$ belongs to
        $\complete{\mathcal{A}}_\hbar(\mathbb{D}_n)$ iff the
        coefficients $a_{P, Q}$ have sub-factorial growth with respect
        to $\max(|P|, |Q|)$.
    \item \label{item:SubfacSeminormsForDisc} An equivalent system of
        seminorms for $\complete{\mathcal{A}}_\hbar(\mathbb{D}_n)$ is
        given by
        \begin{equation}
            \label{eq:SubfacSeminormsDisc}
            \norm{[a]}_\epsilon = \sup_{P, Q}
            \frac{|a_{P, Q}|}{(\max(|R|, |S|)!)^\epsilon},
        \end{equation}
        with $0 < \epsilon < 1$. It follows that
        $\complete{\mathcal{A}}_\hbar(\mathbb{D}_n)$ is isomorphic to
        a Köthe space of sub-factorial growth. It is strongly nuclear
        and the Schauder basis $\{\basis{f}_{R, S}\}$ is absolute.
    \end{theoremlist}
\end{theorem}
\begin{proof}
    The first part is clear by abstract arguments: a Fréchet algebra
    modulo a closed two-sided ideal is again a Fréchet algebra. Note
    that it is important that we are in a Fréchet situation, otherwise
    quotients by closed subspaces might not be complete again. The
    second is also clear from general arguments on the compatibility
    of quotients and closures.  For $w \in C_{n+1}^+$ with $y(w) = 1$
    the evaluation functional $\delta_w$ is well-defined on the
    quotient since we divide by the intersection of all the kernels of
    these functionals. By the universal property of the locally convex
    quotient topology the resulting functional on
    $\complete{\mathcal{A}}_\hbar(\mathbb{D}_n)$ is still
    continuous. For the fourth part we note that the elements of
    $\mathcal{A}_\hbar(\mathbb{D}_n)$ have this property as they are
    linear combinations of the functions
    \eqref{eq:ClassOfBasisFunction} for which we have the extension,
    explicitly given by
    \[
    [\basis{f}_{P, Q, \alpha}](v, u)
    = 
    \frac{1}{P!(\alpha - |P|)!Q!(\alpha - |Q|)!}
    \Pochhammer{\frac{1}{2\hbar}}_\alpha
    \frac{v^P \cc{u}^Q}{(1 - v\cc{u})^\alpha}.
    \tag{$*$}
    \]
    First let $(z, w) \in C_{n+1}^+ \times C_{n+1}^+$ satisfy
    $\hat{y}(z, w) = 1$.  We claim that $[a] \mapsto \hat{a}(z, w)$ is
    well-defined for $[a] \in \mathcal{A}_\hbar(\mathbb{D}_n)$.
    Indeed, let $b \in \mathcal{A}_\hbar(C_{n+1}^+)$ be given. Then
    $\widehat{(y-1)b}(z, w) = (\hat{y}(z, w) - 1)\hat{b}(z, w) = 0$. Since
    we divide by the ideal generated by $y-1$, the evaluation is
    well-defined.  We have a section $\varphi$ of the projection $\pi
    \times \pi\colon C_{n+1}^+ \times C_{n+1}^+ \longrightarrow
    \mathbb{D}_n \times \mathbb{D}_n$ whose image lies in the
    hypersurface of $\hat{y} = 1$: there are many choices, one
    convenient possibility is
    \[
    \varphi(v, u)
    =
    \left(
        \left(\frac{1}{1-v\cc{u}}, \frac{v}{1 - v \cc{u}}\right),
        (1, u)
    \right).
    \tag{$**$}
    \]
    Moreover, we have the following consistency $[\basis{f}_{P, Q,
      \alpha}](v, u) = \hat{\basis{f}}_{P, Q, \alpha}(\varphi(v, u))$,
    which is checked immediately from ($*$) and ($**$). Thus we get
    for $a \in \mathcal{A}_\hbar(C_{n+1}^+)$ the relation
    \[
    [a](v, u) = \hat{a}(\varphi(v, u)).
    \]
    Now let $K \subseteq \mathbb{D}_n \times \mathbb{D}_n$ be a
    compact subset. Then $\varphi(K)$ is compact and contained in the
    hypersurface of $\hat{y} = 1$. For $(v, u) \in K$ we get the
    estimate
    \[
    \left|[a](v, u)\right|
    =
    \left|a(\varphi(v, u))\right|
    \le
    \norm{a}_{0,0, R},
    \]
    where $R > 0$ is the parameter corresponding to the compact subset
    $\varphi(K)$ as in the proof of
    Proposition~\ref{proposition:HolAntiHolExtension}.  Since this
    holds for all representatives $a$ of $[a]$, we get the estimate
    also for the corresponding seminorm on the quotient
    $\mathcal{A}_\hbar(\mathbb{D}_n)$. This shows that the evaluation
    at $(v, u)$ is continuous with a locally uniform estimate
    concerning the point $(v, u)$. Using the second part, the fourth
    part follows. The filtration property is clear. Now consider the
    vector space basis $\basis{f}_{P, Q}$ from
    Lemma~\ref{lemma:BasisOnTheDisc}. We have to show that the
    evaluation functionals with respect to this basis are
    continuous. Thus let $a = \sum_{(P, Q, \alpha)} a_{P, Q, \alpha}
    \basis{f}_{P, Q, \alpha} \in
    \complete{\mathcal{A}}_\hbar(C_{n+1}^+)$ be given. For its
    equivalence class $[a] \in
    \complete{\mathcal{A}}_\hbar(\mathbb{D}_n)$ we get by a
    straightforward computation
    \begin{align*}
        [a]
        &=
        \sum_{R, S}
        \sum_{(P, Q, \alpha)} a_{P, Q, \alpha}
        \sum_{|I| = 0}^{\alpha - \max(|P|, |Q|)}
        \frac{
          (P + I)!(\max(0, |Q| - |P|))!
          (Q+I)! (\max(0, |P| - |Q|))!
        }
        {
          P!(\alpha - |P|)!
          Q!(\alpha - |Q|)!
        }
        \\
        &\quad
        \frac{\Pochhammer{\frac{1}{2\hbar}}_\alpha}
        {\Pochhammer{\frac{1}{2\hbar}}_{|I| + \max(|P|, |Q|)}}
        \frac{(\alpha - \max(|P|, |Q|))!}
        {(\alpha - \max(|P|, |Q|) - |I|)! I!}
        \delta_{P + I, R}
        \delta_{Q + I, S}
        \basis{f}_{R, S},
    \end{align*}
    using ($*$) from the proof of Lemma~\ref{lemma:BasisOnTheDisc}.
    For a given pair $(R, S)$, only those $P$, $Q$ can contribute
    where we have an $I$ with $P + I = R$ and $Q + I = S$. Moreover,
    $\alpha$ has to satisfy $\alpha \ge \max(|R|, |S|)$. Finally, $I
    \le \min(R, S)$ has to hold as well. For abbreviation we set $M =
    \max(|R|, |S|)$ and $m = \min(|R|, |S|)$.  This allows to compute
    the coefficient $a_{R, S}$ of the basis vector $\basis{f}_{R, S}$
    as follows
    \begin{align*}
        a_{R, S}
        &=
        \sum_{\alpha \ge M}
        \sum_{I \le \min(R, S)}
        a_{R - I, S - I, \alpha}
        \binom{R}{I}\binom{S}{I}
        \frac{I!(M - m)!I!}
        {(\alpha - M + |I|)!(\alpha - m + |I|)!}
        \frac{\Pochhammer{\frac{1}{2\hbar}}_\alpha}
        {\Pochhammer{\frac{1}{2\hbar}}_{M}}
        \frac{(\alpha + |I| - M)!}{(\alpha - M)!I!} \\
        &=
        \sum_{\alpha \ge M}
        \sum_{I \le \min(R, S)}
        a_{R - I, S - I, \alpha}
        \binom{R}{I}\binom{S}{I}
        \frac{I! (M- m)! \alpha!}
        {(\alpha - m + |I|)! M! (\alpha - M)!}
        \frac{
          \frac{1}{\alpha!}\Pochhammer{\frac{1}{2\hbar}}_\alpha
        }
        {
          \frac{1}{M!}\Pochhammer{\frac{1}{2\hbar}}_M
        }.
    \end{align*}
    We claim that this converges for all $a \in
    \complete{\mathcal{A}}_\hbar(C_{n+1}^+)$. Indeed, we know that
    $a_{P, Q, \alpha}$ has sub-factorial growth in $\alpha$. Hence let
    $\epsilon > 0$ and thus $|a_{R - I, S - I, \alpha}| \le c_1
    (\alpha!)^\epsilon$ for some appropriate $c_1$. Using the
    estimates for the Pochhammer symbols according to
    \eqref{eq:EstimatePochhammerSymbols} we get
    \begin{align*}
        |a_{R, S}|
        &\le
        \sum_{\alpha = M}^\infty
        \sum_{I \le \min(R, S)}
        c_1 (\alpha!)^\epsilon
        2^{M + m}
        \frac{I!(M - m)!\alpha!}
        {(\alpha - m + |I|)!M!(\alpha - M)!}
        c_2^\alpha c_3^M \\
        &=
        \sum_{\alpha = M}^\infty
        \sum_{I \le \min(R, S)}
        c_1 (\alpha!)^\epsilon
        2^{M + m}
        \frac{I!}{|I|!}
        \frac{1}{\binom{\alpha - m + |I|}{|I|}}
        \frac{1}{\binom{M}{m}}
        \binom{\alpha}{m}
        \frac{1}{(\alpha - M)!}
        c_2^\alpha c_3^M \\
        &\le
        \sum_{\alpha = M}^\infty
        \sum_{I \le \min(R, S)}
        c_1 (\alpha!)^\epsilon
        c_4^M c_5^\alpha
        \frac{1}{(\alpha - M)!} \\
        &\le
        c_1 c_6^M
        (M!)^\epsilon
        \sum_{\alpha = M}^\infty
        \binom{\alpha}{M}^\epsilon 
        \frac{c_5^\alpha}{(\alpha - M)!^{1-\epsilon}} \\
        &\le
        c_7 c_8^M (M!)^\epsilon,
        \tag{$\star$}
    \end{align*}
    with corresponding appropriate constants $c_1, \ldots, c_8$.  This
    shows two things: first the series converges (absolutely) and
    hence the linear functional $a \mapsto a_{R, S}$ is a weakly$^*$
    convergent series of continuous linear functionals on
    $\complete{\mathcal{A}}_\hbar(C_{n+1}^+)$. Second, the resulting
    values $a_{R, S}$ have a sub-factorial growth in $M = \max(|R|,
    |S|)$. Since we are in a Fréchet situation, the resulting linear
    functionals are continuous. By the very definition of the quotient
    topology, also the induced functionals $[a] \mapsto a_{R, S}$ are
    continuous on $\complete{\mathcal{A}}_\hbar(\mathbb{D}_n)$. Now
    let $a_{P, Q}$ have sub-factorial growth with respect to
    $\max(|P|, |Q|)$ then set $a_{(P, Q, \alpha)} = a_{P, Q}$ for
    $\alpha = \max(|P|, |Q|)$ and $0$ else. By
    Theorem~\ref{theorem:AlgebraUpstairs},
    \refitem{item:FrakAIsSubfactorialStuff} and
    Theorem~\ref{theorem:Omega},
    \refitem{item:UnconditionalSchauderBasisSecondVersion} we have an
    unconditionally convergent $a = \sum_{(P, Q, \alpha)} a_{P, Q,
      \alpha} \basis{f}_{P, Q, \alpha}$ in
    $\complete{\mathcal{A}}_\hbar(C_{n+1}^+)$. But this shows that
    also $[a] = \sum_{P, Q} a_{P, Q} \basis{f}_{P, Q}$ converges
    unconditionally, proving the sixth part. The next part was already
    shown. Since we can take for $[a] = \sum_{P, Q} a_{P, Q}
    \basis{f}_{P, Q}$ a representative $a = \sum_{(P, Q, \alpha)}
    a_{P, Q, \alpha} \basis{f}_{P, Q, \alpha}$ with $a_{P, Q, \alpha}
    = a_{P, Q}$ for $\alpha = \max(|P|, |Q|)$ and zero elsewhere, the
    estimate $\norm{[a]}_\epsilon \le \norm{a}_\epsilon$ is
    immediate. This shows that the seminorms of
    \eqref{eq:SubfacSeminormsDisc} are continuous seminorms on
    $\complete{\mathcal{A}}_\hbar(\mathbb{D}_n)$. Thanks to the
    seventh part, they make
    $\complete{\mathcal{A}}_\hbar(\mathbb{D}_n)$ a Fréchet space as
    well implying that the two topologies have to coincide. Then the
    last part is an easy consequence as we have argued already several
    times.
\end{proof}

Note that even though we have a vector space basis also for
$\mathcal{A}_\hbar(\mathbb{D}_n)$ the corresponding structure
constants are not as easily described as the ones of
$\mathcal{A}_\hbar(C_{n+1}^+)$. The reason is the quotient procedure
by $\mathcal{J}_{y=1}$. In particular, using the explicit formula for
the product of $\basis{f}_{P, Q}$ with $\basis{f}_{R, S}$ according to
\eqref{eq:StructureConstantsTildeWick} by evaluating $\basis{f}_{P, Q,
  \max(|P|, |Q|)} \tildewick \basis{f}_{R, S, \max(|R|, |S|)}$ for $y
= 1$ gives a \emph{redundant} description which has first to be
expressed again in terms of the linearly independent $\basis{f}_{I,
  J}$ alone and not in terms of the $[\basis{f}_{I, J, \gamma}]$.

Note also that parts of the last statement can be obtained more
directly: in particular, $\complete{\mathcal{A}}_\hbar(\mathbb{D}_n)$
is a quotient of a nuclear space by a closed subspace and hence
nuclear itself. However, the above Schauder basis allows to give an
explicit description of the Fréchet space structure of
$\complete{\mathcal{A}}_\hbar(\mathbb{D}_n)$ as a Köthe space of
sub-factorial growth.

%
%

\section{Further properties of $\mathcal{A}_\hbar(\mathbb{D}_n)$}
\label{sec:FurtherProperties}

In this concluding section we collect some further properties of the
algebra on the Poincaré disk $\mathbb{D}_n$: we discuss the symmetry
inherited from the classical symmetry \eqref{eq:DiscAsSymmetricSpace},
the dependence on the deformation parameter $\hbar$, and the
$^*$-involution.

%
%

\subsection{The $\mathrm{SU}(1, n)$-symmetry}
\label{subsec:SUEinsnSymmetrie}

Recall that on the level of formal star products $\tildewick$ was
invariant under the canonical $\mathrm{SU}(1, n)$-action, allowing
even an equivariant quantum momentum map, see
Remark~\ref{remark:SymmetryOfTildewick}. In fact, the momentum map is
part of $\mathcal{A}_\hbar(C_{n+1}^+)$:
\begin{lemma}
    \label{lemma:MomentumMapInA}%
    For all $\xi \ni \mathfrak{su}(1, n)$ we have $J_\xi \in
    \mathcal{A}_\hbar^1(C_{n+1}^+)$.
\end{lemma}
\begin{proof}
    Since $Sy = y$ this follows immediately from the definition of $J$
    in \eqref{eq:SUEinsnMomentumMap}.
\end{proof}
Thus on $\mathcal{A}_\hbar(C_{n+1}^+)$ we have an \emph{inner} action
of $\mathfrak{su}(1, n)$ via the inner derivations
$\frac{1}{\I\hbar}[J_\xi, \argument]_{\tildewick}$. From the
filtration property \eqref{eq:Filtered} we cannot expect the spaces
$\mathcal{A}_\hbar^{(\gamma)}(C_{n+1}^+)$ to be
invariant. Nevertheless, the strong invariance of $\tildewick$
together with the fact that $\Lie_\xi y = 0$ for all $\xi \in
\mathfrak{su}(1, n)$ shows that even the spaces
$\mathcal{A}_\hbar^\gamma(C_{n+1}^+)$ are invariant. We have
\begin{equation}
    \label{eq:LieXiAgammaInv}
    \frac{1}{\I\hbar}[J_\xi, \argument]_{\tildewick}
    \mathcal{A}_\hbar^\gamma(C_{n+1}^+)
    =
    \Lie_\xi
    \mathcal{A}_\hbar^\gamma(C_{n+1}^+)
    \subseteq
    \mathcal{A}_\hbar^\gamma(C_{n+1}^+)
\end{equation}
for all $\gamma$ and all $\xi \in \mathfrak{su}(1, n)$.  This
observation is one of the motivations for our choice of the basis of
the $\basis{f}_{I, J, \gamma}$. The next lemma gives an integrated
version of this:
\begin{lemma}
    \label{lemma:fIJgammaRepresentativeFunctions}%
    Let $\hbar$ be an allowed value and $\gamma \in \mathbb{N}_0$.
    \begin{lemmalist}
    \item \label{item:AgammaInvariant}
        $\mathcal{A}_\hbar^\gamma(C_{n+1}^+)$ is invariant under
        $\mathrm{SU}(1, n)$ and thus defines a finite-dimensional
        representation.
    \item \label{item:BoundsOnRep} For $U \in \mathrm{SU}(1, n)$ the
        representation on $\mathcal{A}_\hbar^\gamma(C_{n+1}^+)$ has
        matrix coefficients
        \begin{equation}
            \label{eq:MatrixCoeffsSUEinsn}
            U^*\basis{f}_{I, J, \gamma}
            =
            \sum_{K, L} M_{IJ}^{KL}(U, \gamma)
            \basis{f}_{K, L, \gamma},
        \end{equation}
        which satisfy an estimate of the form
        \begin{equation}
            \label{eq:EstimateMIJLKAgamma}
            \left|
                M_{IJ}^{KL}(U, \gamma)
            \right|
            \le
            (n+1)^{2\gamma} \supnorm{U}^{2\gamma}
        \end{equation}
        for all multiindices and $U \in \mathrm{SU}(1, n)$ where
        $\supnorm{\argument}$ denotes the supremum norm of the matrix
        elements.
    \end{lemmalist}
\end{lemma}
\begin{proof}
    The first statement is an easy consequence of the fact that the
    monomial $\basis{e}_{I, J, \gamma}$ is transformed into a linear
    combination of $\basis{e}_{K, L, \gamma}$ for all allowed $K, L$
    but the same $\gamma$. Since the function $y$ is invariant, the
    claim follows at once. For the second part, we consider again the
    monomials $\basis{e}_{I, J, \gamma}$. Expanding the summations in
    $U^*\basis{e}_{I, J, \gamma}$ by means of the multinomial theorem
    gives immediately the result, since the additional prefactor in
    $\basis{f}_{I, J, \gamma}$ only depends on $y$ and is not changed
    by the invariance of $y$.
\end{proof}

With other words, the functions in $\mathcal{A}_\hbar(C_{n+1}^+)$
consist of \emph{representative functions}, i.e.\ functions which have
a finite-dimensional orbit under the group action of $\mathrm{SU}(1,
n)$. Since we already know, see
Remark~\ref{remark:SymmetryOfTildewick}, that $\mathrm{SU}(1, n)$ acts
by \emph{automorphisms} of $\tildewick$ on the whole space
$C^\infty(C_{n+1}^+)[[\lambda]]$ we get an action by automorphisms on
$\mathcal{A}_\hbar(C_{n+1}^+)$, too.

In a next step we show that the action of $\mathrm{SU}(1, n)$ is
continuous. The key observation is contained in the following lemma:
\begin{lemma}
    \label{lemma:ContinuityOfSUEinsn}%
    Let $m \in \mathbb{N}_0$. Then there exists a constant $c_m > 1$
    such that for all $U \in \mathrm{U}(1, n)$, all $a \in
    \mathcal{A}_\hbar(C_{n+1}^+)$, all $\gamma \in \mathbb{N}_0$, and
    all $\ell = 0, \ldots 2^m - 1$ we have
    \begin{equation}
        \label{eq:ContinuityEstimateSUEinsn}
        \norm{U^*a}_{m, \ell, \gamma}
        \le
        c_m^\gamma
        \norm{U}_\infty^{2\gamma}
        \norm{a}_{m, \ell, \gamma}.
    \end{equation}
\end{lemma}
\begin{proof}
    As usual we prove this by induction on $m$. For $m = 0$ we have
    \[
    h_{0, 0, \gamma}(U^*a)
    =
    \sum_{I, J} h_{0, 0, (I, J, \gamma)}(U^*a)
    \le
    \sum_{I, J}
    \left|
        \sum_{K, L}
        a_{K, L, \gamma} M_{IJ}^{KL}(U, \gamma)
    \right|
    \stackrel{\eqref{eq:EstimateMIJLKAgamma}}{\le}
    h_{0, 0, \gamma}(a)
    \sum_{I, J} (n+1)^{2\gamma} \supnorm{U}^{2\gamma}.
    \]
    Since the remaining sum has at most $(n+1)^{2\gamma}$ terms, we
    get \eqref{eq:ContinuityEstimateSUEinsn} by taking $c_0 =
    (n+1)^4$. By induction, we estimate
    \begin{align*}
        h_{m+1, 2\ell, \gamma}(U^*a)
        &=
        \sum_{I, J} \sum_{(P, Q, \alpha)}
        h_{m, \ell, (P, Q, \alpha)}(U^*a)^2
        C^{(I, J, \gamma)}_{(P, Q, \alpha), \boldsymbol{\cdot}} \\
        &\stackrel{\mathclap{\eqref{eq:EstimateDiscConstants}}}{\le}
        \quad
        (\gamma + 1)^{4n+1} 4^\gamma
        \sum_{(P, Q, \alpha)}
        h_{m, \ell, (P, Q, \alpha)}(U^*a)^2 \\
        &\stackrel{\mathclap{\eqref{eq:SumPQhmlPAaSquares}}}{\le}
        \quad
        (\gamma + 1)^{4n+1} 4^\gamma
        \sum_{\alpha = 0}^\gamma
        h_{m, \ell, \alpha}(U^*a)^2 \\
        &\stackrel{\mathclap{\textrm{Ind.}}}{\le}
        \quad
        (\gamma + 1)^{4n+1} 4^\gamma
        \sum_{\alpha = 0}^\gamma
        c_m^{2^{m+1}\alpha} \supnorm{U}^{2^{m+2}\alpha}
        h_{m, \ell, \alpha}(a)^2 \\
        &\stackrel{\mathclap{\eqref{eq:EstimateSumhmellSquares}}}{\le}
        \quad
        (\gamma + 1)^{6n+1} 4^\gamma
        c_m^{2^{m+1}\gamma} \supnorm{U}^{2^{m+2}\gamma}
        h_{m+1, 2\ell, \gamma}(a),
    \end{align*}
    where we use $\supnorm{U} \ge 1$ for $U \in \mathrm{SU}(1, n)$ and
    $c_m > 1$ in the last step. Defining $c_{m+1}$ appropriately and
    taking now the $2^{m+1}$-th root gives the claim. The case $2\ell
    + 1$ is analogous. Since there are only finitely many $\ell$, we
    can take $c_{m+1}$ independently of $\ell$.
\end{proof}

Since the action of $\mathrm{SU}(1, n)$ preserves the filtration by
$\gamma$ and since each $\mathcal{A}_\hbar^\gamma(C_{n+1}^+)$ is
finite-dimensional, it is rather obvious that the action is via
continuous maps with respect to the \emph{Cartesian product topology},
i.e.\ the one determined by the seminorms $\norm{\argument}_{m, \ell,
  \gamma}$. From this point of view,
Lemma~\ref{lemma:ContinuityOfSUEinsn} is not surprising. However, the
precise estimate \eqref{eq:ContinuityEstimateSUEinsn} becomes crucial
for the following result:
\begin{lemma}
    \label{lemma:SUEinsnStetigerAuto}%
    For every $U \in \mathrm{SU}(1, n)$ the automorphism $U^*\colon
    \mathcal{A}_\hbar(C_{n+1}^+) \longrightarrow
    \mathcal{A}_\hbar(C_{n+1}^+)$ is continuous. More precisely, for
    all $R > 0$ we have
    \begin{equation}
        \label{eq:EstimateContinuityOfU}
        \norm{U^*a}_{m, \ell, R}
        \le
        \norm{a}_{m, \ell, R'}
    \end{equation}
    with $R' = c_m^{2^m} \supnorm{U}^{2^{m+1}} R$.
\end{lemma}
\begin{proof}
    Since the seminorms $\norm{\argument}_{m, \ell, R}$ specify the
    topology of $\mathcal{A}_\hbar(C_{n+1}^+)$ by
    Lemma~\ref{lemma:EquivalentSeminormSystem}, we only have to check
    the estimate. We have
    \[
    \sum_{\gamma = 0}^\infty
    \frac{R^\gamma}{\gamma!} h_{m, \ell, \gamma}(U^*a)
    \;
    \stackrel{\eqref{eq:ContinuityEstimateSUEinsn}}{\le}
    \;
    \sum_{\gamma = 0}^\infty
    \frac{
      \left(
          R c_m^{2^m} \supnorm{U}^{2^{m+1}}
      \right)^\gamma
    }
    {\gamma!}
    h_{m, \ell, \gamma}(a).
    \]
    Then taking the $2^m$-th root gives
    \eqref{eq:EstimateContinuityOfU}.
\end{proof}
Thus the action extends uniquely to an action by automorphisms on the
completion $\complete{\mathcal{A}}_\hbar(C_{n+1}^+)$. It turns out to
be a continuous action:
\begin{proposition}
    \label{proposition:ActionIsContinuous}%
    The action of $\mathrm{SU}(1, n)$ on
    $\complete{\mathcal{A}}_\hbar(C_{n+1}^+)$ yields a continuous map
    \begin{equation}
        \label{eq:ActionIsContinuous}
        \mathrm{SU}(1, n) \times
        \complete{\mathcal{A}}_\hbar(C_{n+1}^+)
        \longrightarrow
        \complete{\mathcal{A}}_\hbar(C_{n+1}^+).
    \end{equation}
\end{proposition}
\begin{proof}
    First we note that a sequence $U_j \in \mathrm{SU}(1, n)$
    converges to $U \in \mathrm{SU}(1, n)$ iff $\supnorm{U_j - U}
    \longrightarrow 0$, even though $U_j - U$ is of course not in
    $\mathrm{SU}(1, n)$ in general but just in $M_{n+1}(\mathbb{C})$.
    Moreover, since $\complete{\mathcal{A}}_\hbar(C_{n+1}^+)$ is a
    Fréchet space it suffices to consider sequences instead of
    nets. Thus let $U_j \longrightarrow U$ and $a_j \longrightarrow a$
    be convergent sequences in $\mathrm{SU}(1, n)$ and
    $\complete{\mathcal{A}}_\hbar(C_{n+1}^+)$, respectively.  Since
    $\supnorm{\argument}$ is continuous we have a constant $c > 0$
    with $\supnorm{U_j}, \supnorm{U} \le c$ for all $j$. Now let $m$,
    $\ell$, and $R$ be given and define $R' = c_m^{2^m} c^{2^{m+1}} R$
    such that we can apply Lemma~\ref{lemma:SUEinsnStetigerAuto} for
    all $U_j$ and $U$ simultaneously. Then we first note the estimate
    \[
    \norm{U_j^* a_j - U^* a}_{m, \ell, R}
    \le
    \norm{U_j^* a_j - U_j^* a}_{m, \ell, R}
    +
    \norm{U_j^* a - U^* a}_{m, \ell, R}
    \le
    \norm{a_j - a}_{m, \ell, R'}
    +
    \norm{U_j^* a - U^* a}_{m, \ell, R}.
    \]
    Now for $j \longrightarrow \infty$, the first contribution
    converges to $0$. Hence we only have to take care of the second.
    First we note that from $(\alpha + \beta)^k \le 2^{k-1}(\alpha^k +
    \beta^k)$ for $\alpha, \beta \ge 0$ and all positive $k$ we get
    \[
    h_{m, \ell, \gamma}(U_j^* a - U^* a)
    =
    \norm{U_j^*a - U^*a}_{m, \ell, \gamma}^{2^m}
    \le
    2^{2^m -1}
    \left(
        \norm{U_j^* a}_{m, \ell, \gamma}^{2^m}
        +
        \norm{U^* a}_{m, \ell, \gamma}^{2^m}
    \right)
    \le
    \tilde{c}^\gamma \norm{a}_{m, \ell, \gamma}^{2^m},
    \]
    with $\tilde{c} > 0$ build out of the constants in
    \eqref{eq:ContinuityEstimateSUEinsn} and the above $c$ estimating
    the sup-norms of $U_j$ and $U$. Thus the series $\sum_{\gamma =
      0}^\infty \frac{R^\gamma}{\gamma!} h_{m, \ell, \gamma}(U_j^* a -
    U^*a)$ can be dominated by the series $\sum_{\gamma = 0}^\infty
    \frac{(\tilde{c}R)^\gamma}{\gamma!} h_{m, \ell, \gamma}(a)$. Since
    clearly the group action is strongly continuous with respect to
    the single (semi-) norm $\norm{\argument}_{m, \ell, \gamma}$, we
    have $h_{m, \ell, \gamma}(U_j^*a - U^*a) \longrightarrow 0$ for $j
    \longrightarrow \infty$.  Both arguments together allow us to
    exchange the limit $j \longrightarrow \infty$ with the summation
    over $\gamma$ and we get $\norm{U_j^*a - U^*a}_{m, \ell, \gamma}
    \longrightarrow 0$ as wanted. This establishes the sequential
    continuity of \eqref{eq:ActionIsContinuous} at the point $(U, a)$
    which is all we have to prove.
\end{proof}

Note that though the seminorms $\norm{\argument}_{m, \ell, \gamma}$
just specify the Cartesian product topology for which the group action
is continuous for rather obvious reasons, the continuity with respect
to the topology of $\complete{\mathcal{A}}_\hbar(C_{n+1}^+)$ uses the
specific form of the $h_{m, \ell, \gamma}$ and is therefore
nontrivial. Having the continuity it is now fairly easy to see that
the action is even smooth: we can rely on some general arguments for
this. The following lemma should be well-known.
\begin{lemma}
    \label{lemma:CinftyAction}%
    Let $G$ be a connected Lie group acting on a sequentially complete
    Hausdorff locally convex space $V$ by continuous endomorphism such
    that the action map $G \times V \longrightarrow V$ is
    continuous. Suppose that for every $\xi \in \mathfrak{g}$ we are
    given a continuous operator $L_\xi\colon V \longrightarrow V$ such
    that $\xi \mapsto L_\xi$ is linear.  Finally, suppose $V_0
    \subseteq V$ is a dense subspace invariant under $G$ and under all
    the operators $L_\xi$ such that for $v \in V_0$ one has
    \begin{equation}
        \label{eq:DiffbarOnVnull}
        L_\xi v
        = 
        \lim_{t \longrightarrow 0}
        \frac{\exp(t\xi)v - v}{t}
    \end{equation}
    in the topology of $V$. Then the action is smooth and the smooth
    topology of $V$ coincides with the original one.
\end{lemma}
\begin{proof}
    The crucial point is that we assume that $L_\xi$ is
    \emph{continuous}. Let $\mathsf{p}$ be a continuous seminorm on
    $V$ and let $\xi \in \mathfrak{g}$. Since $L_\xi$ is continuous we
    get a continuous seminorm $\mathsf{q}$ such that $\mathsf{p}(L_\xi
    v) \le \mathsf{q}(v)$ for all $v \in V$. Moreover, we have a
    continuous function $t \mapsto \exp(t\xi)v$. For $v \in V_0$ we
    get for $t \ne 0$
    \begin{align*}
        \mathsf{p}
        \left(
            \frac{\exp(t\xi)v - v}{t} - L_\xi v
        \right)
        &=
        \mathsf{p}
        \left(
            \int_0^1
            \frac{\D}{\D s} \exp(ts\xi)v \D s
            - L_\xi v
        \right) \\
        &=
        \mathsf{p}
        \left(
            \int_0^1
            \frac{\D}{\D \epsilon}\At{\epsilon = 0}
            \exp(t\epsilon\xi) \exp(ts\xi) v \D s
            - L_\xi v
        \right) \\
        &=
        \mathsf{p}
        \left(
            \int_0^1
            L_\xi
            \exp(t\epsilon\xi) \exp(ts\xi) v \D s
            - L_\xi v
        \right) \\
        &=
        \mathsf{p}
        \left(
            L_\xi
            \int_0^1
            \left(
                \exp(ts\xi) v - v
            \right)
            \D s
        \right) \\
        &\le
        \int_0^1
        \mathsf{q}
        \left(
            \exp(ts\xi) v - v
        \right)
        \D s,
    \end{align*}
    where we have used the invariance of $V_0$ and the existence of
    the derivative on $V_0$. Note that a naive Riemann integral will
    do the job for the above estimates, hence sequential completeness
    is all we have to require here. Since now $L_\xi$ is continuous
    and hence both sides are a continuous expression in $v$, the
    estimate is still true for \emph{all} vectors $v \in V$ by $V_0$
    being dense, i.e.\ we get
    \[
    \mathsf{p}
    \left(
        \frac{\exp(t\xi)v - v}{t} - L_\xi v
    \right)
    \le
    \int_0^1
    \mathsf{q}
    \left(
        \exp(ts\xi) v - v
    \right)
    \D s.
    \]
    The continuity of the integrand allows to take the limit $t
    \longrightarrow 0$ yielding $0$ on the right hand side. Thus
    \eqref{eq:DiffbarOnVnull} holds for \emph{all} $v \in V$.  Using
    the action property we see that $\frac{\D}{\D t} \exp(t\xi)v =
    L_\xi \exp(t\xi)v = \exp(t\xi) L_\xi v$ holds for all $t$. Hence
    all first directional derivatives exist and are continuous in a
    neighborhood of $e \in G$ which implies that the action is $C^1$
    everywhere.  Since the operators $L_\xi$ are defined on all of $V$
    we can iterate this argument now to conclude that the action is
    $C^\infty$. Finally, the additional seminorms of the
    $C^\infty$-topology on $V$ are given by $v \mapsto
    \mathsf{p}\left(L_{\xi_1} \cdots L_{\xi_n} v\right)$ where $n \in
    \mathbb{N}_0$ and $\xi_1, \ldots, \xi_n \in \mathfrak{g}$. Since
    all the operators $L_\xi$ are continuous, the resulting system of
    seminorms consists of continuous seminorms with respect to the
    original topology. Hence they coincide.
\end{proof}

The following lemma shows that for the $\mathrm{SU}(1, n)$-action on
$\complete{\mathcal{A}}_\hbar(C_{n+1}^+)$ we are indeed in this
situation:
\begin{lemma}
    \label{lemma:ActionIsSmoothUpstairs}%
    The dense subspace $\mathcal{A}_\hbar(C_{n+1}^+)$ of
    $\complete{\mathcal{A}}_\hbar(C_{n+1}^+)$ satisfies the conditions
    of Lemma~\ref{lemma:CinftyAction} where $L_\xi =
    \frac{\I}{\hbar}[J_\xi, \argument]_{\tildewick}$ for $\xi \in
    \mathfrak{su}(1, n)$.
\end{lemma}
\begin{proof}
    First we note that $\mathcal{A}_\hbar(C_{n+1}^+)$ is invariant
    under the action of $\mathrm{SU}(1, n)$ as well as under taking
    commutators with $J_\xi$ as we have seen in
    \eqref{eq:LieXiAgammaInv}. For $a \in
    \mathcal{A}_\hbar(C_{n+1}^+)$ we find a $\gamma$ with $a \in
    \mathcal{A}_\hbar^\gamma(C_{n+1}^+)$. But this is a
    \emph{finite-dimensional} subspace and hence there is only one
    Hausdorff locally convex topology. For this topology it is clear
    that the action is differentiable with derivatives given by
    \eqref{eq:DiffbarOnVnull} and $L_\xi = \frac{\I}{\hbar}[J_\xi,
    \argument]_{\tildewick}$. Indeed, this is just a trivial
    computation. The claim follows since the product $\tildewick$ is
    continuous and hence also the commutators with $J_\xi$ are
    continuous.
\end{proof}

We can now formulate the main result of this subsection:
\begin{theorem}
    \label{theorem:SUEinsnSymmetrie}%
    Let $\hbar$ be an allowed value.
    \begin{theoremlist}
    \item \label{item:SUEinsnGlatteWirkung} The $\mathrm{SU}(1,
        n)$-action on $\complete{\mathcal{A}}_\hbar(C_{n+1}^+)$ by
        automorphisms of $\tildewick$ is smooth and the smooth
        topology with respect to this action coincides with the
        original one.
    \item \label{item:SUEinsnGlattAufDisc} The action of
        $\mathrm{SU}(1, n)$ descends to a smooth action on
        $\complete{\mathcal{A}}_\hbar(\mathbb{D}_n)$ by automorphisms
        and the smooth topology coincides with the original one.
    \item \label{item:DerivedActionInner} In both cases the
        corresponding Lie algebra action of $\mathfrak{su}(1, n)$ is
        inner via the equivariant momentum map $J$.
    \item \label{item:ActionIsStarAction} If in addition $\hbar$ is
        real then $\mathrm{SU}(1, n)$ acts by $^*$-automorphisms.
    \end{theoremlist}
\end{theorem}
\begin{proof}
    The first part is an immediate consequence of
    Lemma~\ref{lemma:CinftyAction} and
    Lemma~\ref{lemma:ActionIsSmoothUpstairs}. For the second, we note
    that $y$ is invariant and hence the vanishing ideal of $y=1$ is
    invariant as well. Thus the maps $U^*$ descend to well-defined
    automorphisms of $\complete{\mathcal{A}}_\hbar(\mathbb{D}_n)$. By
    the very definition of the locally convex quotient topology the
    maps $U^*$ are continuous as well. Moreover, it is clear that the
    action is continuous in $U$ as well. By the usual argument the
    continuity of the action follows. The smoothness property simply
    descends to the quotient and since the momentum map also descends,
    we get the second part.  The last two parts are clear.
\end{proof}
\begin{remark}
    \label{remark:InnerAction}%
    It is tempting to try to use the inner Lie algebra action via
    $\frac{\I}{\hbar}[J_\xi, \argument]_{\tildewick}$ to obtain also
    an \emph{inner} action of $\mathrm{SU}(1, n)$ by exponentiating
    $J_\xi$. This $\tildewick$-exponential would indeed yield an inner
    action and hence the continuity and smoothness statements would
    follow trivially: the action would be even analytic with an entire
    extension. However, we expect that the algebra is \emph{not}
    locally multiplicatively convex and hence the existence of an
    exponential is far from being obvious. We leave it as an open
    problem whether one can actually exponentiate $J_\xi$.
\end{remark}

%
%

\subsection{The dependence on $\hbar$}
\label{subsec:DependenceOnhbar}

We have already seen that the structure constants with respect to the
basis $\basis{f}_{P, Q, \alpha}$ are independent of $\hbar$. However,
since we have the concrete interpretation of $a \in
\complete{\mathcal{A}}_\hbar(C_{n+1}^+)$ as \emph{functions} on
$C_{n+1}^+$, the coefficients $a_{P, Q, \alpha}$ depend on the choice
of $\hbar$ since the functions $\basis{f}_{P, Q, \alpha}$ do so, see
\eqref{eq:fPQalphaDef}. Hence the algebra
$\complete{\mathcal{A}}_\hbar(C_{n+1}^+)$ including its topology
depend on $\hbar$, a priori. To emphasize the dependence on $\hbar$ we
denote the coefficients of $a \in
\complete{\mathcal{A}}_\hbar(C_{n+1}^+)$ by $a_{P, Q, \alpha}(\hbar)$
and write $\basis{f}_{P, Q, \alpha}(\hbar)$ in this subsection.  We
also decorate the seminorms $\norm{\argument}_{m, \ell, (I, J,
  \gamma), \hbar}$ with an additional $\hbar$.

In order to understand this dependence we introduce the following
auxiliary algebra $\mathfrak{A}(C_{n+1}^+)$ being the span of basis
vectors $\basis{F}_{P, Q, \alpha}$ for all index triples $(P, Q,
\alpha)$ endowed with the product $\star$ specified by the structure
constants $C^{(I, J, \gamma)}_{(P, Q, \alpha), (R, S, \beta)}$ as in
\eqref{eq:StructureConstantsTildeWick}. Thanks to the precise form of
the structure constants we have no dependence on $\hbar$ in the
definition of $\mathfrak{A}(C_{n+1}^+)$. Repeating our general
construction we get a definition of a locally convex topology on
$\mathfrak{A}(C_{n+1}^+)$ first based on the $h_{m, \ell, (I, J,
  \gamma)}$, yielding the Cartesian product topology. In a second step
we use the definition of the seminorms $\norm{\argument}_{m, \ell, R}$
as in \eqref{eq:NewSeminormsWithR} also for $\mathfrak{A}(C_{n+1}^+)$
for which the completion $\complete{\mathfrak{A}}(C_{n+1}^+)$ becomes
a Fréchet algebra, still independent on any choice of
$\hbar$. Equivalently, we can use the seminorms
$\norm{\argument}_\epsilon$ build as in
\eqref{eq:YetAnotherEquivalentSystemOfSeminorms}, making
$\complete{\mathfrak{A}}(C_{n+1}^+)$ isomorphic to a Köthe space of
sub-factorial growth.
\begin{lemma}
    \label{lemma:PhihbarIso}%
    Let $\hbar, \hbar' \ne 0$.
    \begin{lemmalist}
    \item \label{item:PhihbarIso} The map
        \begin{equation}
            \label{eq:PhihbarMap}
            \Phi_\hbar\colon
            \complete{\mathcal{A}}_\hbar(C_{n+1}^+)
            \ni
            \sum_{P, Q, \alpha} a_{P, Q, \alpha}(\hbar)
            \basis{f}_{P, Q, \alpha}(\hbar)
            \; \mapsto \;
            \sum_{(P, Q, \alpha)}
            a_{P, Q, \alpha}(\hbar) \basis{F}_{P, Q, \alpha}
            \in \complete{\mathfrak{A}}(C_{n+1}^+)
        \end{equation}
        is a well-defined isomorphism of Fréchet algebras mapping
        $\mathcal{A}_\hbar(C_{n+1}^+)$ to $\mathfrak{A}(C_{n+1}^+)$.
    \item \label{item:TwohbarPhis} Suppose $t = \frac{\hbar}{\hbar'} >
        0$. Then
        \begin{equation}
            \label{eq:PhiPhiInverseScaling}
            \Phi_{\hbar'}^{-1} \circ \Phi_\hbar
            =
            R_{\sqrt{t}}^*
        \end{equation}
        where $R_t\colon C_{n+1}^+ \ni z \mapsto tz \in C_{n+1}^+$ is
        the rescaling by $t$.
    \end{lemmalist}
\end{lemma}
\begin{proof}
    First we consider $a \in \mathcal{A}_\hbar(C_{n+1}^+)$ for which
    the above map is clearly a well-defined linear map as all involved
    summations are finite. It clearly gives a bijection onto
    $\mathfrak{A}(C_{n+1}^+)$. Since the structure constants of the
    $\basis{f}_{P, Q, \alpha}(\hbar)$ do \emph{not} depend on $\hbar$,
    $\Phi_\hbar$ is an algebra homomorphism. Finally, by the very
    construction of the seminorms we have $\norm{a}_{m, \ell, (I, J,
      \gamma), \hbar} = \norm{\Phi_{\hbar}(a)}_{m, \ell, (I, J,
      \gamma)}$. From this it is clear that also the seminorms
    $\norm{\argument}_{m, \ell, R, \hbar}$ of
    $\mathcal{A}_\hbar(C_{n+1}^+)$ correspond to the seminorms
    $\norm{\argument}_{m, \ell, R}$ of $\mathfrak{A}(C_{n+1}^+)$ under
    $\Phi_\hbar$, showing that $\Phi_\hbar$ as well as its inverse
    are continuous. The first part follows. For the second part we
    note that $\left(\Phi_{\hbar'}^{-1} \circ \Phi_\hbar\right)
    (\basis{f}_{P, Q, \alpha}(\hbar)) = \basis{f}_{P, Q,
      \alpha}(\hbar')$ by the very definition. For positive $t =
    \frac{\hbar}{\hbar'}$ we see immediately that $R^*_{\sqrt{t}} y =
    t y$. Since the functions $\frac{e_{P, Q, \alpha}}{y^\alpha}$ are
    constant along the complex rays in $C_{n+1}^+$ this gives
    \[
    R^*_{\sqrt{t}} \basis{f}_{P, Q, \alpha}(\hbar)
    =
    \basis{f}_{P, Q, \alpha}(\hbar').
    \]
    Since $\Phi_{\hbar'}^{-1} \circ \Phi_\hbar\colon
    \complete{\mathcal{A}}_\hbar(C_{n+1}^+) \longrightarrow
    \complete{\mathcal{A}}_{\hbar'}(C_{n+1}^+)$ is an isomorphism of
    Fréchet algebras by the first part, we get
    \begin{align*}
        \left(\Phi_{\hbar'}^{-1} \circ \Phi_{\hbar}\right)(a)
        &=
        \left(\Phi_{\hbar'}^{-1} \circ \Phi_{\hbar}\right)\left(
            \sum_{(P, Q, \alpha)}
            a_{P, Q, \alpha}(\hbar)
            \basis{f}_{P, Q, \alpha}(\hbar)
        \right) \\
        &=
        \sum_{(P, Q, \alpha)}
        a_{P, Q, \alpha}(\hbar)
        \left(\Phi_{\hbar'}^{-1} \circ \Phi_{\hbar}\right)\left(
            \basis{f}_{P, Q, \alpha}(\hbar)
        \right) \\
        &=
        \sum_{(P, Q, \alpha)}
        a_{P, Q, \alpha}(\hbar)
        R_{\sqrt{t}}^*
        \basis{f}_{P, Q, \alpha}(\hbar)
    \end{align*}
    for $a \in \complete{\mathcal{A}}_\hbar(C_{n+1}^+)$. Now for
    finite sums, i.e. $a \in \mathcal{A}_\hbar(C_{n+1}^+)$ this
    clearly coincides with $R_{\sqrt{t}}^* a$. Since in general the
    series converges in the topology of
    $\complete{\mathcal{A}}_\hbar(C_{n+1}^+)$ and since all evaluation
    functionals at points in $C_{n+1}^+$ are continuous, the last
    series also coincides with $R_{\sqrt{t}}^* a$ defined
    \emph{pointwise}. Hence we have shown two things: first
    $R_{\sqrt{t}}^*$ maps $\complete{\mathcal{A}}_\hbar(C_{n+1}^+)$
    continuously into $\complete{\mathcal{A}}_{\hbar'}(C_{n+1}^+)$,
    and, second, \eqref{eq:PhiPhiInverseScaling} holds.
\end{proof}
\begin{corollary}
    \label{corollary:AhbarAllIso}%
    For $\hbar \ne 0$ all the Fréchet algebras
    $\complete{\mathcal{A}}_\hbar(C_{n+1}^+)$ are isomorphic.
\end{corollary}

However, the isomorphism is nontrivial in the following sense: since
we know that $\complete{\mathcal{A}}_\hbar(C_{n+1}^+) \subseteq
C^\omega(C_{n+1}^+)$ is also a subspace of the real-analytic functions
on $C_{n+1}^+$, we can directly compare the elements as
functions. Here we have the following result:
\begin{lemma}
    \label{lemma:AhbarNotEqual}%
    For $\hbar \ne \hbar'$ the subspaces
    $\complete{\mathcal{A}}_\hbar(C_{n+1}^+)$ and
    $\complete{\mathcal{A}}_{\hbar'}(C_{n+1}^+)$ of
    $C^\omega(C_{n+1}^+)$ are different.
\end{lemma}
\begin{proof}
    Here the argument is similar as for the proof of
    Lemma~\ref{lemma:VanishingIdeal}. First we note that we can
    evaluate $a \in \complete{\mathcal{A}}_\hbar(C_{n+1}^+)$ at every
    point $z \in \mathbb{C}^{n+1}$ in a continuous way: indeed, the
    sub-factorial growth of the coefficients $a_{I, J, \gamma}$ with
    respect to the Schauder basis $\basis{f}_{I, J, \gamma}$ together
    with the factorial decrease of $\basis{f}_{I, J, \gamma}(z)$
    according to \eqref{eq:BasisAtwEvaluation} shows the continuity of
    all evaluation functionals. Thus the vanishing ideal of the
    hypersurface $y = -2\hbar$ is a closed subspace in
    $\complete{\mathcal{A}}_{\hbar'}(C_{n+1}^+)$ for all nonzero
    $\hbar$, $\hbar'$. Now it is easy to see that for $\hbar = \hbar'$
    this subspace has finite co-dimension. In fact, the quotient is
    spanned by the classes of $\basis{f}_{0, 0, 0} = 1$ and the
    $\basis{f}_{i, j, 1}$ with $i, j = 1, \ldots, n$. For all other
    $\hbar'$ the co-dimension is strictly bigger: either infinite in
    the generic case, or, for particular values of $\hbar'$, finite
    but larger. Hence we conclude that the two subspaces
    $\complete{\mathcal{A}}_\hbar(C_{n+1}^+)$ and
    $\complete{\mathcal{A}}_{\hbar'}(C_{n+1}^+)$ are different for
    $\hbar \ne \hbar'$.
\end{proof}

Things change when we pass to the disk: here the Pochhammer symbol
becomes just a numerical constant (depending on $\hbar$) but the span
of the basis vectors $\basis{f}_{PQ}$ is independent of $\hbar$:
\begin{lemma}
    \label{lemma:OnDiscSameSpace}%
    Let $\hbar, \hbar'$ be allowed values. Then the subspaces
    $\complete{\mathcal{A}}_\hbar(\mathbb{D}_n)$ and
    $\complete{\mathcal{A}}_{\hbar'}(\mathbb{D}_n)$ of
    $C^\omega(\mathbb{D}_n)$ coincide. Moreover, the Fréchet topology
    of $\complete{\mathcal{A}}_\hbar(\mathbb{D}_n)$ is independent of
    $\hbar$.
\end{lemma}
\begin{proof}
    First we note that the only difference between $\basis{f}_{P,
      Q}(\hbar)$ and $\basis{f}_{P, Q}(\hbar')$ are the prefactors
    given by the Pochhammer symbols
    $\Pochhammer{\frac{1}{2\hbar}}_{\max(|P|, |Q|)}$ and
    $\Pochhammer{\frac{1}{2\hbar'}}_{\max(|P|, |Q|)}$. Since both are
    non-vanishing their quotient can be estimated with exponential
    bounds. But then it is clear that for $a = \sum_{P, Q} a_{P,
      Q}(\hbar) \basis{f}_{P, Q}(\hbar)$ with coefficients $a_{P,
      Q}(\hbar)$ having sub-factorial growth, the coefficients
    \[
    a_{P, Q}(\hbar')
    =
    \frac{\Pochhammer{\frac{1}{2\hbar}}_{\max(|P|, |Q|)}}
    {\Pochhammer{\frac{1}{2\hbar'}}_{\max(|P|, |Q|)}}
    \:
    a_{P, Q}(\hbar)
    \tag{$*$}
    \]
    have still sub-factorial growth. More precisely, for $\epsilon' >
    0$ and every $\epsilon' > \epsilon > 0$ we have
    \begin{align*}
        \norm{a}_{\epsilon'}^{\hbar'}
        &= \sup_{P, Q}
        \frac{|a_{P, Q}(\hbar')|}{(\max(|P|, |Q|)!)^{\epsilon'}} \\
        &=
        \sup_{P, Q}
        \frac{\Pochhammer{\frac{1}{\hbar}}_{\max(|P|, |Q|)}}
        {\Pochhammer{\frac{1}{\hbar'}}_{\max(|P|, |Q|)}}
        \frac{|a_{P, Q}(\hbar)|}{(\max(|P|, |Q|)!)^{\epsilon'}} \\
        &\le
        \sup_{P, Q}
        \frac{\mathsf{c}(\hbar)^{\max(|P|, |Q|)}}
        {\mathsf{a}(\hbar')\mathsf{b}(\hbar')^{\max(|P|, |Q|)}}
        \frac{|a_{P, Q}(\hbar)|}{(\max(|P|, |Q|)!)^{\epsilon'}} \\
        &\le
        \sup_{P, Q}
        \frac{1}{(\max(|P|, |Q|)!)^{\epsilon' - \epsilon}}
        \frac{\mathsf{c}(\hbar)^{\max(|P|, |Q|)}}
        {\mathsf{a}(\hbar')\mathsf{b}(\hbar')^{\max(|P|, |Q|)}}
        \;
        \sup_{P, Q}
        \frac{|a_{P, Q}(\hbar)|}{(\max(|P|, |Q|)!)^{\epsilon}} \\
        &=
        c(\hbar, \hbar') \norm{a}_\epsilon^\hbar,
        \tag{$**$}
    \end{align*}
    where $\mathsf{a}(\hbar')$, $\mathsf{b}(\hbar')$, and
    $\mathsf{c}(\hbar)$ are the constants from the standard estimate
    of the Pochhammer symbols
    \eqref{eq:EstimatePochhammerSymbols}. Since in
    \eqref{eq:EstimatePochhammerSymbols} we can chose the parameters
    to be locally uniform, we can arrange the constant $c(\hbar,
    \hbar')$ to be locally uniform in $\hbar$ and $\hbar'$. Clearly,
    the above estimate shows that the topologies coincide.
\end{proof}

This observation is crucial to make the following questions
meaningful: for given $a, b \in
\complete{\mathcal{A}}_\hbar(\mathbb{D}_n)$ we can ask how $a
\stardisk b$ will depend on $\hbar$ with respect to the
($\hbar$-independent) topology of
$\complete{\mathcal{A}}_\hbar(\mathbb{D}_n)$. We start with the
following technical lemma:
\begin{lemma}
    \label{lemma:KoefficientsHolomorphic}%
    Let $a \in \complete{\mathcal{A}}_\hbar(\mathbb{D}_n)$. Then the
    coefficients $a_{P, Q}(\hbar)$ of $a$ with respect to the Schauder
    basis $\basis{f}_{P, Q}(\hbar)$ are holomorphic with possible
    (simple) poles at the non-allowed values of $\hbar$.
\end{lemma}
\begin{proof}
    Indeed, let $\hbar_0$ be a fixed allowed value. Then by ($*$) of
    the previous proof $a_{P, Q}(\hbar)$ is, up to a constant, the
    inverse of the Pochhammer symbol
    $\Pochhammer{\frac{1}{2\hbar}}_{\max(|P|, |Q|)}$.
\end{proof}

Moreover, ($**$) in the proof of Lemma~\ref{lemma:OnDiscSameSpace}
shows that the sub-factorial growth of $a_{P, Q}(\hbar)$ is locally
uniform in $\hbar$, i.e.\ for all $\hbar$ in some compact subset $K$
within the allowed values we get for all $\epsilon > 0$ a constant
$c_K$ with
\begin{equation}
    \label{eq:LocallyUniformSubFacInhbar}
    |a_{P, Q}(\hbar)| \le c_K (\max(|P|, |Q|)!)^\epsilon.
\end{equation}
For the next lemma it is again crucial that the topology of
$\complete{\mathcal{A}}_\hbar(\mathbb{D}_n)$ is independent of
$\hbar$. An analogous statement for the Schauder basis upstairs does
not even make sense according to Lemma~\ref{lemma:AhbarNotEqual}.
\begin{lemma}
    \label{lemma:SchauderBasisHolomorphic}%
    For all $P, Q$ the map $\mathbb{C} \ni \hbar \mapsto \basis{f}_{P,
      Q}(\hbar) \in \complete{\mathcal{A}}_\hbar(\mathbb{D}_n)$ is
    holomorphic on $\mathbb{C} \setminus \{0\}$.
\end{lemma}
\begin{proof}
    This is trivial as $\basis{f}_{P, Q}(\hbar)$ is the holomorphic
    Pochhammer symbol $\Pochhammer{\frac{1}{2\hbar}}_{\max(|P|, |Q|)}$
    times a fixed vector in
    $\complete{\mathcal{A}}_\hbar(\mathbb{D}_n)$.
\end{proof}
\begin{lemma}
    \label{lemma:HolLiftToA}%
    Let $a \in \complete{\mathcal{A}}_\hbar(\mathbb{D}_n)$. Then
    \begin{equation}
        \label{eq:liftaToA}
        A(\hbar)
        =
        \sum_{P, Q} a_{P, Q}(\hbar) \basis{F}_{P, Q, \max(|P|, |Q|)}
    \end{equation}
    defines a holomorphic map $A$ from the allowed values of $\hbar$
    to $\mathfrak{A}(C_{n+1}^+)$.
\end{lemma}
\begin{proof}
    Indeed, each $\hbar \mapsto a_{P, Q}(\hbar) \basis{F}_{P, Q,
      \max(|P|, |Q|)}$ is holomorphic by
    Lemma~\ref{lemma:KoefficientsHolomorphic}. Moreover,
    $\norm{\basis{F}_{P, Q, \max(|P|, |Q|)}}_\epsilon =
    \frac{1}{(\max(|P|, |Q|)!)^\epsilon}$. But then the locally
    uniform estimate \eqref{eq:LocallyUniformSubFacInhbar} shows that
    the series \eqref{eq:liftaToA} converges absolutely and locally
    uniformly with respect to each seminorm
    $\norm{\argument}_\epsilon$. Thus the resulting map is again
    holomorphic.
\end{proof}

We will need a refinement of our consideration in
Theorem~\ref{theorem:DiscAlgebra}: Let $A = \sum_{(P, Q, \alpha)}
A_{P, Q, \alpha} \basis{F}_{P, Q, \alpha} \in \mathfrak{A}(C_{n+1}^+)$
be given and consider $a(\hbar) = [\Phi_\hbar^{-1}(A)] \in
\complete{\mathcal{A}}_\hbar(\mathbb{D}_n)$. The $\hbar$-dependence of
$a$ comes from the map $\Phi_\hbar$ and from the quotient
procedure. We are then interested in the growth properties of the
coefficients $a_{R, S}$ of $a$.
\begin{lemma}
    \label{lemma:MapAtoAdiscCoeffFuns}%
    For all $\epsilon' > 0$ and all compact subsets $K$ of allowed
    values of $\hbar$ there exists an $\epsilon > 0$ and a constant
    $c_K$ such that
    \begin{equation}
        \label{eq:EstimateForCoefficientsaRS}
        |a_{R, S}(\hbar)|
        \le
        c_K \norm{A}_\epsilon (\max(|R|, |S|)!)^{\epsilon'}
    \end{equation}
    for all $R, S \in \mathbb{N}_0^n$ and all $\hbar \in K$.
\end{lemma}
\begin{proof}
    The main point is that $c_K$ can be chosen locally uniform in
    $\hbar$. In the proof of Theorem~\ref{theorem:DiscAlgebra} in
    ($\star$) we had an estimate for $a_{R, S}$ involving several
    constants: fixing $\epsilon' > 0$, we chose a $0 < \epsilon <
    \epsilon'$ and set $c_1 = \norm{A}_\epsilon$ yielding the estimate
    $|A_{P, Q, \alpha}| \le c_1 (\alpha!)^\epsilon$ needed. Moreover,
    we note that in ($\star$) we can choose $c_2$ and $c_3$ locally
    uniformly in $\hbar$ thanks to
    \eqref{eq:EstimatePochhammerSymbols}. Then ($\star$) will yield an
    estimate $|a_{R, S}(\hbar)| \le c_7 c_8^M (M!)^\epsilon$ with $c_7
    = \norm{A}_\epsilon$ times a numerical constant being locally
    uniform in $\hbar$ and $c_8$ being locally uniform in $\hbar$ as
    well. From this we deduce the claim.
\end{proof}

With other words, the continuity of the map $A \mapsto a(\hbar)$,
which was clear before, can be sharpened to the estimate
\begin{equation}
    \label{eq:ContinuityAtoa}
    \norm{a(\hbar)}^\hbar_{\epsilon'} \le c \norm{A}_\epsilon
\end{equation}
with a \emph{locally uniform} constant $c$ concerning the dependence
on $\hbar$.

Finally, assume that $\hbar \mapsto A(\hbar)$ is itself a holomorphic
function with values in $\mathfrak{A}(C_{n+1}^+)$. Then clearly all
the coefficients $A_{P, Q, \alpha}(\hbar)$ are scalar holomorphic
functions since we have a Schauder basis. Moreover, $\hbar \mapsto
\norm{A(\hbar)}_\epsilon$ is a continuous function for all $\epsilon$
and hence locally bounded. From the construction of the $a_{R, S}$ we
see that we have absolute and locally uniform convergence of
holomorphic functions, implying that also the scalar function $\hbar
\mapsto a_{R, S}(\hbar)$ is holomorphic. This will eventually lead to
the following result:
\begin{theorem}[Holomorphic deformation]
    \label{theorem:HolomorphicProduct}%
    Let $a, b \in \complete{\mathcal{A}}_\hbar(\mathbb{D}_n)$ be
    given. Then the product $a \stardisk b$ depends holomorphically on
    $\hbar$ for all allowed values of $\hbar$.
\end{theorem}
\begin{proof}
    First we note that by Lemma~\ref{lemma:HolLiftToA} the maps $\hbar
    \mapsto A(\hbar), B(\hbar) \in \mathfrak{A}(C_{n+1}^+)$ are
    holomorphic where $A(\hbar) = \sum_{P, Q} a_{P, Q}(\hbar)
    \basis{F}_{P, Q, \max(|P|, |Q|)}$ and analogously for
    $B(\hbar)$. Since the product $\star$ of $\mathfrak{A}(C_{n+1}^+)$
    is continuous (and independent of $\hbar$), also $A(\hbar) \star
    B(\hbar)$ is holomorphic and we have $a \stardisk b =
    [\Phi_\hbar^{-1}(A(\hbar) \star B(\hbar))]$. Thus the coefficients
    $(a \stardisk b)_{R, S}(\hbar)$ of
    \[
    a \stardisk b
    =
    \sum_{R, S} (a \stardisk b)_{R, S}(\hbar) \basis{f}_{R, S}(\hbar)
    \tag{$*$}
    \]
    are scalar holomorphic functions as we just argued. By
    Lemma~\ref{lemma:SchauderBasisHolomorphic} also $\basis{f}_{R,
      S}(\hbar)$ depends holomorphically on $\hbar$. Finally,
    Lemma~\ref{lemma:MapAtoAdiscCoeffFuns} together with
    $\norm{\basis{f}_{R, S}}_\epsilon = \frac{1}{(\max(|R|,
      |S|)!)^\epsilon}$ shows that the convergence of the series ($*$)
    is not only absolute with respect to the seminorms
    $\norm{\argument}_\epsilon$ but even locally uniform in
    $\hbar$. Hence the result is again holomorphic, proving the claim.
\end{proof}
\begin{remark}
    \label{remark:Pols}%
    We are thus in the situation of a \emph{holomorphic} deformation
    in the sense of \cite{pflaum.schottenloher:1998a} with the crucial
    difference that $\hbar = 0$ is \emph{not} in the domain of
    definition. In fact, the singularities at the non-allowed values
    on the negative axis, i.e.\ the zeros of the Pochhammer symbols,
    accumulate at $0$. Hence there is no chance to extend the domain
    such that $\hbar = 0$ is included. As an alternative to the above
    argument, a more direct estimate of $a \stardisk b$ can be used to
    show the locally uniform convergence in $\hbar$ by expanding
    everything in terms of the coefficients with respect to the
    Schauder basis.
\end{remark}

%
%

\subsection{The $^*$-involution for real $\hbar$}
\label{subsec:StarInvolutionRealHbar}

We consider now the particular case of real $\hbar$. From
Proposition~\ref{proposition:StructureConstantsAbove},
\refitem{item:AlgebraForRealHbar}, we know that
$\mathcal{A}_\hbar(C_{n+1}^+)$ is a $^*$-algebra with respect to the
pointwise complex conjugation as $^*$-involution. In fact, the complex
conjugation is continuous and extends therefore to a $^*$-involution
of $\complete{\mathcal{A}}_\hbar(C_{n+1}^+)$. We arrive at the
following statement:
\begin{proposition}
    \label{proposition:DiscAlgebraStarAlgebra}%
    Let $\hbar$ be real and an allowed value.
    \begin{propositionlist}
    \item \label{item:CCStetig} The complex conjugation extends to a
        continuous $^*$-involution of
        $\complete{\mathcal{A}}_\hbar(C_{n+1}^+)$ and for all $0 <
        \epsilon < 1$ and $a \in
        \complete{\mathcal{A}}_\hbar(\mathbb{C}_{n+1})$ we have
        \begin{equation}
            \label{eq:EpsilonSeminormCCStable}
            \norm{\cc{a}}_\epsilon = \norm{a}_\epsilon.
        \end{equation}
    \item \label{item:IdealIsStarIdeal} The ideal $\mathcal{J}_{y=1}
        \subseteq \complete{\mathcal{A}}_\hbar(C_{n+1}^+)$ is a
        $^*$-ideal.
    \item \label{item:CCforDiscAlgebra} The induced continuous
        $^*$-involution on $\complete{\mathcal{A}}_\hbar(\mathbb{D}_n)
        = \complete{\mathcal{A}}_\hbar(C_{n+1}^+) \big/
        \mathcal{J}_{y=1}$ is the pointwise complex conjugation once
        we identify elements of
        $\complete{\mathcal{A}}_\hbar(\mathbb{D}_n)$ with functions on
        $\mathbb{D}_n$. For all $0 < \epsilon < 1$ and $[a] \in
        \complete{\mathcal{A}}_\hbar(\mathbb{D}_n)$ we have
        \begin{equation}
            \label{eq:EstimateOnDiscForCC}
            \norm{\cc{[a]}}_\epsilon = \norm{[a]}_\epsilon.
        \end{equation}
    \end{propositionlist}
\end{proposition}
\begin{proof}
    For the first part it suffices to show
    \eqref{eq:EpsilonSeminormCCStable}. From \eqref{eq:ccfPQalpha} we
    conclude $(\cc{a})_{P, Q, \alpha} = \cc{a_{Q, P, \alpha}}$ and
    hence \eqref{eq:EpsilonSeminormCCStable} is obvious. For any $w
    \in C_{n+1}^+$ the evaluation functional $\delta_w$ is real in the
    sense that $\cc{\delta_w(a)} = \delta_w(\cc{a})$. Hence $\ker
    \delta_w$ is stable under complex conjugation and thus also
    $\mathcal{J}_{y=1}$, proving the second part. The third part is
    then clear by construction.
\end{proof}

%
%

\subsection{Positivity of $\delta$-functionals and their GNS construction}
\label{subsec:PositivityAndGNS}

Let $\hbar > 0$ be positive. Then we consider the evaluation
functionals on the disk for the algebra
$\complete{\mathcal{A}}_\hbar(\mathbb{D}_n)$. Since the group
$\mathrm{SU}(1, n)$ acts transitively on $\mathbb{D}_n$, it will be
sufficient to consider one point only, all others will be obtained by
pulling back the evaluation functional via the $^*$-automorphisms
$U^*$ of $\complete{\mathcal{A}}_\hbar(\mathbb{D}_n)$. Hence we can
concentrate on the point $v = 0 \in \mathbb{D}_n$.
\begin{lemma}
    \label{lemma:DeltaIsPositive}%
    Let $\hbar > 0$.
    \begin{lemmalist}
    \item \label{item:DeltaPositive} The evaluation functional
        $\delta_0 \colon \complete{\mathcal{A}}_\hbar(\mathbb{D}_n)
        \longrightarrow \mathbb{C}$ is positive.
    \item \label{item:GelfandIdealOfDelta} The Gel'fand ideal of
        $\delta_0$ is given by
        \begin{equation}
            \label{eq:GelfandDeltaNull}
            \mathcal{J}_0
            =
            \left\{
                a = \sum\nolimits_{P, Q} a_{P, Q} \basis{f}_{P, Q}
                \in
                \complete{\mathcal{A}}_\hbar(\mathbb{D}_n)
                \; \bigg| \;
                a_{0, Q} = 0
                \;
                \textrm{for all}
                \;
                Q \in \mathbb{N}_0^n
            \right\}.
        \end{equation}
    \end{lemmalist}
\end{lemma}
\begin{proof}
    First we note that $\delta_0(a)$ can be obtained by evaluating a
    representative of $a$ in $\complete{\mathcal{A}}_\hbar(C_{n+1}^+)$
    at the point $w = (1, 0, \ldots, 0) = \pi^{-1}(0) \in
    C_{n+1}^+$. For a general index triple $(I, J, \gamma)$ we have
    \[
    \delta_w\left(
        \basis{f}_{I, J, \gamma}
    \right)
    =
    \frac{1}{I!J!(\gamma - |I|)!(\gamma - |J|)!}
    \Pochhammer{\frac{1}{2\hbar}}_\gamma
    \delta_{I, 0} \delta_{J, 0}
    =
    \frac{1}{\gamma!\gamma!}
    \Pochhammer{\frac{1}{2\hbar}}_\gamma
    \delta_{I, 0} \delta_{J, 0}.
    \tag{$*$}
    \]
    Now let $P$, $Q$, $R$, and $S$ be given and set $\alpha =
    \max(|Q|, |P|)$ and $\beta = \max(|R|, |S|)$. Then we have
    \begin{align*}
        &\delta_0\left(
            \cc{\basis{f}_{P, Q}} \stardisk \basis{f}_{R, S}
        \right) \\
        &\quad=
        \delta_0\left(
            \basis{f}_{Q, P} \stardisk \basis{f}_{R, S}
        \right) \\
        &\quad=
        \delta_w\left(
            \basis{f}_{Q, P, \alpha}
            \tildewick
            \basis{f}_{R, S, \beta}
        \right) \\
        &\quad=
        \sum_{k=0}^{\min(\alpha - |Q|, \beta - |S|)}
        \sum_{K=0}^{\min(Q, S)}
        \frac{(-1)^k}{k!K!}
        \delta_w\left(
            \basis{f}_{
              Q + R - K,
              P + S - K,
              \alpha + \beta - k - |K|
            }
        \right) \\
        &\qquad
        \binom{Q + R - K}{R}
        \binom{P + S - K}{P}        
        \binom{\alpha + \beta - k - |Q| - |R|}{\beta - |R|}
        \binom{\alpha + \beta - k - |P| - |S|}{\alpha - |P|}.
        \tag{$**$}
    \end{align*}
    In order to get a nonzero contribution in ($**$) we need $Q + R -
    K = 0 = P + S - K$ according to ($*$). Since $K$ has range $0,
    \ldots, \min(Q, S)$ this implies $K = Q = S$ and $R = P =
    0$. Hence $\alpha = |Q|$ and $\beta = |S|$ which implies that in
    ($**$) only $k = 0$ occurs. For these values, all the binomial
    coefficients in ($**$) are $1$. Then remaining evaluation is
    $\delta_w(\basis{f}_{0, 0, |Q|}) = \frac{1}{|Q|!|Q|!}
    \Pochhammer{\frac{1}{2\hbar}}_{|Q|}$.  This allows to compute
    \[
    \delta_0(\cc{a} \stardisk b)
    =
    \sum_{P, Q, R, S}
    \cc{a_{P, Q}} b_{R, S}
    \delta_0\left(
        \cc{\basis{f}_{P, Q}} \stardisk \basis{f}_{R, S}
    \right)
    =
    \sum_{Q} \frac{\cc{a_{0, Q}}b_{0, Q}}{|Q|!|Q|!}
    \Pochhammer{\frac{1}{2\hbar}}_{|Q|}.
    \]
    Since for $\hbar > 0$ the Pochhammer symbol
    $\Pochhammer{\frac{1}{2\hbar}}_{|Q|}$ is also positive, we see
    that $\delta_0$ is a positive functional by setting $a =
    b$. Moreover, \eqref{eq:GelfandDeltaNull} is clear from the last
    formula.
\end{proof}
\begin{corollary}
    \label{corollary:AllDeltaPositive}%
    Let $\hbar > 0$.  For all points $v \in \mathbb{D}_n$, the
    evaluation functional $\delta_v \colon
    \complete{\mathcal{A}}_\hbar(\mathbb{D}_n) \longrightarrow
    \mathbb{C}$ is positive.
\end{corollary}
\begin{lemma}
    \label{lemma:GNSPreHilbertSpace}%
    Let $\hbar > 0$.
    \begin{lemmalist}
    \item \label{item:GelfandIdealClosed} The Gel'fand ideal
        $\mathcal{J}_0$ of $\delta_0$ is a closed subspace.
    \item \label{item:GelfandIdealComplement} The Gel'fand ideal
        $\mathcal{J}_0$ has a closed complementary subspace
        $\mathcal{J}_0 \oplus \mathfrak{D}_\hbar =
        \complete{\mathcal{A}}_\hbar(\mathbb{D}_n)$ explicitly given
        by
        \begin{equation}
            \label{eq:ComplementGelfand}
            \mathfrak{D}_\hbar =
            \left\{
                \psi = \sum\nolimits_Q
                \psi_Q \basis{f}_{0, Q}
                \; \bigg| \;
                (\psi_Q)_{Q \in \mathbb{N}_0^n}
                \;
                \textrm{has sub-factorial growth}
            \right\}
            \subseteq
            \complete{\mathcal{A}}_\hbar(\mathbb{D}_n).
        \end{equation}
    \item \label{item:GNSPreHilbertSpace} The GNS pre Hilbert space is
        canonically isomorphic to $\mathfrak{D}_\hbar$ with the inner
        product
        \begin{equation}
            \label{eq:GNSInnerProduct}
            \SP{\psi, \phi}_\hbar
            =
            \sum_{Q} \frac{\cc{\psi_Q}\phi_Q}{|Q|!|Q|!}
            \Pochhammer{\frac{1}{2\hbar}}_{|Q|}.
        \end{equation}
    \item \label{item:DhbarFinerTopology} The nuclear Fréchet topology
        of $\mathfrak{D}_\hbar$ is finer than the pre Hilbert space
        topology induced by $\SP{\argument, \argument}_\hbar$.
    \end{lemmalist}
\end{lemma}
\begin{proof}
    The first part is clear since $\stardisk$, the complex
    conjugation, and the evaluation functional $\delta_0$ are
    continuous. Having an absolute basis it is clear that we can split
    the total space into a direct sum of two closed subspaces by
    splitting the set of basis vectors and taking their closures of
    their spans, respectively. Thus the second statement follows from
    the characterization of the Gel'fand ideal according to
    Lemma~\ref{lemma:DeltaIsPositive},
    \refitem{item:GelfandIdealOfDelta}. The third part is clear from
    the computation in the proof of the previous lemma. For the last
    part, we first recall that $\mathfrak{D}_\hbar$ is necessarily
    nuclear, being a closed subspace of a nuclear space. In fact, it
    is again isomorphic to a Köthe space of sub-factorial
    growth. Since $\stardisk$, the complex conjugation, and the
    evaluation functional $\delta_0$ is continuous, also the inner
    product is continuous, proving that the pre Hilbert space topology
    is coarser.
\end{proof}

Clearly, the Hilbert space completion $\mathfrak{H}_\hbar$ of
$\mathfrak{D}_\hbar$ can be described as the $\ell^2$-space of
sequences indexed by multiindices $Q \in \mathbb{N}_0^n$ with a
weighted counting measure with weights given by $\frac{1}{|Q|!|Q|!}
\Pochhammer{\frac{1}{2\hbar}}_{|Q|}$. This is clear from
\eqref{eq:GNSInnerProduct}. More surprising is the fact that we can
view elements in the Hilbert space $\mathfrak{H}_\hbar$ still as
\emph{functions} on $\mathbb{D}_n$. We collect this with some other
first properties of $\mathfrak{H}_\hbar$ in the following lemma:
\begin{lemma}
    \label{lemma:ReproducingKernel}%
    Let $\hbar > 0$.
    \begin{lemmalist}
    \item \label{item:HilbertBasis} The vectors $\basis{f}_{0, Q} \in
        \mathfrak{D}_\hbar$ are pairwise orthogonal and yield a
        Hilbert basis
        \begin{equation}
            \label{eq:HilbertBasis}
            \basis{E}_Q
            =
            \frac{|Q|!}{\sqrt{\Pochhammer{\frac{1}{2\hbar}}_{|Q|}}}
            \basis{f}_{0, Q}
        \end{equation}
        after normalization, where $Q \in \mathbb{N}_0^n$.
    \item \label{item:EvaluationContinuous} The evaluation functional
        at $w \in \mathbb{D}_n$ extends to a continuous linear
        functional on $\mathfrak{H}_\hbar$. More precisely, for every
        compact $K \subseteq \mathbb{D}_n$ there is a constant $c_K >
        0$ with
        \begin{equation}
            \label{eq:deltavContinuousInLzwei}
            |\psi(w)|
            \le
            c_K
            \norm{\psi}_{\mathfrak{H}_\hbar}
        \end{equation}
        for all $\psi \in \mathfrak{H}_\hbar$ and $w \in K$.
    \item \label{item:ReproducingKernel} For $w \in \mathbb{D}_n$ the
        evaluation functional $\delta_w$ can be written as $\psi(w) =
        \SP{\basis{E}_w, \psi}_{\mathfrak{H}_\hbar}$ with
        \begin{equation}
            \label{eq:Ew}
            \basis{E}_w
            =
            \sum_{Q}
            \frac{w^Q}{(1 -|w|^2)^{|Q|}} \frac{|Q|!}{Q!} 
            \basis{f}_{0, Q}
            \in \mathfrak{D}_\hbar.
        \end{equation}
        The series converges absolutely in the nuclear Fréchet
        topology of $\mathfrak{D}_\hbar$ and in the Hilbert space
        topology of $\mathfrak{H}_\hbar$.
    \end{lemmalist}
\end{lemma}
\begin{proof}
    The first statement is clear from the explicit form of the inner
    product as in \eqref{eq:GNSInnerProduct}. For the second part we
    use the Cauchy-Schwarz inequality to obtain
    \begin{align*}
        |\psi(w)|^2
        &=
        \left|
            \sum\nolimits_Q \psi_Q \basis{f}_{0, Q}(w)
        \right|^2 \\
        &=
        \left|
            \sum\nolimits_Q \psi_Q
            \frac{1}{|Q|!Q!} \Pochhammer{\frac{1}{2\hbar}}_{|Q|}
            \frac{w^Q}{(1 - |w|^2)^{|Q|}}
        \right|^2 \\
        &\le
        \left(
            \sum\nolimits_Q
            \frac{|\psi_Q|^2}{|Q|!|Q|!}
            \Pochhammer{\frac{1}{2\hbar}}_{|Q|}
        \right)
        \left(
            \sum\nolimits_Q
            \frac{1}{Q!Q!}
            \Pochhammer{\frac{1}{2\hbar}}_{|Q|}
            \frac{|w^Q|^2}{(1 - |w|^2)^{2|Q|}}
        \right) \\
        &=
        \norm{\psi}_{\mathfrak{H}_\hbar}^2
        f(w),
    \end{align*}
    with a function $f$ determined by the second series. Now clearly
    the series $f$ converges for all $w \in \mathbb{D}_n$ by the
    standard estimate for the Pochhammer symbols
    \eqref{eq:EstimatePochhammerSymbols}. It even converges locally
    uniformly and thus defines a continuous function $f$. Taking
    $c_K^2 = \max_{w \in K} f(w)$ proves the second part. From this
    continuity statement it is clear that there exists a uniquely
    determined vector $\basis{E}_w \in \mathfrak{H}_\hbar$ with
    $\psi(w) = \SP{\basis{E}_w, \psi}_{\mathfrak{H}_\hbar}$. Using the
    explicit formula \eqref{eq:GNSInnerProduct} for the inner product
    as well as the explicit formula for $\psi(w)$ in terms of the
    series expansion using the functions $\basis{f}_{0, Q}$ gives
    immediately \eqref{eq:Ew}. Since the coefficients with respect to
    the $\basis{f}_{0, Q}$ have exponential and thus sub-factorial
    growth, we get $\basis{E}_w \in \mathfrak{D}_\hbar$ as claimed.
\end{proof}

From the general GNS construction we know that
$\complete{\mathcal{A}}_\hbar(\mathbb{D}_n)$ acts on the GNS pre
Hilbert space $\complete{\mathcal{A}}_\hbar(\mathbb{D}_n) \big/
\mathcal{J}_0$ in the canonical way. The projection onto
$\mathfrak{D}_\hbar$ and the injection of $\mathfrak{D}_\hbar$ into
$\complete{\mathcal{A}}_\hbar(\mathbb{D}_n)$ provide isomorphisms of
the GNS pre Hilbert space with $\mathfrak{D}_\hbar$. Using those, we
can compute the GNS representation explicitly:
\begin{proposition}
    \label{proposition:GNSRep}%
    Let $\hbar > 0$.  Let $a = \sum_{P, Q} a_{P, Q} \basis{f}_{P, Q}
    \in \complete{\mathcal{A}}_\hbar(\mathbb{D}_n)$ and $\psi = \sum_Q
    \psi_Q \basis{f}_{0, Q} \in \mathfrak{D}_\hbar$ be given. Then the
    GNS representation is determined by
    \begin{align}
        &\pi_0(a)\psi
        =
        \pr_{\mathfrak{D}_\hbar}
        \left(a \stardisk \iota(\psi)\right) \\
        &=
        \sum_{Q, S} \sum_{P \le S} a_{P, Q} \psi_S
        \frac{1}{P!(\alpha - |Q|)!}
        \binom{Q + S - P}{Q}
        \binom{\alpha + |S| - |P|}{|S|}
        \frac{
          \frac{1}{(\alpha + |S| - |P|)!}
          \Pochhammer{\frac{1}{2\hbar}}_{\alpha + |S| - |P|}
        }
        {
          \frac{1}{(|Q| + |S| - |P|)!}
          \Pochhammer{\frac{1}{2\hbar}}_{|Q| + |S| - |P|}
        }
        \basis{f}_{0, Q + S - P},
        \label{eq:GNSrepExplicitly}
    \end{align}
    where $\alpha = \max(|P|, |Q|)$ as usual. The series converges
    in the Fréchet topology of $\mathfrak{D}_\hbar$.
\end{proposition}
\begin{proof}
    First we note that the projection $\pr_{\mathfrak{D}_\hbar} \colon
    \complete{\mathcal{A}}_\hbar(\mathbb{D}_n) \longrightarrow
    \mathfrak{D}_\hbar$ is continuous, this is clear from the explicit
    description of the topology by means of the seminorms
    $\norm{\argument}_\epsilon$. Similarly, the injection $\iota$ is
    continuous. Together with the continuity of the product
    $\stardisk$ and the (absolute) convergence of the series $a =
    \sum_{P, Q} a_{P, Q} \basis{f}_{P, Q}$ and $\psi = \sum_S \psi_S
    \basis{f}_{0, S}$ it is clear that we can take all the summations
    in from of the algebraic manipulations and still have convergence
    in the Fréchet topology of $\mathfrak{D}_\hbar$. Thus the
    following computations are justified. We compute things upstairs
    on $C_{n+1}^+$ with $\alpha = \max(|P|, |Q|)$ and $\beta = \max(0,
    |S|) = |S|$ as usual and get from
    Proposition~\ref{proposition:StructureConstantsAbove}
    \begin{align*}
        a \stardisk \iota(\psi)
        &=
        \sum_{P, Q} \sum_S a_{P, Q} \psi_S
        \sum_{k=0}^{\min(\alpha - |P|, \beta - |S|)}
        \sum_{K = 0}^{\min(P, S)}
        \frac{(-1)^k}{k!K!}
        \left[
            \basis{f}_{P + 0 - K, Q + S - K, \alpha + \beta - k - |K|}
        \right] \\
        &\quad\quad
        \binom{P + 0 - K}{0}
        \binom{Q + S - K}{Q}
        \binom{\alpha + \beta - k - |P| - 0}{\beta - 0}
        \binom{\alpha + \beta - k - |Q| - |S|}{\alpha - |Q|} \\
        &=
        \sum_{P, Q} \sum_S a_{P, Q} \psi_S
        \sum_{K = 0}^{\min(P, S)}
        \frac{1}{K!}
        \left[
            \basis{f}_{P - K, Q + S - K, \alpha + |S| - |K|}
        \right]
        \binom{Q + S - K}{Q}
        \binom{\alpha - |P| + |S|}{|S|}.
    \end{align*}
    Next we have to project back to $\mathbb{D}_\hbar$. Hence we have
    to compute $\pr_{\mathfrak{D}_\hbar} [\basis{f}_{P - K, Q + S - K,
      \alpha + |S| - |K|}]$. In general, $[\basis{f}_{I, J, \gamma}]$
    is a nontrivial linear combination of the basis vectors
    $\basis{f}_{R, S}$ which we have computed in the proof of
    Theorem~\ref{theorem:DiscAlgebra}. However, since we are
    interested in the $\mathfrak{D}_\hbar$-component only, the
    computation simplifies drastically. In fact, from ($*$) in the
    proof of Lemma~\ref{lemma:BasisOnTheDisc} we see that the only
    nontrivial contribution in $\pr_{\mathfrak{D}_\hbar} [f_{I, J,
      \gamma}]$ arises for $I = 0$ and is given by
    \[
    \pr_{\mathfrak{D}_\hbar}\left[\basis{f}_{I, J, \gamma}\right]
    =
    \frac{1}{\gamma!J!(\gamma - |J|)!}
    \Pochhammer{\frac{1}{2\hbar}}_\gamma
    \frac{\cc{v}^J}{(1 - |v|^2)^{|J|}}
    \delta_{I, 0}
    =
    \frac{\Pochhammer{\frac{1}{2\hbar}}_\gamma}
    {\Pochhammer{\frac{1}{2\hbar}}_{|J|}}
    \frac{|J|!}{\gamma!(\gamma - |J|)!}
    \basis{f}_{0, J} \delta_{I, 0}.
    \]
    Inserting this in the above computation gives us first the
    condition $P - K = 0$ and hence $P \le S$. Then it results in
    \eqref{eq:GNSrepExplicitly}.
\end{proof}
\begin{remark}
    \label{remark:RelationToBerezinETC}%
    The above characterization of the GNS representation is of course
    not yet very illuminating but shows that one can efficiently
    compute the representation. It will be a future project to relate
    it to the more familiar construction of a Berezin-Toeplitz like
    quantization as this has been used by various people. In
    particular, a direct comparison with the results on
    \cite[Sect.~4]{cahen.gutt.rawnsley:1994a} should be within
    reach. Moreover, the vector states $\basis{E}_w$ should directly
    correspond to the coherent states used in
    \cite{cahen.gutt.rawnsley:1994a}. Finally, we believe that the
    above GNS representation will provide the bridge in order to
    compare the algebra $\complete{\mathcal{A}}_\hbar(\mathbb{D}_n)$
    to other (deformation) quantizations of the Poincaré disk which
    are based on more operator-algebraic approaches. Here in
    particular the approaches of
    \cite{bordemann.meinrenken.schlichenmaier:1991a,
      borthwick.lesniewski.upmeier:1993a,
      bieliavsky.detournay.spindel:2009a, bieliavsky:2002a} should be
    mentioned. As this certainly will require some more effort we
    postpone a detailed study to some future projects.
\end{remark}

%
%

\appendix

%
%

\section{Köthe spaces of sub-factorial growth}
\label{sec:SubfactorialGrowth}

In this appendix we collect some well-known and basic features of the
Köthe space of those sequences which have sub-factorial growth. For
details on Köthe spaces we refer e.g.\ to \cite{jarchow:1981a}, in
particular to the Sections~1.7.E, 3.6.D, and 21.6. First, let
\begin{equation}
    \label{eq:SubfactorialSpace}
    \Lambda
    =
    \left\{
        a = (a_n)_{n \in \mathbb{N}}
        \; \Bigg| \;
        p_\epsilon(a)
        =
        \norm{
          \left(
              \frac{a_n}{(n!)^\epsilon}
          \right)_{n \in \mathbb{N}}
        }_{\ell^1}
        <
        \infty
        \; \textrm{for all} \;
        0 < \epsilon < 1
    \right\}
\end{equation}
be the Köthe sequence space of sequences with sub-factorial growth,
endowed with its topology induced by the seminorms $p_\epsilon$ as
usual. Equivalently, we can use the system of seminorms
\begin{equation}
    \label{eq:EpsionSubFac}
    \norm{a}_\epsilon
    =
    \sup_{n \in \mathbb{N}} \frac{|a_n|}{(n!)^\epsilon},
\end{equation}
where again $0 < \epsilon < 1$. Indeed, the estimate
$\norm{a}_\epsilon \le p_\epsilon(a)$ is obvious. For the reverse
we note that
\begin{equation}
    \label{eq:SubFacReverseEstimateEllEins}
    p_\epsilon(a)
    =
    \sum_{n \in \mathbb{N}} \frac{|a_n|}{(n!)^\epsilon}
    \le
    \sup_{n \in \mathbb{N}} \frac{|a_n|}{(n!)^{\epsilon/2}}
    \sum_{n \in \mathbb{N}} \frac{1}{(n!)^{\epsilon/2}}.
\end{equation}
Since the last series converges, this gives the estimate
$p_\epsilon(a) \le c \norm{a}_{\epsilon/2}$ with $c$ being the value
of the above series.  Since clearly a countable subset of the
parameters $\epsilon$ will suffice to specify the topology, $\Lambda$
is a Fréchet space.

As for all Köthe spaces, the sequences $\basis{e}_k = (\delta_{nk})_{n
  \in \mathbb{N}}$ constitute an \emph{absolute} Schauder basis, see
e.g.~\cite[Theorem~24.8.8]{jarchow:1981a}.

Consider now $\epsilon > 0$ and the corresponding sequence
$\left(\frac{1}{(n!)^\epsilon}\right)_{n \in \mathbb{N}}$. Then the
sequence of quotients
$\left(\frac{(n!)^{\epsilon/2}}{(n!)^\epsilon}\right)_{n \in
  \mathbb{N}} = \left(\frac{1}{(n!)^{\epsilon/2}}\right)_{n \in
  \mathbb{N}}$ is $p$-summable for all $p > 0$. Thus, by the
Grothendieck-Pietch criterion, this implies that $\Lambda$ is
\emph{nuclear} and, in fact, even \emph{strongly nuclear} by
\cite[Prop.~21.8.2]{jarchow:1981a}

Sometimes we will meet ``sequences'' not indexed by $n \in \mathbb{N}$
but by multiindices. Thus consider
\begin{equation}
    \label{eq:GeneralKoethe}
    \Lambda_d
    =
    \left\{
        a = (a_N)_{N \in \mathbb{N}^d}
        \; \Bigg| \;
        p_\epsilon(a)
        =
        \norm{
          \left(
              \frac{a_N}{(|N|!)^\epsilon}
          \right)_{N \in \mathbb{N}^d}
        }_{\ell^1}
        <
        \infty
        \; \textrm{for all} \;
        0 < \epsilon < 1
    \right\}.
\end{equation}
Since the number of points $N$ with given $|N|$ grows polynomially we
see that the same arguments as for $\Lambda = \Lambda_1$ apply. We
have an equivalent system of seminorms
\begin{equation}
    \label{eq:EquivalentKoetheSeminorms}
    \norm{a}_\epsilon
    =
    \sup_{N \in \mathbb{N}^d} \frac{|a_n|}{(|N|!)^\epsilon}.
\end{equation}
By choosing an enumeration, $\Lambda_d$ can be viewed as a Köthe space
as well. However, it will be easier to use multiindices in many
situations.  The vectors $\basis{e}_{K} = (\delta_{KN})_{N \in
  \mathbb{N}^d}$ yield an absolute Schauder basis. Finally, the
summability properties of the functions $\frac{1}{(|N|!)^\epsilon}$
over $\mathbb{N}^d$ are the same and thus $\Lambda_d$ is strongly
nuclear, too.

There are several further generalizations. In particular, the index
triples $(I, J, \gamma)$ can be used to construct a Köthe space by
imposing sub-factorial growth with respect to $\gamma$. Since there
are only finitely many $|I|, |J| \le \gamma$ in an index triple and
their number grows polynomially in $\gamma$, also the index triples
give rise to a Köthe space as above.

We shall speak of \emph{Köthe spaces of sub-factorial growth} whenever
the reference to the index set and the relevant factorial is clear
from the context.

%
%

\addcontentsline{toc}{section}{References}

\begin{footnotesize}

\end{footnotesize}

%
%

\end{document}